\documentclass[a4paper]{article}
\usepackage[T1]{fontenc}
\usepackage{amsfonts,amsmath,amsthm,amssymb}
\usepackage{calc}
\usepackage{gensymb}
\usepackage[nottoc, notlof, notlot]{tocbibind}
\usepackage{fullpage}
\usepackage{mathrsfs}  
\usepackage{bigints}
\usepackage{indentfirst}
\usepackage{enumerate}
\usepackage{authblk}
\usepackage{hyperref}
\usepackage{mathtools}

\DeclareMathOperator{\tr}{Tr}
\DeclareMathOperator{\ts}{tr}
\DeclareMathOperator{\ev}{ev}

\DeclareMathOperator{\id}{id}

\let\limsup\relax
\DeclareMathOperator*{\limsup}{limsup}
\let\liminf\relax
\DeclareMathOperator*{\liminf}{liminf}
\newcommand{\norm}[1]{\left\Vert #1\right\Vert}

\begin{document}

\newtheorem{theorem}{Theorem} [section]
\newtheorem{prop}[theorem]{Proposition} 
\newtheorem{defi}[theorem]{Definition} 
\newtheorem{exe}[theorem]{Example} 
\newtheorem{lemma}[theorem]{Lemma} 
\newtheorem{rem}[theorem]{Remark} 
\newtheorem{cor}[theorem]{Corollary} 
\newtheorem{conj}[theorem]{Conjecture}
\renewcommand\P{\mathbb{P}}
\newcommand\E{\mathbb{E}}
\newcommand\N{\mathbb{N}}
\newcommand\Z{\mathbb{Z}}
\newcommand\1{\mathbf{1}}
\newcommand\C{\mathbb{C}}
\newcommand\CC{\mathcal{C}}
\newcommand\D{\mathcal{D}}
\newcommand\M{\mathbb{M}}
\newcommand\R{\mathbb{R}}
\newcommand\U{\mathcal{U}}
\newcommand\V{\mathcal{V}}
\newcommand\A{\mathcal{A}}
\newcommand\Hi{\mathcal{H}}
\newcommand\B{\mathcal{B}}
\renewcommand\L{\mathcal{L}}
\newcommand\F{\mathcal{F}}
\newcommand\G{\mathcal{G}}
\newcommand\K{\mathcal{K}}
\renewcommand\i{\mathbf{i}}
\renewcommand\S{\mathcal{S}}
\renewcommand\d{\partial_i}
\newcommand\x{\mathbf{x}}
\newcommand\y{\mathbf{y}}
\newcommand\PP{\mathcal{A}}

\def\etc{,\dots ,}

\allowdisplaybreaks

\begin{minipage}{0.85\textwidth}
	\vspace{2.5cm}
\end{minipage}
\begin{center}
	\large\bf The spectrum of a tensor of random and deterministic matrices
	
\end{center}

\renewcommand{\thefootnote}{\fnsymbol{footnote}}	
\vspace{0.8cm}

\begin{center}
				F\'elix Parraud\\
				\footnotesize 
				{Institut Mittag-Leffler}\\
				{\it felix.parraud@gmail.com}
\end{center}

\bigskip
\bigskip

\begin{abstract}
	We consider operator-valued polynomials in Gaussian Unitary Ensemble random matrices and we show that its $L^p$-norm can be upper bounded, up to an asymptotically small error, by the operator norm of the same polynomial evaluated in free semicircular variables as long as $p=o(N^{2/3})$. As a consequence, if the coefficients are $M$-dimensional matrices with $M=\exp(o(N^{2/3}))$, then the operator norm of this polynomial converges towards the one of its free counterpart. In particular this provides another proof of the Peterson-Thom conjecture thanks to the result of Ben Hayes. We also obtain similar results for polynomials in random and deterministic matrices.
	
	The approach that we take in this paper is based on an asymptotic expansion obtained by the same author in a previous paper combined with a new result of independent interest on the norm of the composition of the multiplication operator and a permutation operator acting on a tensor of $\CC^*$-algebras. \\


\end{abstract}

\section{Introduction}

The model of random matrix that we consider is the following. Given $X^N$ a $d$-tuple of independent Gaussian or Haar Unitary matrices, and $\A$ a faithful tracial $\CC^*$-algebra, we study polynomials in $d$ variables evaluated in $X^N$ with coefficients in $\A$. More precisely we study the following,
\begin{equation}
	\label{iresveos}
	\sum_{M \text{ monomials}} a_M\otimes M(X^N),
\end{equation}
where there is a finite number of $M$ such that $M\neq 0$. This type of random matrix model has seen a lot of interest recently due to a result from Hayes \cite{petom} which proved that the now solved Peterson-Thom conjecture--which asserts that any diffuse, amenable subalgebra of a free group factor is contained in a unique maximal amenable subalgebra--is equivalent to proving the convergence of the norm of the random matrix above in the case where $\A$ is generated by another $d$-tuple of independent GUE random matrices, independent of $X^N$ but of the same dimension. Belinschi-Capitaine, Bordenave-Collins, as well as Magee-de la Salle all announced a different proof of this result in respectively \cite{bitid}, \cite{ptbordeconv} and \cite{undemi}. The strategy used in this paper to study this random matrix model is independent from those works and can be summarized as follows. 
\begin{itemize}
	\item First we make use of the expansion that we obtained in \cite{trois} to rewrite the moments of our random matrices as a power series in the inverse of the dimension. The formula for the coefficients of this power series are fully explicit and most of the work for this step was already done in \cite{trois}. However, one of the major drawback of those formulas is that the following operator is involved, given $\A,\B$ two $\CC^*$-algebras, we define it on simple tensors of $\A\otimes\B$ by 
	$$ m: ((a_1\otimes b_1),\dots,(a_4\otimes b_4))\mapsto (a_2a_1a_4a_3)\otimes(b_1b_2b_3b_4).$$
	And it is a priori quite difficult to show that for any $t,x,y,z\in\A\otimes\B$ endowed with the minimal tensor norm, the operator norm of $m(t,x,y,z)$ is bounded by the one of $t,x,y$ and $z$ up to a ``reasonable'' constant. However it is absolutely necessary to be able to do so in order to use the moment method.
	\item Thus the main technical tool that we develop in this paper is a way to handle this operator. In particular this yields a result of independent interest which we summarize in Theorem \ref{masterineq}. In a way that we detail later on, this result could be compared to Haagerup inequality, see Lemma 1.4 of \cite{haagineq}. The key element to tackle this problem is a generalization of the following result, given a product of semicircular variables, its trace can be expressed with the help of the set of non crossing pair partitions and the covariance of the different semicircular variables. It turns out that one can prove a similar formula if one replaces individual semicircular variables by monomials in them, see Proposition \ref{edgtn}.
\end{itemize}

\subsection{The norm of operator-valued polynomials}

Most papers studying the random matrix model described in Equation \eqref{iresveos}--with the notable exception of \cite{ptbordeconv}--usually focus on the case where $\A$ is the set of matrices of dimension $M$ where $M$ is a function of $N$. Indeed, in this case one can usually deduce the convergence of the spectrum thanks to the moment method. Thus there researchers have been focusing on maximizing the dimension $M$ for whom one can make this method. The first result of this kind was found by Pisier who, inspired by \cite{HT}, proved that if $M$ was of order $o(N^{1/4})$, then the norm did converge under some rather basic assumptions. This was improved to $o(N^{1/3})$ in \cite{un,deux}, and then to $o(N(\log N)^{-3})$ for GUE matrices in \cite{tensorsc}, and to $o(N(\log N)^{-5/2})$ for Haar unitary matrices in \cite{sept}. Finally in \cite{ptbordeconv}, Bordenave and Collins proved that one could consider $M$ much larger than $N$, with a cut-off at $M=\exp(N^{\alpha})$ with $\alpha = (32d+160)^{-1}$. Very recently, in \cite{vargas}, there has been a new approach to prove so-called strong convergence results, and although in this paper, they only consider matrix coefficients of size $o(N)$, with a refinement of their method, and with the help of the asymptotic expansion proved in \cite{sept}, the authors proved as a corollary of their main result that one could take matrix coefficients of size $\exp(N^{1/2} (\log N)^{-4})$. In this paper, one of our main result is that one can use the moment method with matrix coefficients of the size $\exp(o(N^{2/3}))$. However this result is a corollary of the following theorem that studies polynomials as in Equation \eqref{iresveos} for $\A$ a faithful tracial $\CC^*$-algebra. Note that this assumptions on $\A$ is only used to prove Proposition \ref{docmmfsokcm} and could probably be relaxed as long as the map defined in this proposition is still bounded.

\begin{theorem}
	
	\label{popopohdvc}
	
	Given the following objects,
	\begin{itemize}
		\item $\A$ a faithful tracial $\CC^*$-algebra,
		\item $P\in\A\langle X_1,\dots,X_d\rangle$ a polynomial of degree $n$,
		\item $X^N=\left(X_1^N,\dots,X_d^N\right)$ a $d$-tuple of GUE random matrices,
		\item $x=\left(x_1,\dots,x_d\right)$ a free semicircular system,
	\end{itemize}
	Then there exists a universal constant $C_{n,d}$ such that,
	\begin{equation}
		\label{mlskmfv-1}
		\norm{\E\left[\id_{\A}\otimes\ts_N\left(P^k\left(X^N\right)\right)\right]} \leq (3nk+1) \norm{P(x)}^k \exp\left( k\times C_{n,d} \max\left( \left(\frac{k^3}{N^2}\right), \left(\frac{k^3}{N^2}\right)^{1/13}\right) \right),
	\end{equation}
	in particular,
	\begin{equation}
		\label{mlskmfv}
		\E\left[\norm{P\left(X^N\right)}_{L^{2k}}\right] \leq (3nk+1)^{\frac{1}{2k}} \norm{P(x)} \exp\left( C_{2n,d} \max\left( \left(\frac{k^3}{N^2}\right), \left(\frac{k^3}{N^2}\right)^{1/13}\right) \right).
	\end{equation}
	
\end{theorem}

It is important to note that with this result we are beginning to reach the limit of the moment method. Indeed, by using the Harer-Zagier recursions, see Lemma 3.3.3 of \cite{alice}, one can show that 
given $X_N$ a GUE random matrix and $x$ a semicircular variable, one can find a universal constant $c$ such that,
\begin{equation}
	\label{ljfdxrv}
	\E\left[\norm{X_N}_{L^{2k}}\right]\leq \norm{x} e^{c\frac{k^2}{N^2}}.
\end{equation}
This ratio of $k/N$ is optimal since with $k=N^{2/3}$, it implies that
$$ \E\left[\norm{X_N}\right]\leq N^{1/(2k)} \E\left[\norm{X_N}_{L^{2k}}\right] \leq 2 \exp\left(\frac{\log N}{2N^{2/3}}+cN^{-2/3}\right) = 2\left(1+\mathcal{O}\left(N^{-2/3} \log N\right)\right).$$
Thus if one could prove Equation \eqref{ljfdxrv} with $k^{\alpha}$, $\alpha<2$, instead of $k^2$ then one could prove a sharper upper bound in the equation above, however this would contradict the result of Tracy and Widom that states that the largest eigenvalue of a GUE random matrix is of order $N^{-2/3}$, see Theorem 3.1.4 of \cite{alice}. Thus one would expect that at best one could prove an equivalent to Equation \eqref{mlskmfv} with a ratio of $k/N$. However, even for polynomials with scalar coefficients it is far from easy to show that the fluctuations of the largest eigenvalue are of order $N^{-2/3}$. It is known for polynomials of degree at most two thanks to a recent paper \cite{nemish}, but in all generality one can only prove that those fluctuations are at most of order $N^{-1/2}$, see \cite{trois}.

Thanks to Theorem \ref{popopohdvc}, one can deduce the following one.

\begin{theorem}
	\label{1strongconv}
	Let  the  following objects be given:
	\begin{itemize}
		\item $X^N = (X_1^N,\dots,X_d^N)$ a $d$-tuple of independent $GUE$ of size $N$,
		\item $x = (x_1,\dots,x_d)$ a system of free semicircular variable,
		\item $U^N = (U_1^N,\dots,U_{d'}^N)$ a $d$-tuple of independent Haar unitary matrices of size $N$,
		\item $u = (u_1,\dots,u_{d'})$ free Haar unitaries,
		\item $Y^M = (Y_1^M,\dots,Y_p^M)$ deterministic matrices of size $M$, which as $M$ goes to infinity, converge strongly in distribution towards a $p$-tuple $y$ of non-commutative random variables in a $\mathcal C^*$- probability space $\B$ with a faithful trace $\tau_{\B}$.
	\end{itemize}
	
	\noindent Then, the following holds true:
	
	\begin{itemize}
		\item If $(M_N)_{N\geq 1}$ is a sequence of integers such that $M_N = \exp\left(o(N^{2/3})\right)$, then almost surely $(X^N\otimes I_{M_N},U^N\otimes I_{M_N}, I_N\otimes Y^{M_N})$ converges strongly in distribution to the family $\mathcal{F} = (x\otimes 1, u\otimes 1, 1\otimes y)$ in the minimal $C^*$-tensor product $\A\otimes_{\min}\B$ (see definition \ref{1mini}).
	\end{itemize}
	
\end{theorem}

Combined with the result \cite{petom} of Hayes, this theorem provides yet another proof of an important result in the theory of von Neumann algebra, the Peterson-Thom conjecture. See also \cite{conseqpt} for other applications of this result. Finally, note that among every result of strong convergence for polynomials with matricial coefficients \cite{slmcls,un,deux,tensorsc,sept,ptbordeconv,vargas,undemi}, the theorem above is the most optimal in the sense that one can take the coefficients with the largest dimension, up to $\exp\left(o(N^{2/3})\right)$. However it is unknown whether the constant $2/3$ is optimal in the theorem above. In particular, thanks to a result of Pisier \cite{slmcls}, this constant is at most $2$ but no further information is known on this problem. Note however that if one could prove Theorem \ref{popopohdvc} with a ratio of $(k/N)^2$ instead of $k^3/N^2$ then Theorem \ref{1strongconv} would hold with $M_N = \exp\left(o(N)\right)$ instead of $M_N = \exp\left(o(N^{2/3})\right)$. Besides, Pisier proved in Corollary 4.7 of \cite{slmcls} that for any polynomial $A$ of degree one, as long as the coefficients of $A$ are matrices of dimension $\exp\left( o(N) \right)$, then 
$$ \norm{A(X^N)} \leq (2+o(1))\norm{A(x)}. $$
Although this factor $2$ makes it impossible to deduce strong convergence results out of this inequality, given that this constant is still independent of the dimension $M$, coupled with the discussion following Theorem \ref{popopohdvc}, it would seem reasonable for the optimal constant to be $1$.

Finally, we state a version of Theorem \ref{popopohdvc} that also cover the case of polynomials in random and deterministic matrices, which is a model that has never been studied as far as we are aware. More precisely we study the following model, if $\mathfrak{n}N$ is the size of our random matrices, then our deterministic matrices will be of size $N$, and in order to consider them in the same space, we tensorize them with the identity matrix of size $\mathfrak{n}$. Doing so comes at a cost, indeed in the right hand side of \eqref{mlskmfv-1} and \eqref{mlskmfv}, one has to replace $N^2$ by $\mathfrak{n}^2$, hence in order for us to have a good control of the right hand side when $N$ goes to infinity, we need to have $\mathfrak{n}$ to be at least a power of the logarithm of $N$. In particular, one cannot consider the optimal case $\mathfrak{n}=1$.

\begin{theorem}
	
	\label{detcase}
	
	Given the following objects,
	\begin{itemize}
		\item $\A$ a faithful tracial $\CC^*$-algebra,
		\item $P\in\A\langle X_1,\dots,X_d,Z_1,\dots,Z_{q}\rangle$ a polynomial,
		\item $Z^N=(Z_1^N,\dots,Z_q^N)$ deterministic matrices of size $N$,
		\item $X^{\mathfrak{n}N}=\left(X_1^{\mathfrak{n}N},\dots,X_d^{\mathfrak{n}N}\right)$ a $d$-tuple of GUE random matrices of size $\mathfrak{n}N$,
		\item $x=\left(x_1,\dots,x_d\right)$ a free semicircular system,
	\end{itemize}
	Then there exists a constant $C$ which only depends on $P$ and $\sup_j \norm{Z_j^N}$ such that,
	\begin{align}
		&\norm{\E\left[\id_{\A}\otimes\ts_N\left(P^k\left(X^{\mathfrak{n}N}, I_{\mathfrak{n}}\otimes Z^N\right)\right) \right]}\\
		&\quad\quad\quad\quad\quad\quad\quad\quad\quad\quad\quad \leq (3nk+1) \norm{P(x,Z^N)}^k \exp\left( k\times \frac{C}{\norm{P(x,Z^N)}^4} \max\left( \left(\frac{k^3}{\mathfrak{n}^2}\right), \left(\frac{k^3}{\mathfrak{n}^2}\right)^{1/13}\right) \right), \nonumber
	\end{align}
	in particular,
	\begin{equation}
		\label{mlskmfv2}
		\E\left[\norm{P\left(X^N, I_{\mathfrak{n}}\otimes Z^N\right)}_{L^{2k}}\right] \leq (3nk+1)^{\frac{1}{2k}} \norm{P(x,Z^N)} \exp\left( \frac{C}{\norm{P(x,Z^N)}^4} \max\left( \left(\frac{k^3}{\mathfrak{n}^2}\right), \left(\frac{k^3}{\mathfrak{n}^2}\right)^{1/13}\right) \right).
	\end{equation}
	We refer to Subsection \ref{sdkmvdoijiytd} for a definition of the evaluation of a polynomial in $(x,Z^N)$.
	
\end{theorem}

Note that one cannot deduce a theorem similar to \ref{1strongconv} out of the one above. Indeed, if $Y^{M_N}$ and $Z^N$ are families of deterministic matrices which converge strongly in distribution towards families of operators $y$ and $z$, then there is no reason for the family $(Y^{M_N}\otimes I_N, I_{M_N}\otimes Z^N)$ to converge strongly in distribution towards $(y\otimes 1, 1\otimes z)$. However, if one have that $M_N= \exp(o(\mathfrak{n}^{2/3}))$ and $\mathfrak{n} \gg \ln(N)^{3/2}$, then thanks to the theorem above we do get that
$$ \limsup\limits_{N\to\infty} \E\left[\norm{P(X^N\otimes I_{M_N}, Z^N\otimes I_{M_N}, I_N\otimes Y^{M_N})}\right] \leq \limsup\limits_{N\to\infty} \norm{P(x\otimes I_{M_N}, Z^N\otimes I_{M_N}, I_N\otimes Y^{M_N})}.$$

\subsection{An asymptotic expansion for sufficiently smooth functions}

Although so far we have focused on polynomials, because in this paper we compute asymptotic expansions for moments of our random matrices with fully explicit formula, it is also possible to study smooth functions. Indeed we start by simply approximating the exponential function with its Taylor expansion which gives us an expansion formula for the quantity
$$ \E\left[\id_{\A}\otimes\ts_N\left(e^{\i y P(X^N)}\right)\right],$$
for any $y\in\R$. Next we use Fourier inversion to study sufficiently smooth functions. This yields he following theorem.

\begin{theorem}
	\label{3lessopti}
	We define,
	\begin{itemize}
		\item $\A$ a faithful tracial $\CC^*$-algebra,
		\item $P\in\A\langle X_1,\dots,X_d\rangle$ a self-adjoint polynomial,
		\item $X^N = (X_1^N,\dots,X_d^N)$ independent $GUE$ matrices of size $N$,
		\item $f:\R\mapsto\R$ a bounded function of class $\mathcal{C}^{16(k+1)}$.
	\end{itemize}
	Then there exist unique deterministic coefficients $(\alpha_p^f(P)_{1\leq p\leq k}\in \A^k$ such that
	\begin{align}
		\label{3mainresu0}
		\E\left[\id_{\A}\otimes\ts_N\left(f(P(X^N)) \right)\right] = \sum_{p= 0}^k \frac{\alpha_p^f(P)}{N^{2p}} +	 \mathcal{O}\left(\frac{1}{N^{2(k+1)}}\right).
	\end{align}
\end{theorem}

This theorem should be compared with Theorem 1.1 of \cite{trois,sept} which prove similar result with $\A=\C$. The major difference being that in \cite{trois,sept} one only needs to differentiate $4k+7$ times whereas here it is necessary to be able to differentiate $16(k+1)$ times. Hence in the theorem above we trade some smoothness assumption for more generality on the coefficients of $P$. Besides, it is important to note that the method of proof are rather different. In \cite{trois,sept}, we show that there exist an operator $\Theta$ such that
$$ \E\left[\ts_N\left(f(X^N)\right)\right] = \tau(f(x)) +\frac{1}{N^2}\E\left[\ts_N\left(\Theta(f)(X^N)\right)\right], $$
and we simply iterate this result $k$ times to get an asymptotic expansion. However, in this paper it is not possible to proceed like this. As mentioned above, it is necessary to compute a full expansion for polynomials and then use it to study smooth functions.

\subsection{An operator inequality}

Finally, as mentioned earlier in the introduction, one of the cornerstone of the proof of the main results of this paper is Proposition \ref{edgtn}. It can be viewed as a generalization of the following well-known result (see Definition 8.15 of \cite{nica_speicher_2006} for a proof). Given $x_1,\dots,x_k$ a family of semicircular variables, then 
\begin{equation*}
	\tau(x_1\cdots x_k) = \sum_{\pi\in NCP(k)}\prod_{(p,q)\in \pi} \tau(x_px_q),
\end{equation*}
where $NCP(k)$ is the set of non-crossing pair partition of of $[1,k]$. However, what we realized is that if one replaces $x_i$ by a polynomial in a family of semicircular variables, then in some way the formula above still holds. In particular, in some sense the trace of a product of polynomials only depends on the covariance between each of them and thus permuting those polynomials does not change the trace as much as one could expect it at first. Eventually, we get the following result of independent interest.

\begin{theorem}
	
	\label{masterineq}
	
	Given $\sigma\in \mathbb{S}_n$, and $\A$ a faithful tracial $\CC^*$-algebra, we define $m_{\sigma}: (\A\otimes\CC^*(x))^{\otimes n}\to \A\otimes\CC^*(x) $ by the following formula on simple tensors,
	$$ m_{\sigma}: (a_1\otimes A_1) \otimes\dots\otimes (a_n\otimes A_n) \mapsto (a_1\dots a_n \otimes A_{\sigma(1)} \dots A_{\sigma(n)}). $$
	Then with $\norm{\cdot}$ the minimal tensor norm, given $P_1,\dots,P_n\in\A\langle X_1,\dots,X_d\rangle$, with $Q=P_1\otimes\cdots\otimes P_n$,
	\begin{equation*}
		\norm{m_{\sigma}\left(Q(x)\right)} \leq (80e)^{n}\times d^{3n}\times \sup_{\substack{(\alpha_{i,j})_{i\geq 0, j\in [1,n]}\in\R^+ \\ \sum_{i,j\geq 0} i \alpha_{i,j}\leq 3n \\ \forall j,\ \sum_i \alpha_{i,j} = 1}} \prod_{j\in [1,n]} \prod_{i\geq 0} \left( \sup_{p_1,\dots,p_i\in [1,d]} \norm{\partial_{p_1}\otimes\cdots\otimes \partial_{p_i}(P_j)(x) } \right)^{\alpha_{i,j}}.
	\end{equation*}
	
\end{theorem}

Due to the rather heavy notations in the equation above, it is not straightforward, but because our results are with the minimal tensor norm instead of for example the projective tensor norm, this theorem let us handle some fairly non-trivial case. Let us take a practical example, if for each $j$, one can write $P_j=R_j^{k_j}$ for some $R_j\in\A\langle X_1,\dots,X_d\rangle$ and $k_j\in\N$, then thanks to Lemma \ref{xsdfgyujk} there is a universal constant $C$ which only depends on the degree of the polynomials $D_j$ such that for all $p_1,\dots,p_i\in [1,d]$, 
$$ \norm{\partial_{p_1}\otimes\cdots\otimes \partial_{p_i}(P_j)(x) } \leq C^ik_j^i\norm{D_j(x)}^{k_j}.$$
Consequently, one has that
\begin{align*}
	&\sup_{\substack{(\alpha_{i,j})_{i\geq 0, j\in [1,n]}\in\R^+ \\ \sum_{i,j\geq 0} i \alpha_{i,j}\leq 3n \\ \forall j,\ \sum_i \alpha_{i,j} = 1}} \prod_{j\in [1,n]} \prod_{i\geq 0} \left( \sup_{p_1,\dots,p_i\in [1,d]} \norm{\partial_{p_1}\otimes\cdots\otimes \partial_{p_i}(P_j)(x) } \right)^{\alpha_{i,j}} \\
	&\leq \sup_{\substack{(\alpha_{i,j})_{i\geq 0, j\in [1,n]}\in\R^+ \\ \sum_{i,j\geq 0} i \alpha_{i,j}\leq 3n \\ \forall j,\ \sum_i \alpha_{i,j} = 1}} \prod_{j\in [1,n]} \prod_{i\geq 0} \left( C^ik_j^i\norm{D_j(x)}^{k_j} \right)^{\alpha_{i,j}} \\
	&\leq \sup_{\substack{(\alpha_{i,j})_{i\geq 0, j\in [1,n]}\in\R^+ \\ \sum_{i,j\geq 0} i \alpha_{i,j}\leq 3n \\ \forall j,\ \sum_i \alpha_{i,j} = 1}} \left( C \max_j k_j\right)^{\sum_{i,j} i\alpha_{i,j}} \prod_{j\in [1,n]} \left(\norm{D_j(x)}^{k_j} \right)^{\sum_i \alpha_{i,j}} \\
	&\leq \left( C \max_j k_j\right)^{3n} \prod_{j\in [1,n]} \norm{D_j(x)}^{k_j}
\end{align*}
Thus assuming that the polynomials $D_j$ are self-adjoint, Theorem \ref{masterineq} implies that there is a constant $C$ such that for any $k_j\in [0,k]$, 
\begin{equation}
	\label{dnvflksms}
	\norm{m_{\sigma}(Q(x))} \leq C k^{3n} \norm{Q(x)}. 
\end{equation}
This is much better than any previous estimate that would have used the projective tensor norm and which would typically yield an inequality with $e^{nk}$ instead of $k^{n}$. \\

In order to really understand the subtlety of  Equation \eqref{dnvflksms}, let us compare it with a different case. We set $m_{(1 2)} : a_1\otimes b_1 \otimes a_2\otimes b_2 (\M_N(\C)\otimes\M_N(\C))^{\otimes 2} \mapsto a_1a_2\otimes b_2b_1 \M_N(\C)\otimes\M_N(\C)$, then with $U=\sum_{1\leq i,j\leq N} E_{i,j}\otimes E_{j,i}$
$$ m_{(1 2)}(U\otimes U) = N U,$$
Thus since $U$ is a unitary operator, we have that
$$ \norm{m_{(1 2)}(U\otimes U)} = N \norm{U\otimes U}. $$
Hence one can find an operator $\Delta$ which is the sum of $N^4$ simple tensor but $\norm{m_{(1 2)}(\Delta)} = N \norm{\Delta}$. On the contrary, if $k=\log N$, then $Q(x)$ is also the sum of $N^c$ simple tensors for some constant $c$, but Equation \eqref{dnvflksms} guaranties that $	\norm{m_{\sigma}(Q(x))}$ is bounded by $\norm{Q(x)}$ up to a constant proportional to a power of $\log N$. \\

For those reasons, one can compare Theorem \ref{masterineq} with the Haagerup inequality, see \cite{haagineq}. Indeed it states that for any polynomial $P$ of degree $n$ evaluated in free Haar unitaries and their adjoints, 
$$ \norm{P(u) } \leq (n+1) \norm{P(u)}_2,$$
where $\norm{\cdot}$ is the operator norm but $\norm{\cdot}_2$ is the $l^2$-norm on a Free group $F_d$. Although one would have expected that the operator norm would be bounded by the $l^2$-norm only up to a constant proportional to the exponential of the degree, it turns out that this constant is linear in the degree. Keeping in mind that $n$ times the norm of a polynomial is heuristically the norm of its derivative, with $m:l^2(F_d) \to B(l^2(F_d))$ the unbounded operator that associates an element of $l^2(F_d)$ to its equivalent as an operator on $l^2(F_d)$, heuristically the Haagerup inequality is equivalent to 
$$ \norm{m(P)} \leq \norm{\partial P}_2.$$
Hence the parallel with Theorem \ref{masterineq}, notably since one could adapt the proof of this paper to prove a theorem similar to Theorem \ref{masterineq} but with free Haar unitaries instead of free semicircular variables. \\

To conclude, let us mention that even though we assumed in Theorem \ref{masterineq} that $P_j$ was a polynomial, we could also take it to be a power series. In particular, one could take the function $X\mapsto e^{\i y_j P_j(X)}$, for any $y_j\in\R$. This would then imply that for some constant $C$,
$$ \norm{m_{\sigma}\left(Q(x)\right)} \leq (C \max_j |y_j|)^{3n}.$$
And thus by using Fourier inversion theorem, one could even work with sufficiently smooth functions.

\subsection*{Acknowledgments}

The main result of this paper is a project whom I have been working on for multiple years. Consequently, I would like to thank those with whom I have had the opportunity to exchange about it, starting with my former PhD advisors, Alice Guionnet and Beno\^it Collins, as well as my postdoc mentor Kevin Schnelli. I would also like to thank Ben Hayes, David Jekel as well as Evangelos Nikitopoulos for our conversations on the topic of the Peterson-Thom conjecture.

\section{Framework and standard properties}
\label{3definit}

\subsection{Usual definitions in free probability}
\label{3deffree}

In order to be self-contained, we begin by recalling the following definitions  from free probability.

\begin{defi}~
	\label{3freeprob}
	\begin{itemize}
		\item A \textbf{$\mathcal{C}^*$-probability space} $(\A,*,\tau,\norm{.}) $ is a unital $\mathcal{C}^*$-algebra $(\A,*,\norm{.})$ endowed with a \textbf{state} $\tau$, i.e. a linear map $\tau : \A \to \C$ satisfying $\tau(1_{\A})=1$ and $\tau(a^*a)\geq 0$ for all $a\in \A$. In this paper we always assume that $\tau$ is a \textbf{trace}, i.e. that it satisfies $\tau(ab) = \tau(ba) $ for any $a,b\in\A$. An element of $\A$ is called a 
		\textbf{noncommutative random variable}. We will always work with a faithful trace, namely, for $a\in\A$, $\tau(a^*a)=0$ if and only if $a=0$.
		
		\item  Let $\A_1,\dots,\A_n$ be $*$-subalgebras of $\A$, having the same unit as $\A$. They are said to be \textbf{free} if for all $k$, for all $a_i\in\A_{j_i}$ such that $j_1\neq j_2$, $j_2\neq j_3$, \dots , $j_{k-1}\neq j_k$:
		\begin{equation}
			\label{kddkdxkfl}
			\tau\Big( (a_1-\tau(a_1))(a_2-\tau(a_2))\dots (a_k-\tau(a_k)) \Big) = 0.
		\end{equation}
		Families of noncommutative random variables are said to be free if the $*$-subalgebras they generate are free.
		
		\item Let $ A= (a_1,\ldots ,a_k)$ be a $k$-tuple of random variables. The \textbf{joint $*$-distribution} of the family $A$ is the linear form $\mu_A : P \mapsto \tau\big[ P(A, A^*) \big]$ on the set of polynomials in $2k$ noncommutative variables. By \textbf{convergence in distribution}, for a sequence of families of variables $(A_N)_{N\geq 1} = (a_{1}^{N},\ldots ,a_{k}^{N})_{N\geq 1}$  
		in $\mathcal C^*$-algebras $\big( \mathcal A_N, ^*, \tau_N,\norm{.} \big)$,
		we mean the pointwise convergence of the map 
		$$ \mu_{A_N} : P \mapsto \tau_N \big[ P(A_N, A_N^*) \big],$$
		and by \textbf{strong convergence in distribution}, we mean convergence in distribution, and pointwise convergence
		of the map
		$$P \mapsto \big\| P(A_N, A_N^*) \big\|.$$
		
		\item A family of noncommutative  random variables $ x=(x_1,\dots ,x_d)$ is called a \textbf{free semicircular system} when the noncommutative random variables are free, self-adjoint ($x_i=x_i^*$), and for all $k$ in $\N$ and $i$, one has
		\begin{equation*}
			\tau( x_i^k) =  \int_{\R} t^k d\sigma(t),
		\end{equation*}
		with $d\sigma(t) = \frac 1 {2\pi} \sqrt{4-t^2} \ \mathbf 1_{|t|\leq2} \ dt$ the semicircle distribution. 
		
	\end{itemize}
	
\end{defi}

Let us finally fix a few notations concerning the spaces and traces that we use in this paper.

\begin{defi}
	\label{3tra}
	\begin{itemize}
		\item $\tr_N$ is the non-normalized trace on $\M_N(\C)$.
		\item $\ts_N$ is the normalized trace on $\M_N(\C)$.
		\item $I_N\in\M_N(\C)$ is the identity matrix.
		\item We denote $E_{r,s}$ the matrix with $1$ in the $(r,s)$ entry and zeros in all the other entries. 
		\item We identify $\M_N(\C)\otimes \M_k(\C)$ with $\M_{kN}(\C)$ through the isomorphism $E_{i,j}\otimes E_{r,s} \mapsto E_{i+rN,j+sN} $, similarly we identify $\tr_N\otimes\tr_k$ with $\tr_{kN}$.
		\item If $A=(A_1,\dots,A_d)$ and $B=(B_1,\dots,B_d)$ are two families of operators, then we denote $A\otimes B=(A_1\otimes B_1,\dots,A_d\otimes B_d)$. We typically use the notation $X\otimes \id$ for the family $(X_1\otimes \id,\dots,X_1\otimes \id)$.
	\end{itemize}
\end{defi}

\subsection{The full Fock space}

The full Fock space gives us an explicit construction for a $\CC^*$-algebra which contains a system of free semicircular variables. For a full introduction to the topic, we refer to Chapter 7 of \cite{nica_speicher_2006}. Unlike previous work of the same author, in this paper we will need this construction. Let us start by giving its definition.

\begin{defi}
	Let $\Hi$ be a Hilbert space.
	
	\begin{itemize}
		
		\item We define the \textbf{full Fock space} over $\Hi$ as 
		$$ \F(\Hi) = \bigoplus_{n\geq 0} \Hi^{\otimes n}, $$
		where by convention $\Hi^{\otimes 0}$ is a one-dimensional vector space which is usually denoted by $\C\Omega$ where $\Omega$ is a distinguished vector of norm one, called the \textbf{vacuum vector}.
		\item The \textbf{vacuum expectation state} $\tau$ on $\F(\Hi)$ is the linear functional on $B(\F(\Hi))$ defined by 
		$$\tau: T\in B(\F(\Hi)) \mapsto \langle T\Omega,\Omega\rangle. $$
		\item For each $\xi\in\F(\Hi)$, one defines $l(\xi)$ and $r(\xi)$ the \textbf{left and right creation operator} associated to $\xi$ by $l(\xi) \Omega = r(\xi)\Omega = \xi$ and for all $n\geq 1$,
		$$ l(\xi) \zeta_1\otimes\cdots\otimes\zeta_n = \xi\otimes \zeta_1\otimes\cdots\otimes\zeta_n, $$
		$$ r(\xi) \zeta_1\otimes\cdots\otimes\zeta_n = \zeta_1\otimes\cdots\otimes\zeta_n\otimes\xi.$$
		\item With the convention that the scalar product on $\Hi$ is antilinear with respect to the first argument, the adjoint of the operators above, usually denoted as the \textbf{left and right annihilation operators}, are such that $l(\xi)^* \Omega = r(\xi)^*\Omega = 0$ and for all $n\geq 1$,
		$$ l(\xi)^* \zeta_1\otimes\cdots\otimes\zeta_n = \langle\xi,\zeta_1\rangle\ \zeta_2\otimes\cdots\otimes\zeta_n, $$
		$$ r(\xi)^* \zeta_1\otimes\cdots\otimes\zeta_n = \langle\xi,\zeta_n\rangle\ \zeta_1\otimes\cdots\otimes\zeta_{n-1}.$$
		
	\end{itemize}
	
\end{defi}

One can now state the following proposition which is the reason as to why one needs the full Fock space. For a proof, we refer to Corollary 7.19 of \cite{nica_speicher_2006} as well as the propositions preceding it.
	
\begin{prop}
	Let $\Hi$ be a Hilbert space and let $e_1,\dots,e_n$ be orthonormal vectors, then
	$$x=\big(l(e_1)+l(e_1)^*,\dots,l(e_n)+l(e_n)^*\big)$$
	is a free semicircular system. Besides, the restriction of $\tau$ to the $\CC^*$-algebra generated by $x$ is a faithful trace. We will denote this $\CC^*$-algebra by $\CC^*(x)$, and $l(e_i)$ by $l_i$.
\end{prop}

In the rest of the paper, when we define $x$ to be a $d$-tuple of free semicircular variables, we implicitly pick a Hilbert space of dimension $d$ as well as an orthonormal basis, and $x$ is built with the Proposition above.

\subsection{The free product of $\M_N(\C)$ and $\CC^*(x)$}

\label{sdkmvdoijiytd}

In order to consider deterministic matrices and free semicircular variables, we need to define a space in which they can both live in simultaneously. To do so, we define	$(\widetilde{\B}_N,\widetilde{\tau}_N)$ as the free product of $\M_N(\C)$ with a system of $d$ free semicircular variable, that is the $\mathcal{C}^*$-probability space built with Theorem 7.9 of \cite{nica_speicher_2006}. Note that when restricted to $\M_N(\C)$, $\widetilde{\tau}_N$ is just the regular renormalized trace on matrices. However this definition is not especially intuitive, besides we will need a different construction in the rest of the paper. Hence we set the following.

\begin{defi}
	\label{eiucskc}
	Given $d,N\in\N$, as well as free semicircular variables $(x^i_{r,s})_{1\leq i\leq d, 1\leq r\leq s\leq N}$ as well as $(y^i_{r,s})_{1\leq i\leq d, 1\leq r< s\leq N} $, we set $\B_N=\M_N(\B)$ and 
	$$
	(x_i^N)_{r,s} = \left\{
	\begin{array}{ll}
		\frac{x_{r,s}^i+\i\ y_{r,s}^i}{\sqrt{2N}} & \mbox{if } r<s, \\
		x^i_{r,s} & \mbox{if } r=s, \\
		\frac{x_{s,r}^i-\i\ y_{s,r}^i}{\sqrt{2N}} & \mbox{if } r>s.
	\end{array}
	\right.
	$$
	Then we endow $\B_N$ with the involution ${(a_{i,j})_{1\leq i,j\leq N}}^* = (a_{j,i}^*)_{1\leq i,j\leq N}$ and the trace $\tau_N : A\in\B_N\mapsto\tau\left(\frac{1}{N}\tr_N(A)\right)$.	
\end{defi}

Note that with $\1_{\B}$ the unit of $\B$, one can easily view $\M_N(\C)$ as a subalgebra of $\B_N$ thanks to the morphism $(a_{r,s})\in\M_N(\C)\mapsto (a_{r,s}\1_{\B})\in\B_N$.

\begin{prop}
	\label{dojvdsomv}
	With the trace and the involution defined as above, $\B_N$ is a $\mathcal{C}^*$-probability spaces. Besides the family $(x_i^N)_{1\leq i\leq d}$ is a free semicircular system, and it is free from $\M_N(\C)$.
\end{prop}

\begin{proof}
	The trace $\tau_N$ is faithful since given $a\in\B_N$, 
	$$ \tau_N(a^*a) = \frac{1}{N} \sum_{1\leq i,j\leq N} \tau(a_{i,j}^*a_{i,j}),$$
	which is equal to $0$ if and only if for all $i,j\in [1,N]$, $\tau(a_{i,j}^*a_{i,j})=0$, and since $\tau$ is faithfull this is equivalent to $a=0$. Let us now prove that $(x_i^N)_{1\leq i\leq d}$ is a free semicircular system, free from $\M_N(\C)$. Thanks to \cite[Lemma 5.4.7]{alice}, the family $(X^{i,k}_{r,s})_{1\leq i\leq d, 1\leq r\leq s\leq N}$ and $(Y^{i,k}_{r,s})_{1\leq i\leq d, 1\leq r< s\leq N} $, where $X^{i,k}_{r,s}$ and $Y^{i,k}_{r,s}$ are independent GUE random matrices of size $k$ converges in distribution towards $(x^i_{r,s})_{1\leq i\leq d, 1\leq r\leq s\leq N}$ and $(y^i_{r,s})_{1\leq i\leq d, 1\leq r< s\leq N} $. Consequently, the family $(X^N_i)_{1\leq i\leq d}\cup (Z_j^N\otimes I_k)_{1\leq j\leq r}$, where $Z_j^N\in \M_N(\C)$ and $X^N_i$ is defined with independent GUE random matrices of size $k$ instead of free semicircular variables, converges in distribution towards $(x^N_i)_{1\leq i\leq d}\cup (Z_j^N)_{1\leq j\leq r}$. 
	
	However $(X^N_i)_{1\leq i\leq d}$ has the same law as a $d$-tuple of GUE random matrices of size $kN$. Thus thanks once again to \cite[Lemma 5.4.7]{alice}, $(X^N_i)_{1\leq i\leq d}\cup (Z_j^N\otimes I_k)_{1\leq j\leq r}$ converges in distribution towards $(x_i)_{1\leq i\leq d}\cup (Z_j^N)_{1\leq j\leq r}$, where $(x_i)_{1\leq i\leq d}$ is a free semicircular system free from the family $(Z_j^N)_{1\leq j\leq r}$. Hence the conclusion by unicity of the limit.
\end{proof}

Let us finally note that in order to shorten the notations, we often drop the $N$ and simply write $x_i$ instead $x_i^N$ since thanks to the previous theorem, we have that the trace $\tau_N$ of a polynomial in $x^N$ is the same for any $N$ (i.e. the trace of a polynomial evaluated in a free semi-circular system).

\subsection{A few notions of operator algebra}

In this paper we will be taking the norm of tensor of $\CC^*$-algebras. To do so, we work with the minimal tensor product also named the spatial tensor product. For more information we refer to \cite[Chapter 6]{murphy}.

\begin{defi}
	\label{1mini}
	Let $\A$ and $\B$ be $\CC^*$-algebras with faithful representations $(H_{\A},\phi_{\A})$ and $(H_{\B},\phi_{\B})$, then if $\otimes_2$ is the tensor product of Hilbert spaces, $\A\otimes_{\min}\B$ is the completion of the image of $\phi_{\A}\otimes\phi_{\B}$ in 	$B(H_{\A}\otimes_2 H_{\B})$ for the operator norm in this space. This definition is independent of the representations that we fixed.
	
\end{defi}

It is not necessary to understand in depth the minimal tensor product to read the rest of the paper. We mainly need it in order to use the following two lemmas. The first one is Lemma 4.1.8 from \cite{ozabr}.
\begin{lemma}
	\label{1faith}
	Let $(\A,\tau_{\A})$ and $(\B,\tau_{\B})$ be $\CC^*$-algebra with faithful traces, then $\tau_{\A}\otimes\tau_{\B}$ extends uniquely to a faithful trace $\tau_{\A}\otimes_{\min}\tau_{\B}$ on $\A\otimes_{\min}\B$. 
\end{lemma}

We will need the following lemma to ensure that the map $\id_{\A}\otimes\tau$ is bounded.

\begin{prop}
	\label{docmmfsokcm}
	Let $(\A,\tau_{\A})$ and $(\B,\tau_{\B})$ be $\CC^*$-algebras with $\tau_{\A}$ a faithful trace and $\tau_{\B}$ a state, then the norm of the map $\id_{\A}\otimes\tau_{\B}: \A\otimes_{\min}\B \to\A$ is equal to one.
\end{prop}

\begin{proof}
	
	Let us start by showing that if $\tau$ is a faithful trace on a $\CC^*$-algebra $\mathcal{M}$, and $x\in\mathcal{M}$ is such that for all $y\in\mathcal{M}$, $\tau(xy^*y)\geq 0$ then $x$ is a non-negative operator. Indeed $L:\mathcal{M}\to B(L^2(\mathcal{M},\tau))$ which to an element $x\in\mathcal{M}$ associates the operator $\widehat{x}:\widehat{a}\in L^2(\mathcal{M},\tau)\mapsto \widehat{xa}$ is an isometry since
	$$ \norm{\widehat{x}}^2 = \sup_{y\neq 0} \frac{\tau((xy)^*xy)}{\tau(y^*y)} \leq \norm{x}^2,$$
	and with $\sigma_{x^*x}$ the spectral measure of $x$, $f:\R\to\R$ a continuous function, $C$ the infimum of the support of $f$,
	$$ \norm{\widehat{x}}^2 \geq \frac{\tau((xf(x^*x))^*xf(x^*x))}{\tau(f(x^*x)^*f(x^*x))} = \frac{\int_{\R} a|f(a)|^2 d\sigma_{x^*x}(a) }{\int_{\R} |f(a)|^2 d\sigma_{x^*x}(a)} \geq C.$$
	Hence $L$ is indeed an isometry since one can take $f$ such that $C$ is as close of $\norm{x^*x}$ as we want and $f(x^*x)\neq 0$. 
	As a consequence, if $x\in\mathcal{M}$ is such that for all $y\in\mathcal{M}$, $\tau(xy^*y)\geq 0$, then $\langle \widehat{x} \widehat{y^*}, \widehat{y^*} \rangle_{L^2(\mathcal{M},\tau)} = \tau(xy^*y)\geq 0$, and $\widehat{x}$ is a positive operator. However since $L$ is an isometry and a morphism of $\CC^*$algebra, $L(M)$ is a $\CC^*$algebra and $L^{-1}$ a morphism of  $\CC^*$algebra. Hence $x=L^{-1}(\widehat{x})$ is also positive.
	
	Thus for any $x\in \A\otimes\B$, $y\in A$, if $x$ is non-negative,
	$$ \tau_{\A}\left( \id_{\A}\otimes\tau_{\B}(x) y^*y \right) = \tau_{\A}\otimes\tau_{\B}\big(x((y^*y)\otimes\id_{\B})\big) \geq 0.$$
	Thus $\id_{\A}\otimes\tau_{\B}(x)$ is non-negative. Thus for any $x\in \A\otimes\B$, 
	$$ \norm{x}\id_{\A} - \id_{\A}\otimes\tau_{\B}\left( x \right) = \id_{\A}\otimes\tau_{\B}\left( \norm{x}\id_{\A}\otimes\id_{\B} - x \right)\geq 0,$$
	$$ \norm{x}\id_{\A} + \id_{\A}\otimes\tau_{\B}\left( x \right) = \id_{\A}\otimes\tau_{\B}\left( \norm{x}\id_{\A}\otimes\id_{\B} + x \right)\geq 0.$$
	Thus $\norm{\id_{\A}\otimes\tau_{\B}\left( x \right)} \leq \norm{x}$, hence the conclusion.
	
\end{proof}

\subsection{Noncommutative polynomials and derivatives}
\label{3poly}

We denote by 
$$ \A\langle X_1,\dots,X_d,Z_1,\dots,Z_{q}\rangle = \A\otimes \C\langle X_1,\dots,X_d,Z_1,\dots,Z_{q}\rangle$$
the set of \textbf{noncommutative polynomials} in $d+q$ variables with coefficients in $\A$. Thus an element $P\in \A\langle X_1,\dots,X_d,Z_1,\dots,Z_{q}\rangle$ can be written uniquely as such 
$$ P = \sum_{M \text{ monomial}} a_M\otimes M,$$
where $a_M\in \A$, and there is a finite number of monomials $M$ such that $a_M\neq 0$. Besides for $P\in \A\langle X_1,\dots,X_d,Z_1,\dots,Z_{q}\rangle$ as above, we also define its degree by
\begin{equation}
	\deg P = \sup \{ \deg M\ |\ M\text{ monomial such that } a_M\neq 0\}.
\end{equation}

%\noindent We define an involution $*$ on $\A\langle X_1,\dots,X_d,Z_1,\dots,Z_{q}\rangle$ by $(AB)^* = B^*A^*$ for any unitary monomials, $X_i^*=X_i$ for all $i\in [1,d]$, $Z_i^*=Z_{i+q}$ for all $i\in [1,q]$, $Z_i^*=Z_{i-q}$ for all $i\in [q+1,2q]$, and 
%$$ P^* = \sum_{M \text{monomial}} a_M^*\otimes M^*.$$
%$P\in \A\langle X_1,\dots,X_d,Z_1,\dots,Z_{q}\rangle$ is said to be self-adjoint if $P^* = P$. Self-adjoint polynomials have the property that if $x_1,\dots,x_d,z_1,\dots,z_q$ are elements of a $\mathcal{C}^*$-algebra such that $x_1,\dots,x_d$ are self-adjoint, then so is $P(x_1,\dots,x_d)$.

Now let us define the noncommutative derivative, it is a widely used tool in the field of probability, see for example the work of Voiculescu, \cite{Vo91,refdif,refdif2}.

\begin{defi}
	\label{3application}
	
	If $1\leq i\leq d$, one defines the \textbf{noncommutative derivative} 
	$$\partial_i: \C\langle X_1,\dots,X_d,Z_1,\dots,Z_{q}\rangle \longrightarrow \C\langle X_1,\dots,X_d,Z_1,\dots,Z_{q}\rangle^{\otimes 2}$$
	by its value on a monomial  $M\in \C\langle X_1,\dots,X_d,Z_1,\dots,Z_{q}\rangle$  given by
	$$ \partial_i M = \sum_{M=AX_iB} A\otimes B \,$$
	and then extend it by linearity to all polynomials. We can also define $\partial_i$ by induction with the formulas,
	\begin{equation}
		\label{3leibniz}
		\begin{array}{ccc}
			&\forall P,Q\in \C\langle X_1,\dots,X_d,Z_1,\dots,Z_{q}\rangle,\quad \partial_i (PQ) = \partial_i P \times \left(1\otimes Q\right)	 + \left(P\otimes 1\right) \times \partial_i Q , \\
			& \\
			&\forall i,j,\quad \partial_i X_j = \delta_{i,j} 1\otimes 1,\quad \partial_i Z_j = 0\otimes 0.
		\end{array}
	\end{equation}
	
\end{defi}

In this paper however, we will need to work not only with scalar polynomials, but also operator valued polynomials, i.e. elements of $\A\langle X_1,\dots,X_d,Z_1,\dots,Z_{q}\rangle$. In order to do so we introduce the following definition which is an extension of the previous one.

\begin{defi}
	\label{3applicationopv}
	
	Given $1\leq i_1,\dots,i_n\leq d$, one defines the \textbf{operator-valued noncommutative derivative} 
	$$\partial_{i_1}\otimes\cdots\otimes \partial_{i_n}: \A\langle X_1,\dots,X_d,Z_1,\dots,Z_{q}\rangle \longrightarrow \A\otimes (\C\langle X_1,\dots,X_d,Z_1,\dots,Z_{q}\rangle^{\otimes n})$$
	by the following composition of maps,
	\begin{align}
		\partial_{i_1}\otimes\cdots\otimes \partial_{i_n} := (\id_{\A}\otimes\partial_{i_n}\otimes &\id_{\C\langle X_1,\dots,X_d,Z_1,\dots,Z_{q}\rangle^{\otimes n-1}})\circ \cdots \\
		&\dots\circ (\id_{\A}\otimes\partial_{i_2}\otimes\id_{\C\langle X_1,\dots,X_d,Z_1,\dots,Z_{q}\rangle}) \circ (\id_{\A}\otimes\partial_{i_1}) \nonumber
	\end{align}
	
\end{defi}

To conclude this subsection, we would like to mention that the map $\partial_i$ is related to the so-called \textbf{Schwinger-Dyson equations} on semicircular variable thanks to the following proposition whom one can find a proof in \cite[Lemma 5.4.7]{alice}. Although we will not make direct use of it, the heuristic behind certain results of this paper, such as Proposition \ref{edgtn}, are linked with this one.

\begin{prop}
	\label{3SDE}
	Let $ x=(x_1,\dots ,x_d)$ be a free semicircular system in a $\mathcal{C}^*$-probability space $(\A,*,\tau,\norm{.}) $, then for any $Q\in \C\langle X_1,\dots,X_d\rangle$,
	$$ \tau(Q(x)\ x_i) = \tau\otimes\tau(\partial_i Q(x))\ .$$
	
\end{prop}

\subsection{Definition of our random matrices}

Let us now recall the definition of Gaussian and Haar unitary random matrices and stating a few useful properties about them. 

\begin{defi}
	\label{3GUEdef}
	A \textbf{GUE random matrix} $X^N$ of size $N$ is a self-adjoint matrix whose coefficients are random variables with the following laws:
	\begin{itemize}
		\item For $1\leq i\leq N$, the random variables $\sqrt{N} X^N_{i,i}$ are independent centered Gaussian random variables of 
		variance $1$.
		\item For $1\leq i<j\leq N$, the random variables $\sqrt{2N}\ \Re{X^N_{i,j}}$ and $\sqrt{2N}\ \Im{X^N_{i,j}}$ are independent 
		centered Gaussian random variables of variance $1$, independent of  $\left(X^N_{i,i}\right)_i$.
	\end{itemize}
\end{defi}

\begin{defi}
	\label{2Haardef}
	A \textbf{Haar unitary random matrix} of size $N$ is a random matrix distributed according to the Haar measure on the group of unitary matrices of size $N$.
\end{defi}

Now to finish this subsection we state a property that we use several times in this paper. For the proof we refer to \cite[Proposition 2.11]{un}.

\begin{prop}
	\label{3bornenorme}
	There exist constants $C,D$ and $\alpha$ such that for any $N\in\N$, if $X^N$ is a GUE random matrix of size $N$, then for any 
	$u\geq 0$,
	
	$$ \P\left(\norm{X^N}\geq u+D \right) \leq e^{-\alpha u N} . $$
	
	\noindent Consequently, for any $k\leq \alpha N /2 $,
	
	$$ \E\left[\norm{X^N}^k\right] \leq C^k .$$
	
\end{prop}

\subsection{Circle sets}

The notion of circle set was invented for this paper. Before giving the heuristic that led us to introduce this notation,let us give its rigorous definition.

\begin{defi}
	A set $E$ with $n$ elements is said to be a \textbf{circle set} if its elements can be ordered in a list $\{k_1,\dots,k_n\}$. Then if $i\in E$ is the $l$-th element of $E$, one defines $ i^+ := k_{l+1}$ and $i^-:=k_{l-1}$ with convention $k_{n+1}=k_1$ and $k_0=k_n$.
\end{defi}

For example, a natural example of circle set with $n$ elements is the set $\{e^{\frac{2\i \pi j}{n}}\ |\ j\in[1,n]\}$. There is a natural ordering of the points since given a point on the unit circle, one can find the next one by moving in the counter-clockwise direction. Note that the starting point of a circle set $E$, i.e. the element $k_1$ in a given ordering $\{k_1,\dots,k_n\}$ of $E$, usually does not matter. Unless we specifically mention the opposite, one could use the ordering $\{k_l,\dots,k_n,k_{1},\dots,k_{l-1}\}$ interchangeably. Following those remarks we define a ``canonical'' circle set with $n$ elements.

\begin{defi}
	Let $n>0$, then we denote by $((1,n))$ the circle set,
	$$\{e^{\frac{2\i \pi j}{n}}\ |\ j\in[1,n]\},$$
	whose order is given by the counter-clockwise direction.
\end{defi}

Note that since a circle set is ordered, one can define a notion of interval. We will denote an open interval by $]i,j[$ rather than $(i,j)$ since the later will simply be an (ordered) pair of element.

\begin{defi}
	Let $E$ ba a circle set, $i\neq j\in E$, the \textbf{interval} $]i,j[\subset E$ (respectively $[i,j]$) is defined with the following construction, let $k_0=i$ and $k_{l+1}=k_l^+$, then if $m$ is such that $k_m=j$, we set $]i,j[ = \{k_1,\dots,k_{m-1}\}$ (respectively $[i,j] = \{k_0,\dots,k_{m}\}$).
	
\end{defi}

Finally, we will need to be able to be able to take a $l$-tuple of element of $E$ while following the ordering, which is why we introduce the following notation.

\begin{defi}
	Given $E,F$ circle sets, the set $(k_{j})_{j\in F}\in E$ is defined as such,
	$$ \{(k_{j})_{j\in F} \in E^F\ |\ \forall j\in F, k_j\in ]k_{j^-},k_{j^+}[\}. $$
\end{defi}

Note that for example, if $E=((1,10))$ and $F=((1,3))$, then $(2,6,10)$ and $(10,2,6)$ are distinct elements of $(k_{j})_{j\in F}\in E$, but $(6,2,10)$ is not. 

Finally, we need to define a way to add points to a given circle set, we do so with the following definition.

\begin{defi}
	Given $E,F$ circle sets, the set $[l_j]_{j\in F}\in E$ is defined by induction.
	\begin{itemize}
		\item First for any given $j_1\in F$, we define a collection of circle set $E^1$ defined as such, given $i\in E$, we build a circle set whose elements are $E$ and an extra point $l_{j_1}$, and whose ordering is the same one as $E$ with the exception that $l_{j_1}$ is in between $i$ and $i^+$.
		\item To build $E^{m+1}$, we set $j_{m+1}= j_m^+$, then given $F\in E^{m}$, $i\in [l_{j_m}, l_{j_1}[$, we build a new circle set whose elements are $F$ and an extra point $l_{j_{m+1}}$, and whose ordering is the same one as $F$ with the exception that $l_{j_{m+1}}$ is in between $i$ and $i^+$.
		\item we stop this process when $j_m^+=j_1$.
	\end{itemize}
	Finally the set $[l_j]_{j\in F}\in E$ is the collection of added points $l_{j_1},\dots,l_{j_m}$ with possible orderings given by $E^m$.
\end{defi}

\subsection{The set of configurations}

Since the set of non-crossing partitions of a given set $E=\{k_1,\dots,k_n\}$ is the same as the one of the set $\{k_l,\dots,k_n,k_{1},\dots,k_{l-1}\}$, there is a natural link between circle sets and non-crossing partition. In this paper we will use a variation of non-crossing pair partitions where we allow elements to belong to several subset or to none.

\begin{defi}
	Given $E$ a circle set, a \textbf{configuration} $K$ of $E$ is a subset of $(\{i,j\})_{i,j\in E, i\neq j}$ (with the convention $\{i,j\}=\{j,i\}$) such that for all $\{i,j\}, \{k,l\} \in K$, one has that either $i,j\in [k,l]$ or $i,j\in [l,k]$.
	
	We denote by $\K(E)$ the \textbf{set of configuration} of $E$, and by $\K_n$ the one of $((1,n))$. Besides, one also sets,
	$$ \widehat{K} = \{i\in E\ |\ \forall j\in E, \{i,j\}\notin K \}.$$
\end{defi}

Practically speaking, $\widehat{K}$ is the set of elements of $E$ which in the configuration $K$ are not linked with any other element of $E$. The main reason we introduce the notion of circle set is due to its link with configurations. Notably it allows us to define the following object.

\begin{defi}
	\label{cgoijdvokds}
	Given a circle set $E$ and a configuration $K\in \K(E)$, for any $i\in E$, we set 
	$$E_i(K)=\{j\in E\ |\ \{i,j\}\in K \} \cup \{i\}.$$
	This set inherit the ordering of $E$, indeed we can write $E_i=\{k_1,\dots k_m\}$ with $k_1=i$ and for all $l\in [1,m]$, $k_l\in E_i$ is chosen such that $k_l\in ]k_{l-1},k_{l+1}[$ with convention $k_0=k_m$ and $k_{m+1}=k_1$. Finally we reverse this ordering and we set 
	$$ C_i(K) = \{k_m,k_{m-1},\dots,k_2,k_1\}.$$
\end{defi}

In practice, we will not use the set $E_i(K)$, however the set $C_i(K)$ will be used in most of the proofs. We will also need a variation of this set with the following definition.

\begin{defi}
	Given a circle set $E$ and a configuration $K\in \K(E)$, we define 	
\begin{equation}
	C_i^*(K) =
	\begin{cases}
		C_i(K)\setminus\{i\} & \text{if } C_i(K)\neq\{i\}\\
		\{i\} & \text{otherwise.}
	\end{cases}       
\end{equation}
	We also set $ C_i^{**}(K) = C_i(K)\setminus\{i,i^-\}$ where $i^-$ is chosen with respect to the new ordering of $C_i(K)$. More precisely, with the notation of Definition \ref{cgoijdvokds}, 
	$$ C_i^{**}(K) = \{k_m,k_{m-1},\dots,k_3\}.$$
\end{defi}

\section{Proof of the main lemmas}

The aim of this section is to prove several lemmas which we will then use in the proof of our main results in the next section.

\subsection{An asymptotic expansion for scalar polynomials}

In this subsection, we focus on adapting the main results of \cite{trois} with the notations of this paper. We also reformulate the formula so that one can handle them more easily with the results of the next subsection. Although this step comes first in the proof of our main theorems, we invite the reader to skip this subsection at first and to focus instead on the next one, notably Proposition \ref{edgtn} and Lemma \ref{sdtyuikmn}, to better understand why we reformulate the expansion of \cite{trois} with those new notations.

\begin{lemma}
	\label{3apparition}
	Let $x,x^1,\dots$ be free semicircular systems of $d$ variables, free between each other. Then with $T_n = \{t_1,\dots,t_{2n}\}$ a sequence of non-negative number,  $T_n = \{\widetilde{t}_1,\dots,\widetilde{t}_{2n}\}$ the same set but ordered by increasing orders, and $I = \{I_1,\dots,I_{2n}\}\in (\N\setminus\{0\})^{2n}$, with $t_0=0$, we set
	$$ x_{i,I}^{T_{n}} = \sum_{l=1}^{2n} (e^{-\widetilde{t}_{l-1}} -e^{-\widetilde{t}_l})^{1/2} x^{I_{l}}_i + e^{- \widetilde{t}_{2n}/2} x_i . $$
	
	\noindent There exists families $J_n\subset (\N\setminus\{0\})^{2n}$, as well as families $J_{n+1}^{l,1},J_{n+1}^{l,2},\widetilde{J}_{n+1}^{l,1},\widetilde{J}_{n+1}^{l,2} \subset J_{n+1}$ with a natural bijection between each of them and $J_n$, such that if we define the following subfamily of $(X_{i,I})_{i\in [1,d], I\in J_{n+1}}$,
	$$X_{l,1} = \left(X_{i,I}\right)_{i\in [1,d], I\in J_{n+1}^{l,1}}\cup (Z_j)_{1\leq j\leq q},\quad X_{l,2} = \left(X_{i,I}\right)_{i\in [1,d],I\in J_{n+1}^{l,2}}\cup (Z_j)_{1\leq j\leq q}, $$
	$$\widetilde{X}_{l,1} = \left(X_{i,I}\right)_{i\in [1,d],I\in \widetilde{J}_{n+1}^{l,1}} \cup (Z_j)_{1\leq j\leq q},\quad \widetilde{X}_{l,2} = \left(\widetilde{X}_{i,I}\right)_{i\in [1,d],I\in \widetilde{J}_{n+1}^{l,2}}\cup (Z_j)_{1\leq j\leq q},$$
	
	\noindent since there is a natural bijection between $J_n$ and $J_{n+1}^{l,1}$, one can evaluate a polynomial in $(X_{i,I})_{i\in [1,d], I\in J_{n}} \cup (Z_j)_{1\leq j\leq q}$ in the variables $X_{l,1}$, and similarly for $X_{l,2},\widetilde{X}_{l,1}$ and $\widetilde{X}_{l,2}$. Besides, we then define the following operators, for $s$ from $1$ to $2n+1$, for $W$ a monomial in $(X_{i,I})_{i\in [1,d], I\in J_{n+1}}$,
	\begin{align}
		\label{xfgukokgc}
		&L_s^{n+1}(Q) := \frac{1}{2} \sum_{\substack{1\leq i,j\leq d \\ I,I',J,J'\in J_n}}	\sum_{\substack{W\sim AX_{j,J'}BX_{i,I'}CX_{j,J}DX_{i,I} \\\text{where } J\underset{\mathclap{{}^{s}}}{\sim}J'}} B\left(X_{s,1}\right) A(\widetilde{X}_{s,1}) D(\widetilde{X}_{s,2}) C\left(X_{s,2}\right),
	\end{align}
	
	\noindent where $W\sim W'$ if there exists $P,Q$such that $W=PQ$ and $W'=QP$, and $I \underset{\mathclap{{}^{s}}}{\sim} I'$ if $I_s=I_s'$. Note that since $I\in J_n$ only has $2n$ elements, the condition "$I,I'\in J_n$, such that $I_{2n+1}=I_{2n+1}'$" is satisfied for any $I,I'$. Finally, if $T_{n+1}$ is a set of $2n+2$ numbers, with $\widetilde{T}_{n} = \{\widetilde{t}_1,\dots,\widetilde{t}_{2n}\}$ the set which contains the first $2n$ elements of $T_{n+1}$ but  sorted by increasing order, we set 
	\begin{equation}
		\label{xdffgyui}
		L^{T_{n+1}}(Q) := \1_{[\widetilde{t}_{2n},t_{2n+2}]}(t_{2n+1}) L^{n+1}_{2n+1}(Q) + \sum_{1\leq s\leq 2n} \1_{[\widetilde{t}_{s-1},\widetilde{t}_s]}(t_{2n+1}) L_s^{n+1}(Q).
	\end{equation}
	
	\noindent Then, given $Q\in \C\langle X_1,\dots,X_d,Z_1,\dots,Z_{q}\rangle$,
	\begin{align*}
		&\E\left[\tau_N\Big(Q\left(X^N,Z^N\right)\Big)\right] = \sum_{p= 0}^{\deg Q /4} \frac{1}{N^{2p}} \int_{A_p} e^{-t_1-\cdots-t_{2p}}\tau\Big( \left(L^{{T}_p} \dots L^{{T}_1}\right)(Q) \left(x^{T_p},Z^N\right) \Big)\  dt_1\dots dt_{2p}.
	\end{align*}
	where $X^N$ is a $d$-tuple of independent GUE random matrices, $Z^N$ is a family of deterministic matrices, and
	$$A_p = \{ t_{2p}\geq t_{2p-2}\geq \dots \geq t_2\geq 0 \}\cap\{\forall s\in [1,p], t_{2s} \geq t_{2s-1} \geq 0\} \subset \R^{2p}.$$	
\end{lemma}

Note that some notations in the Lemma above have not been written explicitly, such as for example the construction of the family $J_n$. However its definition will not be of much use in this paper and hence for the sake of limiting the size of the previous theorem we choose to only explicit what is strictly necessary.

\begin{proof}
	This lemma is a reformulation of Lemma 3.6 and Proposition 3.7 of \cite{trois}. With the notations of this paper, one has that
	\begin{align*}
		&\E\left[ \ts_N\Big( Q(X^N,Z^N) \Big) \right] =\sum_{i\geq 0} \frac{1}{N^{2i}} \int_{A_i } \int_{[0,1]^{4i}} \tau\Big( \left(L^{{T}_i}_{\rho_i,\beta_i,\gamma_i,\delta_i} \dots L^{{T}_1}_{\rho_1,\beta_1,\gamma_1,\delta_1}\right)(Q) (x^{T_i},Z^N) \Big)\ d\rho\, d\beta\, d\gamma\, d\delta\  dt,
	\end{align*}
	where
	\begin{align*}
		L^{T_{n+1}}_{\rho_{n+1},\beta_{n+1},\gamma_{n+1},\delta_{n+1}}(Q) :=  e^{-t_{2n+2}-t_{2n+1}} \Bigg(& \1_{[\widetilde{t}_{2n},t_{2n+2}]}(t_{2n+1}) L^{n,\rho_{n+1},\beta_{n+1},\gamma_{n+1},\delta_{n+1}}_{2n+1}(Q) \\
		& + \sum_{1\leq s\leq 2n} \1_{[\widetilde{t}_{s-1},\widetilde{t}_s]}(t_{2n+1}) L_s^{n,\rho_{n+1},\beta_{n+1},\gamma_{n+1},\delta_{n+1}}(Q) \Bigg),
	\end{align*}
	and
	\begin{align*}
		&L_s^{n,\rho_{n+1},\beta_{n+1},\gamma_{n+1},\delta_{n+1}}(Q) \\
		&:= \frac{1}{2} \sum_{\substack{1\leq i,j\leq d \\ I,J\in J_n\\ \text{such that } I_s=J_s}} \Big(\partial_{\delta_{n+1},j,I}^2\left( \partial_{\beta_{n+1},i}^1 D_{\rho_{n+1},i} Q\right)\left(X_{s,1}\right) \boxtimes \partial_{\delta_{n+1},j,I}^1 \left( \partial_{\beta_{n+1},i}^1 D_{\rho_{n+1},i} Q\right)\left(\widetilde{X}_{s,1}\right) \Big) \\
		&\quad\quad\quad\quad\quad\quad\quad\quad \boxtimes \Big(\partial_{\gamma_{n+1},j,J}^2\left( \partial_{\beta_{n+1},i}^2 D_{\rho_{n+1},i} Q\right)\left(\widetilde{X}_{s,2}\right) \boxtimes \partial_{\gamma_{n+1},j,J}^1 \left( \partial_{\beta_{n+1},i}^2 D_{\rho_{n+1},i} Q\right)\left(X_{s,2}\right) \Big).
	\end{align*} 
	First and foremost, since $Q$ will always be a polynomial in this paper, one has that
	$$ L_s^{n,\rho_{n+1},\beta_{n+1},\gamma_{n+1},\delta_{n+1}}(Q) = L_s^{n+1}(Q), $$
	where
	\begin{align*}
		L_s^{n+1}(Q) &:= \frac{1}{2} \sum_{\substack{1\leq i,j\leq d \\ J,J'\in J_n\\ \text{such that } J_s=J_s'}} \Big(\partial_{j,J}^2\left( \partial_{i}^1 D_{i} Q\right)\left(X_{s,1}\right) \boxtimes \partial_{j,J}^1 \left( \partial_{i}^1 D_{i} Q\right)\left(\widetilde{X}_{s,1}\right) \Big) \\
		&\quad\quad\quad\quad\quad\quad\quad\quad \boxtimes \Big(\partial_{j,J'}^2\left( \partial_{i}^2 D_{i} Q\right)\left(\widetilde{X}_{s,2}\right) \boxtimes \partial_{j,J'}^1 \left( \partial_{i}^2 D_{i} Q\right)\left(X_{s,2}\right) \Big).
	\end{align*} 
	Besides, one has the following equalities, for $Q,R,A$ monomials,
	$$ D_iQ = \sum_{I\in J_n}\sum_{Q\sim RX_{i,I}} R, $$
	$$ \partial_iR = \sum_{I'\in J_n}\sum_{R= A X_{i,I'} B} A\otimes B, $$
	$$ \partial_{j,J}A = \sum_{A= U X_{j,J} V} U\otimes V, $$
	$$ (A\boxtimes B)\boxtimes (C\boxtimes D) = BADC.$$
	Consequently, we have that
	\begin{align*}
		L_s^{n+1}(Q) &:= \frac{1}{2} \sum_{\substack{1\leq i,j\leq d \\ I,I',J,J'\in J_n}} \sum_{\substack{W\sim AX_{j,J'}BX_{i,I'}CX_{j,J}DX_{i,I} \\ \text{where } J\underset{\mathclap{{}^{s}}}{\sim}J'}} B\left(X_{s,1}\right) A\left(\widetilde{X}_{s,1}\right) D\left(\widetilde{X}_{s,2}\right) C\left(X_{s,2}\right).
	\end{align*} 
	Thus with $L^{T_{i}}$ as in Equation \eqref{xdffgyui},
	\begin{align*}
		&\E\left[ \ts_N\Big( Q(X^N,Z^N) \Big) \right] =\sum_{i\geq 0} \frac{1}{N^{2p}} \int_{A_p } e^{-t_1-\cdots-t_{2p}} \tau\Big( \left(L^{{T}_p} \dots L^{{T}_1} \right)(Q) (x^{T_p},Z^N) \Big)\ dt,
	\end{align*}
	Finally since $L^{T_{i}}$ differentiates a polynomial four times, if $p>\deg(Q)/4$, then $\left(L^{{T}_p} \dots L^{{T}_1} \right)(Q) =0$.
	
\end{proof}

\begin{lemma}
	
	\label{poiuytrew}
	There exist a constant $c_{s_1,\dots,s_p}\leq (4p-1)!$, $J^l$ a $4p$-tuple of elements of $J_{p}$ as well as permutations $\pi_l^{s_1,\dots,s_p},\sigma_l^{s_1,\dots,s_p}\in \mathbb{S}_{4p}$ which we will denote by $\pi_l,\sigma_l$ for the sake of brevity, such that with convention $k_{4p+1}=k_1$, for any monomial $A=M_1(Z)X_{i_1}M_2(Z)\dots M_k(Z)X_{i_k}M_{k+1}(Z)$, where $M_j$ is a monomial in $Z$,
	\begin{align}
		\label{asdfghjkl}
		\left(L^{p}_{s_p} \dots L^{1}_{s_1}\right)(A) &= \frac{1}{2^p} \sum_{ (k_1,\dots,k_{4p}) \in ((1,k))} \sum_{l=1}^{c_{s_1,\dots,s_p}}\ \delta_{i_{k_{\pi_l(1)}} = i_{k_{\pi_l(2)}}} \dots \delta_{i_{k_{\pi_l(4p-1)}} = i_{k_{\pi_l(4p)}}}\\
		&\quad\quad\quad\quad\quad\quad\quad\quad\quad\quad\quad\quad A_{(k_{\sigma_l(1)},k_{\sigma_l(1)+1})}\left(X_{J^l_1}\right)\cdots A_{(k_{\sigma_l(4p)},k_{\sigma_l(4p)+1})}\left(X_{J_{4p}^l}\right), \nonumber
	\end{align}
	where $A_{(u,v)} = M_u(Z)X_{i_{u+1}}\dots X_{i_{v+1}} M_v(Z)$ and $X_I=(X_{i,I})_{1\leq i\leq d} \cup (Z_j)_{1\leq j\leq q}$. 
	
	Thus one can find $\widehat{c}_{s_1,\dots,s_p}\leq (4p)!$, $\widehat{J}^l$ a $4p$-tuple of elements of $J_{p}$ as well as permutations $\widehat{\pi}_l^{s_1,\dots,s_p},$ $\widehat{\sigma}_l^{s_1,\dots,s_p}\in \mathbb{S}_{4p}$ which we will denote by $\widehat{\pi}_l,\widehat{\sigma}_l$ for the sake of brevity, such that for any polynomial $P$,
	\begin{align}
		\label{asdfghjkl2}
		\left(L^{p}_{s_p} \dots L^{1}_{s_1}\right)(P) &= \frac{1}{2^p} \sum_{1\leq i_1,\dots,i_{4p} \leq d} \sum_{l=1}^{\widehat{c}_{s_1,\dots,s_p}}\ \delta_{i_{\widehat{\pi}_l(1)} = i_{\widehat{\pi}_l(2)}} \dots \delta_{i_{\widehat{\pi}_l(4p-1)} = i_{\widehat{\pi}_l(4p)}} \\
		&\quad\quad\quad\quad\quad\quad\quad\quad\quad\quad\quad\quad\quad\quad\quad\quad\quad\quad\quad \ev_l\circ\widehat{\sigma}_l\circ m\circ\left(\partial_{i_1}\otimes \cdots \otimes \partial_{i_{4p}}\right)(P), \nonumber
	\end{align}
	where
	$$ m(A_1\otimes\cdots\otimes A_{4p+1}) = A_2\otimes \dots\otimes A_{4p}\otimes A_{4p+1}A_1,$$
	$$ \widehat{\sigma}_l(A_1\otimes\cdots\otimes A_{4p}) = A_{\sigma_l(1)}\otimes\cdots\otimes A_{\sigma_l(4p)},$$
	$$ \ev_l(A_1\otimes\cdots\otimes A_{4p}) = A_1\left(X_{\widehat{J}^l_1}\right)\cdots A_{4p}\left(X_{\widehat{J}_{4p}^l}\right).$$
\end{lemma}

\begin{proof}
	
	Thanks to Lemma \ref{3apparition} and notably Equation \eqref{xfgukokgc}, if 
	$$W=M_1(Z)X_{i_1,I_1}M_2(Z)\dots M_k(Z) X_{i_k,I_k}M_{k+1}(Z)$$
	has degree $k$, then
	\begin{align*}
		L_s^{n}(W) := \frac{1}{2} \sum_{(k_1,k_2,k_3,k_4)\in ((1,k)) } \delta_{I_{k_1}\underset{\mathclap{{}^{s}}}{\sim}I_{k_3}} & \delta_{i_{k_1}=i_{k_3}} \delta_{i_{k_2}=i_{k_4}} \\
		&W_{(k_2,k_3)}\left(X_{s,1}\right) W_{(k_1,k_2)}(\widetilde{X}_{s,1}) W_{(k_4,k_1)}(\widetilde{X}_{s,2}) W_{(k_3,k_4)}\left(X_{s,2}\right),
	\end{align*}
	Consequently, let us proceed by induction. We assume that Equation \eqref{asdfghjkl} is true for a given $p$, then with 
	$$ Q= A_{(k_{\sigma_l(1)},k_{\sigma_l(1)+1})}\left(X_{J^l_1}\right)\cdots A_{(k_{\sigma_l(4p)},k_{\sigma_l(4p)+1})}\left(X_{J^l_{4p}}\right),$$
	one has
	\begin{align*}
		&\sum_{(k_1,\dots,k_{4p})\in ((1,k))} \delta_{i_{k_{\pi_l(1)}} = i_{k_{\pi_l(2)}}} \dots \delta_{i_{k_{\pi_l(4p-1)}} = i_{k_{\pi_l(4p)}}} L_{s_{p+1}}^{p+1}\left(Q\right) \\
		&= \frac{1}{2} \sum_{(k_1,\dots,k_{4p})\in ((1,k))} \delta_{i_{k_{\pi_l(1)}} = i_{k_{\pi_l(2)}}} \dots \delta_{i_{k_{\pi_l(4p-1)}} = i_{k_{\pi_l(4p)}}} \sum_{(m_1,m_2,m_3,m_4)\in ((1,k))\setminus \{k_1,\dots,k_{4p}\} } \\
		&\quad\quad\quad\quad\quad\quad\quad \delta_{I_{m_1}\underset{\mathclap{{}^{s}}}{\sim}I_{m_3}} \delta_{i_{m_1}=i_{m_3}} \delta_{i_{m_2}=i_{m_4}} Q_{(m_2,m_3)}\left(X_{s,1}\right) Q_{(m_1,m_2)}(\widetilde{X}_{s,1}) Q_{(m_4,m_1)}(\widetilde{X}_{s,2}) Q_{(m_3,m_4)}\left(X_{s,2}\right) \\
		&= \frac{1}{2} \sum_{[l_1,l_2,l_3,l_4]\in ((1,4p))} \sum_{(k_1,\dots,k_{4p},k_{r_1},k_{r_2},k_{r_3},k_{r_4})\in ((1,k))} \delta_{J^l_{l_1}\underset{\mathclap{{}^{s}}}{\sim}J^l_{l_3}} \\
		&\quad\quad\quad\quad\quad\quad\quad\quad\quad\quad\quad\quad\quad\quad\quad\quad\quad\quad\quad\quad\quad \times \delta_{i_{k_{r_1}}=i_{k_{r_3}}} \delta_{i_{k_{r_2}}=i_{k_{r_4}}} \delta_{i_{k_{\pi_l(1)}} = i_{k_{\pi_l(2)}}} \dots \delta_{i_{k_{\pi_l(4p-1)}} = i_{k_{\pi_l(4p)}}} \\
		&\quad\quad\quad\quad\quad\quad\quad\quad\quad\quad\quad\quad\quad \times W_{(k_{\widetilde{\sigma}_l(l_2)},k_{\widetilde{\sigma}_l(l_2)^+})}\left(X_{s,1}\right) W_{(k_{\widetilde{\sigma}_l(l_2^+)},k_{\widetilde{\sigma}_l(l_2^+)^+})}\left(X_{s,1}\right) \dots 	W_{(k_{\widetilde{\sigma}_l(l_3^-)},k_{\widetilde{\sigma}_l(l_3^-)^+})}\left(X_{s,1}\right) \\
		&\quad\quad\quad\quad\quad\quad\quad\quad\quad\quad\quad\quad\quad \times W_{(k_{\widetilde{\sigma}_l(l_1)},k_{\widetilde{\sigma}_l(l_1)^+})}(\widetilde{X}_{s,1}) W_{(k_{\widetilde{\sigma}_l(l_1^+)},k_{\widetilde{\sigma}_l(l_1^+)^+})}\left(X_{s,1}\right)\dots W_{(k_{\widetilde{\sigma}_l(l_2^-)},k_{\widetilde{\sigma}_l(l_2^-)^+})}(\widetilde{X}_{s,1}) \\
		&\quad\quad\quad\quad\quad\quad\quad\quad\quad\quad\quad\quad\quad \times W_{(k_{\widetilde{\sigma}_l(l_4)},k_{\widetilde{\sigma}_l(l_4)^+})}(\widetilde{X}_{s,2}) W_{(k_{\widetilde{\sigma}_l(l_4^+)},k_{\widetilde{\sigma}_l(l_4^+)^+})}\left(X_{s,1}\right)\dots W_{(k_{\widetilde{\sigma}_l(l_1^-)},k_{\widetilde{\sigma}_l(l_1^-)^+})}(\widetilde{X}_{s,2}) \\
		&\quad\quad\quad\quad\quad\quad\quad\quad\quad\quad\quad\quad \times W_{(k_{\widetilde{\sigma}_l(l_3)},k_{\widetilde{\sigma}_l(l_3)^+})}\left(X_{s,2}\right) W_{(k_{\widetilde{\sigma}_l(l_3^+)},k_{\widetilde{\sigma}_l(l_3^+)^+})}\left(X_{s,1}\right)\dots W_{(k_{\widetilde{\sigma}_l(l_4^-)},k_{\widetilde{\sigma}_l(l_4^-)^+})}\left(X_{s,2}\right),
	\end{align*}
	where 
	\begin{itemize}
		\item $r_1,r_2,r_3,r_4$ are defined with the following rules, if the point $l_i$ follows $j$ then $r_i$ follows $\sigma_l(j)$, if several of the $l_i$ follow the same point $j$, then the associated points $r_i$ follows $\sigma_l(j)$ in the order of the circle set $((1,4))$.  
		\item $\widetilde{\sigma}_l$ is defined by $\widetilde{\sigma}_l(s) = \sigma_l(s)$ for $s\in ((1,4p))$ and $\widetilde{\sigma}_l(l_i) = r_i$.
	\end{itemize}
	Note in particular that with those definitions, one always has $\widetilde{\sigma}_l(l_i^-)^+=\widetilde{\sigma}_l(l_i)=r_i$. Besides the cardinality of $\{[l_1,l_2,l_3,l_4]\in ((1,4p))\}$ is bounded by $4p (4p+1) (4p+2) (4p+3)$ (with the exception of $p=0$ where this number is replaced by $6$). Hence by only keeping indices $l_1,l_2,l_3,l_4$ such that $J^l_{l_1}\underset{\mathclap{{}^{s}}}{\sim}J^l_{l_3}$, one can find $c_l$ smaller than this number, permutations $\widehat{\sigma}_r$ and $\widehat{\pi}_r$, as well as $J^r$ a $4p+4$-tuple of elements of $J_{p+1}$, such that by renaming $((1,4p))\cup\{r_1,\dots,r_4\}$ into $((1,4p+4))$ while keeping the order of the circle set, one has that
	\begin{align*}
		&\sum_{(k_1,\dots,k_{4p})\in ((1,k))} \delta_{i_{k_{\pi_l(1)}} = i_{k_{\pi_l(2)}}} \dots \delta_{i_{k_{\pi_l(4p-1)}} = i_{k_{\pi_l(4p)}}} L_{s_{p+1}}^{p+1}\left(Q\right) \\
		&= \frac{1}{2} \sum_{ (k_1,\dots,k_{4(p+1)}) \in ((1,k))} \sum_{r=1}^{c_l}\ \delta_{i_{k_{\widehat{\pi}_l(1)}} = i_{k_{\widehat{\pi}_l(2)}}} \dots \delta_{i_{k_{\widehat{\pi}_l(4p+3)}} = i_{k_{\widehat{\pi}_l(4p+4)}}}\\
		&\quad\quad\quad\quad\quad\quad\quad\quad\quad\quad\quad\quad A_{(k_{\widehat{\sigma}_l(1)},k_{\widehat{\sigma}_l(1)+1})}\left(X_{J^r_1}\right)\cdots A_{(k_{\widehat{\sigma}_l(4p+4)},k_{\widehat{\sigma}_l(4p+4)+1})}\left(X_{J_{4p+4}^r}\right) \nonumber
	\end{align*}
	Hence the conclusion by summing over $l$. 
	
	Let us now deduce Equation \eqref{asdfghjkl2} out of Equation \eqref{asdfghjkl}. First and foremost, with $\nu_r\in \mathbb{S}_{4p}$ the permutation such that $\nu_r(i)=i-r+1$ fo $i\geq r$ and $\nu_r(i)=i+4p-r+1$ else, one has that
	\begin{align*}
		&\sum_{ (k_1,\dots,k_{4p}) \in ((1,k))} \delta_{i_{k_{\pi_l(1)}} = i_{k_{\pi_l(2)}}} \dots \delta_{i_{k_{\pi_l(4p-1)}} = i_{k_{\pi_l(4p)}}} A_{(k_{\sigma_l(1)},k_{\sigma_l(1)+1})}\left(X_{J^l_1}\right)\cdots A_{(k_{\sigma_l(4p)},k_{\sigma_l(4p)+1})}\left(X_{J_{4p}^l}\right) \\
		&= \sum_{ (k_1,\dots,k_{4p}) \in ((1,k))} \delta_{i_{k_{\pi_l(1)}} = i_{k_{\pi_l(2)}}} \dots \delta_{i_{k_{\pi_l(4p-1)}} = i_{k_{\pi_l(4p)}}} \ev_l\circ\sigma_l\left( A_{(k_1,k_2)}\otimes\cdots\otimes A_{(k_{4p},k_{1})}\right) \\
		&= \sum_{r=1}^{4p}\ \sum_{ 1\leq k_r< \dots<k_{4p}<k_1<\dots<k_{r-1}\leq k}\delta_{i_{k_{\pi_l(1)}} = i_{k_{\pi_l(2)}}} \dots \delta_{i_{k_{\pi_l(4p-1)}} = i_{k_{\pi_l(4p)}}} \ev_l\circ\sigma_l\left( A_{(k_1,k_2)}\otimes\cdots\otimes A_{(k_{4p},k_{1})}\right) \\
		&= \sum_{r=1}^{4p}\ \sum_{ 1\leq k_r< \dots<k_{4p}<k_1<\dots<k_{r-1}\leq k} \delta_{i_{k_{\pi_l(1)}} = i_{k_{\pi_l(2)}}} \dots \delta_{i_{k_{\pi_l(4p-1)}} = i_{k_{\pi_l(4p)}}} \\ 
		&\quad\quad\quad\quad\quad\quad\quad\quad\quad\quad\quad\quad\quad\quad\quad\quad\quad\quad \ev_{l,r}\circ\sigma_l\circ\nu_r\left( A_{(k_r,k_{r+1})}\otimes\cdots\otimes A_{(k_{4p},k_{1})} \otimes\cdots\otimes A_{(k_{r-1},k_{r})} \right) \\
		&= \sum_{r=1}^{4p}\ \sum_{1\leq j_1,\dots,j_{4p} \leq d} \delta_{j_{\pi_l(1)} = j_{\pi_l(2)}} \dots \delta_{j_{\pi_l(4p-1)} = j_{\pi_l(4p)}} \\ 
		&\quad\quad\quad\quad\quad\quad\quad\quad\quad\quad\quad\quad\quad\quad\quad\quad\quad\quad \ev_{l,r}\circ\sigma_l\circ\nu_r\circ m\circ\left( \partial_{j_r}\otimes \cdots \otimes \partial_{j_{4p}}\otimes\partial_{j_1} \otimes \cdots \otimes \partial_{j_{r-1}} \right)(A) \\
		&= \sum_{r=1}^{4p}\ \sum_{1\leq j_1,\dots,j_{4p} \leq d}\ \delta_{j_{\pi_{l,r}(1)} = j_{\pi_{l,r}(2)}} \dots \delta_{j_{\pi_{l,r}(4p-1)} = j_{\pi_{l,r}(4p)}} \ev_{l,r}\circ\sigma_l\circ\nu_r\circ m\circ\left( \partial_{j_1}\otimes \cdots \otimes \partial_{j_{4p}} \right)(A),
	\end{align*}
	where $\ev_{l,r}$ is such that $\ev_{l,r}\circ\sigma_l\circ\nu_r = \ev_l\circ\sigma_l$, and $\pi_{l,r}= \nu_r\circ\pi_l$. Thus after summing over $l$ and reordering the terms in the sum we get Equation \eqref{asdfghjkl2} for any monomials. But since both the right and left hand side of this equation are linear map applied in $A$, this equation remains true for any polynomial.	
	
\end{proof}

\subsection{A decomposition of product of polynomials in free semicircular variables}

The following Theorem is the cornerstone of the different results of this paper. It allows us to do the following, given a product of polynomials in free semicircular variables, then when taking its trace, heuristically, one can assume that those polynomials are commutative at the cost of differentiating them.

\begin{prop}
	
	\label{edgtn}
	
	Let $A_1,\dots,A_n\in \C\langle X_1,\dots,X_d\rangle$ monomials, $x^1,\dots,x^n$ $d$-tuples of semicircular variables. We assume that for all $u,v\in [1,d]$, $i,j\in [1,n]$, 
	$$ \tau(x^i_u x_v^j) = \delta_{u=v} \kappa_{\{i,j\}},$$
	then with $\Delta_{\{i,j\}} := \sum_{l\geq 1} \kappa_{\{i,j\}}^l P_l$, where $P_l$ is the orthogonal projection on the subspace $\Hi^{\otimes l}$ of the Fock space, as well as $k_{i,i}=0$ and $k_{i,i^-}=\deg A_i$, one has that
	\begin{align}
		\label{pojfesv}
		&\tau\left(A_{1}\left(x^{1}\right)\dots A_{n}\left(x^{n}\right)\right) \\ \nonumber
		&= \sum_{K\in \mathcal{K}_n}\ \sum_{\substack{\forall i\in [1,n], j\in C_i^{**}(K),\ 0 < k_{i,i^+} < \dots < k_{i,{i^-}^-}} < \deg A_i}\ \prod_{i\in\widehat{K}} \tau\left(A_i(x)\right) \nonumber\\
		&\quad\quad\quad\quad\quad\quad\quad\quad\times \prod_{\substack{\{i,j\}\in K, \\ \text{such that } j=i^- \text{ in } C_i(K),\ i=j^- \text{ in } C_j(K) }} \tau\left( (A_i)_{(k_{i,j^-},k_{i,j}]} \Delta_{\{i,j\}} (A_j)_{(k_{j,i^-},k_{j,i}]} \right) \nonumber\\
		&\quad\quad\quad\quad\quad\quad\quad\quad\times \prod_{\substack{\{i,j\}\in K, \\ \text{such that } j\neq i^-  \text{ in } C_i(K),\ i=j^- \text{ in } C_j(K) }} \tau\left( (A_i)_{(k_{i,j^-},k_{i,j})} l^*_{z_{k_{i,j}}} \Delta_{\{i,j\}} (A_j)_{(k_{j,i^-},k_{j,i}]} \right) \nonumber\\
		&\quad\quad\quad\quad\quad\quad\quad\quad\times \prod_{\substack{\{i,j\}\in K, \\ \text{such that } j=i^- \text{ in } C_i(K),\ i\neq j^- \text{ in } C_j(K) }} \tau\left( (A_i)_{(k_{i,j^-},k_{i,j}]} \Delta_{\{i,j\}} r_{z_{k_{j,i}}} (A_j)_{(k_{j,i^-},k_{j,i})} \right) \nonumber\\
		&\quad\quad\quad\quad\quad\quad\quad\quad\times \prod_{\substack{\{i,j\}\in K, \\ \text{such that } j\neq i^- \text{ in } C_i(K),\ i\neq j^- \text{ in } C_j(K) }} \tau\left( (A_i)_{(k_{i,j^-},k_{i,j})} l^*_{z_{k_{i,j}}} \Delta_{\{i,j\}} r_{z_{k_{j,i}}} (A_j)_{(k_{j,i^-},k_{j,i})} \right), \nonumber
	\end{align}
	where given a monomial $A=X_{z_1}\dots X_{z_n}$, we set $A_{(u,v)} = X_{z_u}\dots X_{z_{v}}$. Besides, one also has that given another monomial $B=X_{z_{n+1}}\dots X_{z_m}$, with $l_z,r_z$ left and right multiplication operators,
	\begin{equation}
		\label{idjnvcdskmvcos0}
		\tau\left(A(x) P_l B(x)\right) = \tau\left(B(x) P_l A(x)\right),
	\end{equation}
	\begin{equation}
		\label{idjnvcdskmvcos}
		\tau\left(A_{[1,n)}(x) l_{z_n}^*P_l B(x)\right) = \tau\left( B(x) P_l r_{z_n} A_{[1,n)}(x)\right),
	\end{equation}
	\begin{equation}
		\label{idjnvcdskmvcos2}
		\tau\left(A(x) P_l r_{z_m} B_{[n+1,m)}(x)\right) = \tau\left( B_{[n+1,m)}(x) l_{z_m}^*P_l A(x)\right),
	\end{equation}
	\begin{equation}
		\label{idjnvcdskmvcos3}
		\tau\left(A_{[1,n)}(x) l_{z_n}^*P_l r_{z_m} B_{[n+1,m)}(x)\right) = \tau\left( B_{[n+1,m)}(x) l_{z_m}^* P_l r_{z_n} A_{[1,n)}(x)\right).
	\end{equation}
\end{prop}

\begin{proof}
	
	Thanks to Definition 8.15 of \cite{nica_speicher_2006}, one has that for $W(x)=x_1\dots x_k$ a monomial in free semicircular variable,
	\begin{equation}
		\label{osicdlskmc}
		\tau(W(x)) = \sum_{\pi\in NCP(k)}\prod_{(p,q)\in \pi} c_{p,q},
	\end{equation}
	where $NCP(k)$ is the set of non-crossing pair partition of of $[1,k]$, and $c_{p,q}=\tau(x_px_q)$. Thus we set $k=\sum_l \deg A_l$. Consequently,
	\begin{align*}
		&\tau\left(A_{1}\left(x^{1}\right)\dots A_{n}\left(x^{n}\right)\right) \\
		&= \sum_{K\in \mathcal{K}_n}\ \sum_{\forall  \{i,j\}\in K,\ l_{\{i,j\}}\geq 1} \sum_{\substack{\pi\in NCP(k) \text{ such that} \\ \#\{(p,q)\in\pi\ |\ p\in A_{i},\ q\in A_{j}\} = l_{\{i,j\}}\delta_{(i,j)\in K}}} \prod_{(p,q)\in \pi} c_{p,q} \\
		&= \sum_{K\in \mathcal{K}_n}\ \sum_{\forall  \{i,j\}\in K,\ l_{\{i,j\}}\geq 1}\ \prod_{\{i,j\}\in K} \kappa_{\{i,j\}}^{l_{\{i,j\}}} \sum_{\substack{\pi\in NCP(k) \\ \#\{(p,q)\in\pi\ |\ p\in A_{i},\ q\in A_{j}\} = l_{\{i,j\}}\delta_{(i,j)\in K}}} \prod_{(p,q)\in \pi} \delta_{z_p=z_q} 
	\end{align*}
	We say that $\pi$ is compatible with a configuration $K$ if for all $\{i,j\}\in K$,
	$$ \#\{(p,q)\in\pi\ |\ p\in A_{i},\ q\in A_{j}\} \geq 1,$$
	and for all $\{i,j\}\notin K$ with $i\neq j$,
	$$ \#\{(p,q)\in\pi\ |\ p\in A_{i},\ q\in A_{j}\} =0.$$
	We then proceed to define a quantity $k_{i,j}(\pi)$ for $j\in C_i^{**}(K)$,
	$$ k_{i,j}(\pi) = \sup \{ p\in A\ |\ (p,q)\in \pi, p\in A_i, q\in A_j \}- \sum_{0\leq s< i} \deg A_s .$$
	Hence
	\begin{align*}
		&\tau\left(A_{1}\left(x^{1}\right)\dots A_{n}\left(x^{n}\right)\right) = \sum_{K\in \mathcal{K}_n}\ \sum_{\substack{\forall  \{i,j\}\in K,\ l_{\{i,j\}}\geq 1 \\ \forall i\in [1,n], j\in C_i^{**}(K),\ 0 < k_{i,i^+} < \dots < k_{i,{i^-}^-} < \deg A_i }}\\
		&\quad\quad\quad\quad\quad\quad\quad\quad\quad\quad\quad\quad\quad\quad\quad\quad\quad\quad\quad \prod_{\{i,j\}\in K} \kappa_{\{i,j\}}^{l_{\{i,j\}}} \sum_{\substack{\pi\in NCP(k) \\ \#\{(p,q)\in\pi\ |\ p\in A_{i},\ q\in A_{j}\} = l_{i,j} \delta_{(i,j)\in K} \\ \forall i\in [1,n], j\in C_i(K),\ k_{i,j}(\pi) = k_{i,j} }} \prod_{(p,q)\in \pi} \delta_{z_p=z_q} 
	\end{align*}
	Next, if $Q=ABCDE$ is a monomial, we set $NCP(BD)$ to be the set of non-crossing pair partitions of the set $[\deg A +1, \deg(AB)]\cup [\deg(ABC) +1,\deg(ABCD)]$. Thus for a given profile $K$, with $k_{i,i}=0$ and $k_{i,i^-}=\deg A_i$,
	\begin{align*}
		&\left\{ \pi\in NCP(k)\ \middle|\ \#\{(p,q)\in\pi\ |\ p\in A_{i},\ q\in A_{j}\} = l_{i,j}\delta_{(i,j)\in K},\ \forall i\in [1,n], j\in C_i(K),\ k_{i,j}(\pi) = k_{i,j}  \right\} \\
		&= \Bigg\{ \bigcup_{\{i,j\}\in K} \pi_{i,j} \bigcup_{i\in \widehat{K}} \pi_{i}\ \Bigg|\ \pi_i\in NCP\left( A_i\right),\ \pi_{i,j}\in NCP\left( (A_i)_{(k_{i,j^-},k_{i,j}]} (A_j)_{(k_{j,i^-},k_{j,i}]}\right),\\
		&\quad\quad\quad\quad\quad\quad\quad\quad\quad\quad\quad \#\{(p,q)\in\pi_{i,j}\ |\ p\in (A_i)_{(k_{i,j^-},k_{i,j}]},\ q\in  (A_j)_{(k_{j,i^-},k_{j,i}]}\} = l_{i,j}, \\
		&\quad\quad\quad\quad\quad\quad\quad\quad\quad\quad\quad \text{if } j\neq i^- \text{ in } C_i(K), \text{ and } ((A_i)_{k_{i,j}},q)\in \pi_{i,j}, \text{ then } q\in A_j, \\
		&\quad\quad\quad\quad\quad\quad\quad\quad\quad\quad\quad \text{if } i\neq j^- \text{ in } C_j(K), \text{ and } ((A_j)_{k_{j,i}},q)\in \pi_{i,j}, \text{ then } q\in A_i \Bigg\}.
	\end{align*}
	Thus one has
	\begin{align}
		\label{smcsmcscsc}
		&\tau\left(A_{1}\left(x^{1}\right)\dots A_{n}\left(x^{n}\right)\right) \\ \nonumber
		&= \sum_{K\in \mathcal{K}_n}\ \sum_{\substack{\forall  \{i,j\}\in K,\ l_{\{i,j\}}\geq 1 \\ \forall i\in [1,n], j\in C_i^{**}(K),\ 0 < k_{i,i^+} < \dots < k_{i,{i^-}^-} < \deg A_i }}\ \prod_{\{i,j\}\in K} \kappa_{\{i,j\}}^{l_{\{i,j\}}} \times \prod_{i\in\widehat{K}} \sum_{\pi_i\in NCP\left( A_i\right)}  \prod_{(p,q)\in \pi_i} \delta_{z_p=z_q} \\ \nonumber
		&\quad\quad\quad\quad\quad\quad\quad \times \prod_{\{i,j\}\in K} \sum_{\substack{\pi_{i,j}\in NCP\left( (A_i)_{(k_{i,j^-},k_{i,j}]} (A_j)_{(k_{j,i^-},k_{j,i}]}\right) \\ \#\{(p,q)\in\pi_{i,j}\ |\ p\in (A_i)_{(k_{i,j^-},k_{i,j}]},\ q\in  (A_j)_{(k_{j,i^-},k_{j,i}]}\} = l_{i,j} \\ \text{if } j\neq i^- \text{ in } C_i(K), \text{ and } ((A_i)_{k_{i,j}},q)\in \pi_{i,j}, \text{ then } q\in A_j \\ \text{if } i\neq j^- \text{ in } C_j(K), \text{ and } ((A_j)_{k_{j,i}},q)\in \pi_{i,j}, \text{ then } q\in A_i }} \prod_{(p,q)\in \pi_{i,j}} \delta_{z_p=z_q} 
	\end{align}
	But thanks to Equation \eqref{osicdlskmc}, one has that for a given $d$-tuple $x$ of free semicircular variables,
	\begin{equation}
		\label{wdocsomvcsdo}
		\sum_{\pi\in NCP\left( A\right)}  \prod_{(p,q)\in \pi} \delta_{z_p=z_q} = \tau\left(A(x)\right). 
	\end{equation}
	Besides
	\begin{equation}
		\label{soicdsodcscd}
		\sum_{\substack{\pi\in NCP\left( AB\right) \\ \#\{(p,q)\in\pi\ |\ p\in A,\ q\in B\} = l  }} \prod_{(p,q)\in \pi} \delta_{z_p=z_q} = \tau\left(A(x) P_l B(x)\right),
	\end{equation}
	where $P_l$ is the orthogonal projection on the subspace $\Hi^{\otimes l}$ of the Fock space.
	Indeed if $A=X_{z_{1}}\dots X_{z_n}$ and $B=X_{z_{n+1}}\dots X_{z_m}$, then
	\begin{align*}
		 \tau\left(A(x) P_l B(x)\right) &= \sum_{\varepsilon\in \{1,*\}^{m}}  \langle l_{z_1}^{\varepsilon_1}\dots l_{z_n}^{\varepsilon_n} P_l l_{z_{n+1}}^{\varepsilon_{n+1}}\dots l_{z_{m}}^{\varepsilon_{m}} \Omega, \Omega\rangle.
	\end{align*}
	Let us assume that $\langle l_{z_1}^{\varepsilon_1}\dots l_{z_n}^{\varepsilon_n} P_l l_{z_{n+1}}^{\varepsilon_{n+1}}\dots l_{z_{m}}^{\varepsilon_{m}} \Omega, \Omega\rangle\neq 0$. Since for any $i$, $l_i^*\Omega =0$, necessarily 
	$$ \langle l_{z_1}^{\varepsilon_1}\dots l_{z_n}^{\varepsilon_n} P_l l_{z_{n+1}}^{\varepsilon_{n+1}}\dots l_{z_{m}}^{\varepsilon_{m}} \Omega, \Omega\rangle = \langle l_{z_1}^{\varepsilon_1}\dots l_{z_n}^{\varepsilon_n} P_l l_{z_{n+1}}^{\varepsilon_{n+1}}\dots l_{z_{m-1}}^{\varepsilon_{m-1}} e_{z_{m}} , \Omega\rangle.$$
	Let now $j$ be the largest integer such that $l_{z_j}^{\varepsilon_j}\dots l_{z_{m-1}}^{\varepsilon_{m-1}} e_{z_{m}} = \Omega$. If such an integer does not exist, then in particular $l_{z_1}^{\varepsilon_1}\dots l_{z_n}^{\varepsilon_n} P_l l_{z_{n+1}}^{\varepsilon_{n+1}}\dots l_{z_{m-1}}^{\varepsilon_{m-1}} e_{z_{m}} \neq \Omega$, and more precisely, it is orthogonal to $\Omega$, which would imply that $\langle l_{z_1}^{\varepsilon_1}\dots l_{z_n}^{\varepsilon_n} P_l l_{z_{n+1}}^{\varepsilon_{n+1}}\dots l_{z_{m}}^{\varepsilon_{m}} \Omega, \Omega\rangle= 0$. Consequently,
	\begin{itemize}
		\item if $j\leq n$, 
		$$ \langle l_{z_1}^{\varepsilon_1}\dots l_{z_n}^{\varepsilon_n} P_l l_{z_{n+1}}^{\varepsilon_{n+1}}\dots l_{z_{m}}^{\varepsilon_{m}} \Omega, \Omega\rangle = \langle l_{z_{1}}^{\varepsilon_{1}}\dots l_{z_{j-1}}^{\varepsilon_{j-1}} \Omega, \Omega\rangle \langle l_{z_{j+1}}^{\varepsilon_{j+1}}\dots l_{z_n}^{\varepsilon_n} P_{l-1} l_{z_{n+1}}^{\varepsilon_{n+1}}\dots l_{z_{m-1}}^{\varepsilon_{m-1}} \Omega, \Omega\rangle \delta_{z_j = z_{m}}, $$
		\item if $j\geq n+1$, 
		$$ \langle l_{z_1}^{\varepsilon_1}\dots l_{z_n}^{\varepsilon_n} P_l l_{z_{n+1}}^{\varepsilon_{n+1}}\dots l_{z_{m}}^{\varepsilon_{m}} \Omega, \Omega\rangle = \langle l_{z_1}^{\varepsilon_1}\dots l_{z_n}^{\varepsilon_n} P_l l_{z_{n+1}}^{\varepsilon_{n+1}}\dots l_{z_{j-1}}^{\varepsilon_{j-1}} \Omega, \Omega\rangle \langle l_{z_{j+1}}^{\varepsilon_{j+1}}\dots l_{z_{m-1}}^{\varepsilon_{m-1}} \Omega, \Omega\rangle \delta_{z_j = z_{m}}. $$
	\end{itemize}
	Consequently, one has that
	\begin{align*}
		\tau\left(A(x) P_l B(x)\right) &= \sum_{j=1}^n \delta_{z_j = z_{m}} \sum_{\varepsilon\in \{1,*\}^{j-1}}  \langle l_{z_{1}}^{\varepsilon_{1}}\dots l_{z_{j-1}}^{\varepsilon_{j-1}} \Omega, \Omega\rangle \sum_{\varepsilon\in \{1,*\}^{m-j-1}} \langle l_{z_{j+1}}^{\varepsilon_{j+1}}\dots l_{z_n}^{\varepsilon_n} P_{l-1} l_{z_{n+1}}^{\varepsilon_{n+1}}\dots l_{z_{m-1}}^{\varepsilon_{m-1}} \Omega, \Omega\rangle \\
		&\quad + \sum_{j=n+1}^{m-1} \delta_{z_j = z_{m}} \sum_{\varepsilon\in \{1,*\}^{j-1}} \langle l_{z_1}^{\varepsilon_1}\dots l_{z_n}^{\varepsilon_n} P_l l_{z_{n+1}}^{\varepsilon_{n+1}}\dots l_{z_{j-1}}^{\varepsilon_{j-1}} \Omega, \Omega\rangle \sum_{\varepsilon\in \{1,*\}^{m-j-1}} \langle l_{z_{j+1}}^{\varepsilon_{j+1}}\dots l_{z_{m-1}}^{\varepsilon_{m-1}} \Omega, \Omega\rangle \\
		&= \sum_{j=1}^n \delta_{z_j = z_{m}} \tau\left(A_{[1,j)}(x)\right) \tau\left(A_{(j,n]}(x) P_{l-1} B_{[n+1,m)}(x)\right) \\
		&\quad + \sum_{j=n+1}^{m-1} \delta_{z_j = z_{m}} \tau\left(A(x)P_l B_{[n+1,j)}(x)\right) \tau\left(B_{(j,m)}(x)\right).	
	\end{align*}
	And the formula above allows us to prove Equation \eqref{soicdsodcscd} by induction. Besides by symmetry of the left-hand side of \eqref{soicdsodcscd}, one has that
	\begin{equation}
		\tau\left(A(x) P_l B(x)\right) = \tau\left(B(x) P_l A(x)\right).
	\end{equation}
	Similarly, with $r_i$ the right multiplication operator, one also has
	\begin{equation*}
		\sum_{\substack{\pi\in NCP\left( AB\right), \\ \#\{(p,q)\in\pi\ |\ p\in A,\ q\in  B\} = l, \\ \text{let } (A_{n},q)\in \pi, \text{ then } q\in B }} \prod_{(p,q)\in \pi} \delta_{z_p=z_q}  = \tau\left(A_{[1,n)}(x) l_{z_n}^*P_l B(x)\right) = \tau\left( B(x) P_l r_{z_n} A_{[1,n)}(x)\right).
	\end{equation*}
	\begin{equation*}
		\sum_{\substack{\pi\in NCP\left( AB\right), \\ \#\{(p,q)\in\pi\ |\ p\in A,\ q\in  B\} = l, \\ \text{let } (B_{m},q)\in \pi, \text{ then } q\in A }} \prod_{(p,q)\in \pi} \delta_{z_p=z_q}  = \tau\left(A(x) P_l r_{z_m} B_{[n+1,m)}(x)\right) = \tau\left( B_{[n+1,m)}(x) l_{z_m}^*P_l A(x)\right).
	\end{equation*}
	\begin{equation*}
		\sum_{\substack{\pi\in NCP\left( AB\right), \\ \#\{(p,q)\in\pi\ |\ p\in A,\ q\in  B\} = l, \\ \text{let } (A_{n},q)\in \pi, \text{ then } q\in B, \\ \text{let } (B_{m},q)\in \pi, \text{ then } q\in A }} \prod_{(p,q)\in \pi} \delta_{z_p=z_q}  = \tau\left(A_{[1,n)}(x) l_{z_n}^*P_l r_{z_m} B_{[n+1,m)}(x)\right) = \tau\left( B_{[n+1,m)}(x) l_{z_m}^* P_l r_{z_n} A_{[1,n)}(x)\right).
	\end{equation*}
	Consequently, by plugging those equalities in \eqref{smcsmcscsc}, with 
	$$\Delta_{\{i,j\}} := \sum_{l\geq 1} \kappa_{\{i,j\}}^l P_l, $$
	we get that
	\begin{align*}
		&\tau\left(A_{1}\left(x^{1}\right)\dots A_{n}\left(x^{n}\right)\right) \\ 
		&= \sum_{K\in \mathcal{K}_n}\ \sum_{\substack{\forall i\in [1,n], j\in C_i^{**}(K),\ 0 < k_{i,i^+} < \dots < k_{i,{i^-}^-} < \deg A_i} }\ \prod_{i\in\widehat{K}} \tau\left(A_i(x)\right) \\
		&\quad\quad\quad\quad\quad\quad\quad\quad\times \prod_{\substack{\{i,j\}\in K, \\ \text{such that } j=i^- \text{ in } C_i(K),\ i=j^- \text{ in } C_j(K) }} \tau\left( (A_i)_{(k_{i,j^-},k_{i,j}]} \Delta_{\{i,j\}} (A_j)_{(k_{j,i^-},k_{j,i}]} \right) \\
		&\quad\quad\quad\quad\quad\quad\quad\quad\times \prod_{\substack{\{i,j\}\in K, \\ \text{such that } j\neq i^-  \text{ in } C_i(K),\ i=j^- \text{ in } C_j(K) }} \tau\left( (A_i)_{(k_{i,j^-},k_{i,j})} l^*_{z_{k_{i,j}}} \Delta_{\{i,j\}} (A_j)_{(k_{j,i^-},k_{j,i}]} \right) \\
		&\quad\quad\quad\quad\quad\quad\quad\quad\times \prod_{\substack{\{i,j\}\in K, \\ \text{such that } j=i^- \text{ in } C_i(K),\ i\neq j^- \text{ in } C_j(K) }} \tau\left( (A_i)_{(k_{i,j^-},k_{i,j}]} \Delta_{\{i,j\}} r_{z_{k_{j,i}}} (A_j)_{(k_{j,i^-},k_{j,i})} \right) \\
		&\quad\quad\quad\quad\quad\quad\quad\quad\times \prod_{\substack{\{i,j\}\in K, \\ \text{such that } j\neq i^- \text{ in } C_i(K),\ i\neq j^- \text{ in } C_j(K) }} \tau\left( (A_i)_{(k_{i,j^-},k_{i,j})} l^*_{z_{k_{i,j}}} \Delta_{\{i,j\}} r_{z_{k_{j,i}}} (A_j)_{(k_{j,i^-},k_{j,i})} \right).
	\end{align*}
	
\end{proof}

One of the strength of Equation \eqref{pojfesv} obtained in the previous lemma is that both the left and right hand side of this equation are multilinear as a function in the monomials $A_i$. Consequently, one has the following lemma which we will use during the rest of the paper.

\begin{lemma}
	\label{sdtyuikmn}
	Let $A_1,\dots,A_n\in \A\langle X_1,\dots X_d\rangle$ polynomials, $x^1,\dots,x^n$ $d$-tuples of semicircular variables. We assume that for all $u,v\in [1,d]$, $i,j\in [1,n]$, 
	$$ \tau(x^i_u x_v^j) = \delta_{u=v} \kappa_{\{i,j\}}.$$
	Then with 
	$$\widetilde{K}= \{ (i,j)\ |\ i\in [1,n],\ j\in C_i^{**}(K)\},$$
	one has that, 
	\begin{align*}
		&\tau\left(A_{1}\left(x^{1}\right)\dots A_{n}\left(x^{n}\right)\right) 	= \sum_{K\in \mathcal{K}_n}\ \sum_{z\in [1,d]^{\widetilde{K}}}\ H^K_z\left( \bigotimes_{i\in [1,n]}\ \left(\bigotimes_{j\in C_i^{**}(K) } \partial_{z_{i,j}}\right)(A_i) \right),
	\end{align*}
	where with the notations of Definition \ref{3applicationopv} with $\A=\C$, if $C_i(K)\neq \{i\}$,
	$$ \bigotimes_{j\in C_i^{**}(K) } \partial_{z_{i,j}} = \partial_{z_{i,i^+}}\otimes \partial_{z_{i,{i^+}^+}}\otimes\cdots\otimes \partial_{z_{i,{i^-}^-}}, $$
	and else
	$$ \bigotimes_{j\in C_i^{**}(K) } \partial_{z_{i,j}} = \id_{\C\langle X_1,\dots,X_d\rangle}. $$
	Finally with $\Delta_{\{i,j\}} := \sum_{l\geq 1} \kappa_{\{i,j\}}^l P_l$, where $P_l$ is the orthogonal projection on the subspace $\Hi^{\otimes l}$ of the Fock space,
	\begin{align*}
		H^K_z\left( \bigotimes_{i\in [1,n]}\ \left(\bigotimes_{j\in C_i^*(K) } A_{i,j} \right)\right) &=  \prod_{i\in\widehat{K}} \tau\left(A_{i,i}(x)\right) \\
		&\times \prod_{\substack{\{i,j\}\in K, \\ \text{such that } j=i^- \text{ in } C_i(K),\ i=j^- \text{ in } C_j(K) }} \tau\left( A_{i,j}(x) \Delta_{\{i,j\}} A_{j,i}(x) \right) \\
		&\times \prod_{\substack{\{i,j\}\in K, \\ \text{such that } j\neq i^-  \text{ in } C_i(K),\ i=j^- \text{ in } C_j(K) }} \tau\left( A_{i,j}(x) l^*_{z_{i,j}} \Delta_{\{i,j\}} A_{j,i}(x) \right) \\
		&\times \prod_{\substack{\{i,j\}\in K, \\ \text{such that } j=i^- \text{ in } C_i(K),\ i\neq j^- \text{ in } C_j(K) }} \tau\left( A_{i,j}(x) \Delta_{\{i,j\}} r_{z_{j,i}} A_{j,i}(x) \right) \\
		&\times \prod_{\substack{\{i,j\}\in K, \\ \text{such that } j\neq i^- \text{ in } C_i(K),\ i\neq j^- \text{ in } C_j(K) }} \tau\left( A_{i,j}(x) l^*_{z_{i,j}} \Delta_{\{i,j\}} r_{z_{j,i}} A_{j,i}(x) \right).
	\end{align*}
	
\end{lemma}

\subsection{The combinatorics of circle sets}

This subsection is dedicated to studying the combinatorics of configurations. As mentioned in the previous section, it is linked with the notion of non-crossing partitions. However, because we allow elements of circle sets to be in several pairs, this changes noticeably the combinatorics. In the following lemma, we compute an upper bound on the number of possible pairs as well as on the cardinality of the set of configuration.

\begin{lemma}
	
	\label{ssxdcffvg}
	
	Let $E$ be a circle set of cardinal at least $2$, and $K\in\mathcal{K}(E)$, we define the following quantity
	$$ c_K := \sum_{i\in E} \#\{j\in E\ |\ \{i,j\}\in K\}. $$
	Then with $n=\#E$,
	$$ \forall K\in\mathcal{K}(E),\ c_K \leq 4n - 6,$$
	besides, one has that $\# \mathcal{K}(E) \leq (80e)^n$.
	
\end{lemma}

\begin{rem}
	Note that the inequality $c_K \leq 4n - 6$ is optimal. Indeed if $E=((1,n))$, it is an equality for 
	$$ K= \bigcup_{j\in [2,n]} \{\{1,j\}\} \bigcup_{j\in [1,n]} \{\{j,j+1\}\}. $$
\end{rem}

\begin{proof}
	
	Let us prove the first statement. We proceed by induction on the cardinality of $E$. If $\# E=2$, then $K=\{\{ 1,2 \}\}$ and $c_K=2$, or $K=\emptyset$ and $c_K=0$. Now let $K\in\mathcal{K}(E)$ with $\#E=n$, if $C_1(K)=\{1\}$, then $K\in\mathcal{K}(E\setminus\{1\})$ and since $\# E\setminus\{1\} = n-1$ one can apply our induction hypothesis. Otherwise, for $j\in C_1(K)$, we define 
	$$K_j = \left\{ \{u,v\}\in K\ |\ u,v\in [j,j^-] \right\}\setminus \left\{\{j,j^-\}\right\}. $$
	Note that in the equation above, $j^-$ is the element preceding $j$ in $C_1(K)$ but the interval $[j,j^-]$ is in $E$ whose ordering is reversed.
	 
	We also define $\widetilde{K}_j$ similarly, but if there exists $i\in[1,n]$ such that both $\{i,j^-\}$ and $\{i,j\}$ are distinct and belong to $K_j$ then we remove $\{i,j^-\}$. Finally we define $E_j$ the circle set $[j,j^-]$ quotiented with the equivalence relation $a\sim b$ if and only if $a=b$, or if they are both equal to either $j$ or $j^-$. Consequently one can view $\widetilde{K}_j$ as an element of $\mathcal{K}(E_j)$. Let us now make a few remark.
	\begin{itemize}
		\item If $j^-$ is the point preceding $j$ in $C_1(K)$ but also following it in $E$, then by construction $E_j$ has one element and $c_{E_j}=0$. Besides it is impossible to find $i\in[1,n]$ such that both $\{i,j\}$ and $\{i,j^+\}$ are distinct and belong to $K_j$ since $i$ would have to be in between $j$ and $j^+$.
		\item The cardinal of $E_j$ is equal to the one of $((j,j^+))$ minus one. Consequently $\sum_{j\in C_1(K)} \# E_j = n$.
		\item Since the elements of $K$ are non-crossing, necessarily there can be at most one $i\in[1,n]$ such that both $\{i,j\}$ and $\{i,j^+\}$ belongs to $K_j$.
		\item If $j\in \{1^+,1\}$, then there cannot exist $i\in[1,n]$ such that both $\{i,j\}$ and $\{i,j^-\}$ belongs to $K_j$, since then $i$ would belong to $C_1(K)$ and be in between $1^-$ and $1$ or $1$ and $1^+$, which is impossible by definition.
	\end{itemize}
	Besides, $K$ is the union of the following sets,
	\begin{align*}
		K =&\ \bigcup_{j\in C_1(K)\setminus\{1\}} \{\{1,j\}\} \\
		&\bigcup_{j\in C_1(K)\setminus \{1,1^-\}} \{\{j,j^-\}\}\cap K \\
		&\bigcup_{j\in C_1(K)} \widetilde{K}_j\\
		&\bigcup_{j\in C_1(K)\setminus\{1^-,1\}} \bigcup_{i\in [1,n]} \{\{i,j^-\}\}\cap K.
	\end{align*}
	Consequently, with $l$ the cardinal of $C_1(K)$ and $m$ the number of $j\in C_1(K)$ such that $\# E_j >1$, by induction
	\begin{align*}
		c_K &\leq 2(l-1) + 2(l-2) + \sum_{j\in C_1(K)} c_{\widetilde{K}_j} + 2m \\
		&\leq 4l-6+ 2m + \sum_{j\in C_1(K)} 4\# E_j -4-2\delta_{\# E_j >1}  \\
		&\leq 4n-6.
	\end{align*}
	Hence the conclusion. Let us now prove the second statement. Given $K\in\mathcal{K}(E)$, we define the following quantity
	$$ l_i := \#\{j\in E\ |\ \{i,j\}\in K\}. $$
	and we partition $\mathcal{K}(E)$ like this
	$$ \mathcal{K}(E) = \bigcup_{k_1,\dots, k_n\geq 0} \left\{ K\in \mathcal{K}(E)\ |\ \forall i\in E, l_i=k_i \right\}.$$
	Thus thanks to the first part of the lemma,
	$$ \# \mathcal{K}(E) \leq \sum_{\substack{k_1,\dots, k_n\geq 0 \\ k_1+\cdots+k_n \leq 4n-6}} \#\left\{ K\in \mathcal{K}(E)\ |\ \forall i\in E, l_i=k_i \right\}.$$
	Besides, the map which to an element of $\left\{ K\in \mathcal{K}(E)\ |\ \forall i\in E, l_i=k_i \right\}$ defines an element of $NC2P([1,k_1+\dots+k_n])$ by pairing $k_1+\dots+k_{i-1}+l$ to $k_1+\dots+k_{j-1}+m$ where (starting the count from $i$) $j$ is the $l+1$-th element of $C_i(K)$ and $m$ such that $i$ be the $m+1$-th element of $C_j(K)$ is injective. Consequently, thanks to Lemma 8.9 of \cite{nica_speicher_2006},
	$$  \#\left\{ K\in \mathcal{K}(E)\ |\ \forall i\in E, l_i=k_i \right\} \leq \# NC2P([1,k_1+\dots+k_n]) \leq 2^{k_1+\dots+k_n}.$$
	Hence
	$$ \# \mathcal{K}(E) \leq 2^{4n-6} \#\left\{ k_1,\dots, k_n\geq 0\ |\ k_1+\cdots+k_n \leq 4n-6\right\}.$$
	Besides
	$$ \#\left\{ k_1,\dots, k_n\geq 0\ |\ k_1+\cdots+k_n \leq l\right\} = \sum_{k_n=0}^{l} \#\left\{ k_1,\dots, k_{n-1}\geq 0\ |\ k_1+\cdots+k_{n-1} \leq l-k_{n}\right\} .$$
	Thus since $\#\left\{ k_1\geq 0\ |\ k_1 \leq l\right\} =l+1$, let us assume that $\#\left\{ k_1,\dots, k_n\geq 0\ |\ k_1+\cdots+k_n \leq l\right\} \leq (l+n)^n/n!$, then thanks to the equation above,
	$$ \#\left\{ k_1,\dots, k_{n+1}\geq 0\ |\ k_1+\cdots+k_{n+1} \leq l\right\} \leq \sum_{k_{n+1}=0}^{l} \frac{(l-k_n+n)^n}{n!}\leq \int_{-1}^{l} \frac{(l+n-t)^n}{n!} dt \leq \frac{(l+n+1)^{n+1}}{(n+1)!}.$$
	Hence we get that,
	$$ \# \mathcal{K}(E) \leq 2^{4n-6} \frac{(5n-6)^{n}}{n!} \leq 2^{4n-6} \left(\frac{5n-6}{n}\right)^{n} e^n \leq (5\times 2^4e)^n.$$
	
\end{proof}

\subsection{The action of a permutation on $H^K_z$}

Given that in Lemma \ref{poiuytrew} one has to permutes polynomial in free semicircular variables, in order to Lemma \ref{sdtyuikmn}, we need to understand how the operator $H^K_z$ evaluated in a tensor whose components have been permuted behave. We do so in the following lemma.

\begin{lemma}
	\label{nkytrdds}
	
	With the notations of Lemma \ref{sdtyuikmn} and \ref{ssxdcffvg}, we set $l_K = c_K/2+\#\widehat{K}$ and $m_K = c_K+\#\widehat{K}$. Given a permutation $\sigma\in \mathbb{S}_n $, and $A_{i}\in \C\langle X_1,\dots,X_d\rangle$, one has that
	\begin{align*}
		& H^K_z\left( \bigotimes_{ i\in [1,n]}\ \left(\bigotimes_{j\in C_i^{**}(K)} \partial_{z_{i,j}}\right)(A_{\sigma(i)})\right) \\
		&\quad\quad\quad\quad\quad\quad\quad\quad\quad\quad\quad\quad\quad = \tau^{\otimes l_K} \left( u^{\sigma,K}_z\left( \bigotimes_{p\in [1,n]} \ \left(\bigotimes_{q\in \sigma\left(C_{\sigma^{-1}(p)}^{**}(K)\right)} \partial_{z_{\sigma^{-1}(p),\sigma^{-1}(q)}}\right)(A_p) \right) \right),
	\end{align*}
	where $u^{\sigma,K}_z:\C\langle X_1,\dots,X_d\rangle^{\otimes m_K}\mapsto \CC^*(x)^{\otimes l_k}$ is defined by,
	\begin{align*}
		u^{\sigma,K}_z\left( \bigotimes_{ p\in [1,n]}\ \bigotimes_{q\in \sigma(C_{\sigma^{-1}(p)}^*(K))} A_{p,q}\right) := \prod_{p\in [1,n]} \prod_{q \in \sigma(C_{\sigma^{-1}(p)}^*(K))} u_{p,q}^{\sigma,K}({A}_{p,q}(x)) \Lambda_{p,q}^z.
	\end{align*}
	with ${\Lambda}_{p,q}^z\in B(\F)^{\otimes l_K}$ an operator of norm smaller than one and
	$$u_{p,q}^{\sigma,K}({A}_{p,q}(x)) = A_{p,q}(x) \bigotimes_{\substack{ \{a,b\} \neq \{\sigma^{-1}(p),\sigma^{-1}(q)\}  \text{ such that} \\ a\in [1,n], b\in C^*_a(K) }} \id_{\CC^*(x)}.$$
	More precisely, with $i=\sigma^{-1}(p)$ and $j=\sigma^{-1}(q)$,
	\begin{itemize}
		\item if $p\geq q$, $\Lambda_{p,q}^z = (\id_{\CC^*(x)})^{\otimes l_K}$,
		\item else if  $j=i^-$ in $C_i(K)$, and $i=j^-$ in $C_j(K)$, then $\Lambda_{p,q}^z= \Delta_{\{i,j\}} \bigotimes_{\substack{ \{a,b\} \neq \{i,j\}}} \id_{\CC^*(x)}$,
		\item else if $j\neq i^-$  in $C_i(K)$, $i=j^-$ in $C_j(K)$, $\Lambda_{p,q}^z= (l^*_{z_{i,j}} \Delta_{\{i,j\}}) \bigotimes_{\substack{ \{a,b\} \neq \{i,j\}}} \id_{\CC^*(x)}$,
		\item else if $j=i^-$ in $C_i(K)$, $i\neq j^-$ in $C_j(K)$, $\Lambda_{p,q}^z = (\Delta_{\{i,j\}} r_{z_{j,i}}) \bigotimes_{\substack{ \{a,b\} \neq \{i,j\}}} \id_{\CC^*(x)}$,
		\item else if $j\neq i^-$ in $C_i(K)$, $i\neq j^-$ in $C_j(K)$, $\Lambda_{p,q}^z = l^*_{z_{i,j}} \Delta_{\{i,j\}} r_{z_{j,i}}$.
	\end{itemize}

\end{lemma}
\vspace{0,4cm}

Note that in the lemma above, we use the convention that if $C_{\sigma^{-1}(p)}^*(K)\neq \{\sigma^{-1}(p)\}$, then
\begin{align*}
	\bigotimes_{q\in \sigma\left(C_{\sigma^{-1}(p)}^{**}(K)\right)} \partial_{z_{\sigma^{-1}(p),\sigma^{-1}(q)}} &= \bigotimes_{q\in C_{\sigma^{-1}(p)}^{**}(K)} \partial_{z_{\sigma^{-1}(p),q}} \\
	&= \partial_{z_{\sigma^{-1}(p),\sigma^{-1}(p)^+}}\otimes \partial_{z_{\sigma^{-1}(p),{\sigma^{-1}(p)^+}^+}}\otimes\cdots\otimes \partial_{z_{\sigma^{-1}(p),{\sigma^{-1}(p)^-}^-}}.
\end{align*}

\begin{proof}
	
	Thanks to Lemma \ref{sdtyuikmn}, one has that for any $A_{i,j}\in \C\langle X_1,\dots,X_d\rangle$,
	\begin{align*}
		&H^K_z\left( \bigotimes_{ i\in [1,n]}\ \left(\bigotimes_{j\in C_i^*(K)} A_{\sigma(i),\sigma(j)} \right)\right) \\
		&=  \prod_{i\in\widehat{K}} \tau\left(A_{\sigma(i),\sigma(i)}(x)\right) \times \prod_{\substack{\{i,j\}\in K }} \tau\left( \widetilde{A}_{\sigma(i),\sigma(j)}(x) \widetilde{\Lambda}_{i,j}^z \widetilde{A}_{\sigma(j),\sigma(i)}(x) \right),
	\end{align*}
	where $\widetilde{\Lambda}_{i,j}^z$ is an operator on the full Fock space of norm smaller than $1$ since for any $i,j$, $\norm{\Delta_{\{i,j\}}} \leq \kappa_{\{i,j\}} \leq 1$, and the norm of a left or right multiplication operator is equal to one. Besides thanks to Equation \eqref{idjnvcdskmvcos0},\eqref{idjnvcdskmvcos},\eqref{idjnvcdskmvcos2} and \eqref{idjnvcdskmvcos3}, one can also pick $\widetilde{\Lambda}_{i,j}^z$ such that if $\sigma(i) < \sigma(j)$, then $\widetilde{A}_{\sigma(i),\sigma(j)}(x) = {A}_{\sigma(i),\sigma(j)}(x)$ and $\widetilde{A}_{\sigma(j),\sigma(i)}(x)= {A}_{\sigma(j),\sigma(i)}(x)$, whereas if $\sigma(i)>\sigma(j)$, then $\widetilde{A}_{\sigma(i),\sigma(j)}(x) = {A}_{\sigma(j),\sigma(i)}(x)$ and $\widetilde{A}_{\sigma(j),\sigma(i)}(x)= {A}_{\sigma(i),\sigma(j)}(x)$. Consequently,
	\begin{align*}
		&H^K_z\left( \bigotimes_{ i\in [1,n]}\ \left(\bigotimes_{j\in C_i^*(K)} A_{\sigma(i),\sigma(j)} \right)\right) \\
		&=  \prod_{p\in\sigma(\widehat{K})} \tau\left(A_{p,p}(x)\right) \times \prod_{\substack{\{\sigma^{-1}(p),\sigma^{-1}(q)\}\in K }} \tau\left( \widetilde{A}_{p,q}(x) \widetilde{\Lambda}_{\sigma^{-1}(p),\sigma^{-1}(q)}^z \widetilde{A}_{q,p}(x) \right) \\
		&=  \prod_{p\in [1,n]}\ \prod_{\substack{q \in \sigma(C_{\sigma^{-1}(p)}^*(K)) \\ \text{such that }p=q }} \tau\left( A_{p,p}(x) \right) \times \prod_{\substack{q \in \sigma(C_{\sigma^{-1}(p)}^*(K)) \\ \text{such that }p<q }} \tau\left( A_{p,q}(x) \widetilde{\Lambda}^z_{\sigma^{-1}(p),\sigma^{-1}(q)} A_{q,p}(x) \right),
	\end{align*}
	Hence this prompt us to define,
	$$u_{p,q}^{\sigma,K}({A}_{p,q}(x)) = A_{p,q}(x) \bigotimes_{\substack{ \{a,b\} \neq \{\sigma^{-1}(p),\sigma^{-1}(q)\}  \text{ such that} \\ a\in [1,n], b\in C^*_a(K) }} \id_{\CC^*(x)},$$
	$$
	\Lambda_{p,q}^z = \left\{
	\begin{array}{ll}
		\widetilde{\Lambda}^z_{\sigma^{-1}(p),\sigma^{-1}(q)} \bigotimes_{\substack{  \{a,b\} \neq \{ \sigma^{-1}(p),\sigma^{-1}(q) \} \\ \text{such that } a\in [1,n], b\in C^*_a(K) }} \id_{B(\F)} & \mbox{if } p<q \\
		\id_{B(\F)}^{\otimes l_K} & \mbox{else.}
	\end{array}
	\right.
	$$
	Finally, one has that
	\begin{align*}
		&H^K_z\left( \bigotimes_{ i\in [1,n]}\ \left(\bigotimes_{j\in C_i^*(K)} A_{\sigma(i),\sigma(j)} \right)\right) = \tau^{\otimes l_K} \left( \prod_{p\in [1,n]}\ \prod_{q \in \sigma(C_{\sigma^{-1}(p)}^*(K))} u_{p,q}^{\sigma,K}({A}_{p,q}(x)) \Lambda_{p,q}^z \right), \nonumber
	\end{align*}
	Consequently, let us now assume that the polynomials $A_i$ are monomials, then with the convention $\prod\limits_{j\in C_i^{**}(K)} D_j= D_{i^+} D_{{i^+}^+} \dots D_{{i^-}^-}$, one has that
	\begin{align*}
		& H^K_z\left( \bigotimes_{ i\in [1,n]}\ \left(\bigotimes_{j\in C_i^{**}(K)} \partial_{z_{i,j}}\right)(A_{\sigma(i)})\right) \\
		&= \sum_{\forall i\in [1,n],\ A_{\sigma(i)} = \left(\prod\limits_{j\in C_i^{**}(K)} A_{\sigma(i),\sigma(j)} X_{z_{i,j}}\right) A_{\sigma(i),\sigma(i^-)}} H^K_z\left( \bigotimes_{ i\in [1,n]}\ \left(\bigotimes_{j\in C_i^*(K)} A_{\sigma(i),\sigma(j)} \right)\right) \\
		&= \sum_{\forall i\in [1,n],\ A_{\sigma(i)} = \left(\prod\limits_{j\in C_i^{**}(K)} A_{\sigma(i),\sigma(j)} X_{z_{i,j}}\right) A_{\sigma(i),\sigma(i^-)}} \tau^{\otimes l_K} \left( \prod_{p\in [1,n]}\ \prod_{q \in \sigma(C_{\sigma^{-1}(p)}^*(K))} u_{p,q}^{\sigma,K}({A}_{p,q}(x)) \Lambda_{p,q}^z \right) \\
		&= \tau^{\otimes l_K} \left( \prod_{p\in [1,n]}\ \sum_{ A_p = \left(\prod\limits_{q\in \sigma\left(C_{\sigma^{-1}(p)}^{**}(K)\right)} A_{p,q} X_{z_{\sigma^{-1}(p),\sigma^{-1}(q)}}\right) A_{p,\sigma(\sigma^{-1}(p)^-)}}  \prod_{q \in \sigma(C_{\sigma^{-1}(p)}^*(K))} u_{p,q}^{\sigma,K}({A}_{p,q}(x)) \Lambda_{p,q}^z \right) \\
		&= \tau^{\otimes l_K} \left( u_z^{\sigma,K}\left( \bigotimes_{p\in [1,n]} \ \left(\bigotimes_{q\in \sigma\left(C_{\sigma^{-1}(p)}^{**}(K)\right)} \partial_{z_{\sigma^{-1}(p),\sigma^{-1}(q)}}\right)(A_p) \right) \right),
	\end{align*}
	where
	\begin{align*}
		u^{\sigma,K}_z\left( \bigotimes_{ p\in [1,n]}\ \bigotimes_{q\in \sigma(C_{\sigma^{-1}(p)}^*(K))} A_{p,q}\right) := \prod_{p\in [1,n]} \prod_{q \in \sigma(C_{\sigma^{-1}(p)}^*(K))} u_{p,q}^{\sigma,K}({A}_{p,q}(x)) \Lambda_{p,q}^z.
	\end{align*}
	
\end{proof}

\subsection{The norm of polynomials in free semicircular variables}

This subsection is dedicated to proving that one can control the norm of any coefficients of a polynomial in free semicircular variables by the norm of this polynomial up to a universal constant which only depends on the degree of our polynomial.

\begin{lemma}
	\label{xsdfgyujk}
	There exist a constant $C_n$ such that for every polynomial 
	$$P=\sum_{M \text{ monomial}} a_M\otimes M \in\A\langle X_1,\dots,X_d\rangle$$
	of degree at most $n$, and for every coefficient $a_M$ of $P$,
	\begin{equation}
		\label{vldksmvls}
		\norm{a_M} \leq C_n \norm{P(x)},
	\end{equation}
	where $x$ is a $d$-tuple of semicircular variables.
\end{lemma}

\begin{proof}
	
	Let $M=X_{i_1}\cdots X_{i_l}$ be a monomial. With $z$ an element of the Fock space, and $P_z$ the orthogonal projection on this vector,
	$$ \norm{P(x)} \geq \norm{ \id_{\A}\otimes P_{e_{i_1}\otimes\cdots\otimes e_{i_l}} P(x) \id_{\A}\otimes P_{\Omega}}.$$
	But for any monomial $N$ of degree strictly smaller than $l$, the vector $N(x)\Omega$ is an element of
	$$ \bigoplus_{n= 0}^{l-1} \Hi^{\otimes n},$$
	hence $P_{e_{i_1}\otimes\cdots\otimes e_{i_l}} N(x) P_{\Omega}=0$. Similarly, if $N=X_{j_1}\cdots X_{j_l}$ is a monomial of degree $l$, then 
	$$ P_{e_{i_1}\otimes\cdots\otimes e_{i_l}} (N(x) \Omega - e_{j_1}\otimes\cdots\otimes e_{j_l}) =0.$$
	Consequently,
	$$P_{e_{i_1}\otimes\cdots\otimes e_{i_l}} N(x) P_{\Omega}= \delta_{N=M}\ |e_{i_1}\otimes\cdots\otimes e_{i_l}\rangle \langle \Omega|,$$
	thus
	$$ \norm{P(x)} \geq \norm{ a_M\otimes |e_{i_1}\otimes\cdots\otimes e_{i_l}\rangle \langle \Omega| + \sum_{N \text{ monomial},\ \deg N > \deg M	} a_M\otimes P_{e_{i_1}\otimes\cdots\otimes e_{i_l}} N(x) P_{\Omega}}. $$
	Hence,
	\begin{align*}
		\norm{a_M} &= \norm{ a_M\otimes |e_{i_1}\otimes\cdots\otimes e_{i_l}\rangle \langle \Omega|} \\
		&\leq \norm{P(x)} + \norm{\sum_{N \text{ monomial},\ \deg M > \deg N	} a_M\otimes P_{e_{i_1}\otimes\cdots\otimes e_{i_l}} N(x) P_{\Omega}} \\
		&\leq \norm{P(x)} + \sum_{N \text{ monomial},\ \deg M > \deg N	} \norm{ a_M} 2^{\deg N}.
	\end{align*}
	In particular, if $\deg M = \deg P$, then $\norm{a_M} \leq \norm{P(x)}$. And the conclusion follows by induction on the degree of $M$ thanks to the equation above.
	
\end{proof}

\section{Proof of the main theorems}

This section is dedicated to proving the main results of this paper that we listed in the introduction. Note that although Theorem \ref{detcase} is a more general statement than Theorem \ref{popopohdvc}, and thus that we could simply prove directly Theorem \ref{detcase}, we instead chose to first prove Theorem \ref{popopohdvc}, and then to prove the more general result by detailing where the proof differs. We do so due to the technicalities involved with those proofs.

\subsection{A few results on the coefficients of the expansion}

In this subsection, we assembles the different lemmas of the previous section to deduce some estimates on the coefficients of the expansion of Lemma \ref{3apparition}.

\begin{lemma}
	
	\label{gfdcvbnjuytr}
	
	With $m_K=\widehat{K}+c_K$, one can find the following objects,
	\begin{itemize}
		\item a constant $c_{s_1,\dots,s_p}$ smaller than $(4p)!$,
		\item partitions $\pi_l^{s_1,\dots,s_p}$ of $[1,m_K-1]$, which we shall simply denote by $\pi_l$, whose elements are all pairs or singleton,
		\item operators $u^{l,K,T_p}_z: \C\langle X_1,\dots,X_d\rangle^{\otimes m_K}\to \CC^*(x)^{\otimes l_K}$, such that
		\begin{align*}
			u^{l,K,T_p}_z\left( \bigotimes_{ s\in [1,m_K]}\ A_{s}\right) := \prod_{s\in [1,m_K]} u_{s}^{l,K}({A}_{s}(x)) \Lambda^{l,T_p,K}_{z,s},
		\end{align*}
		where $\Lambda^{l,T_p,K}_{z,s} \in B(\F)^{\otimes l_K}$ has norm smaller than one, and $u_{s}^{l,K} : \CC^*(x)\to \CC^*(x)^{\otimes l_K}$ is a map of the following form
		$$ u_{s}^{l,K}(a) = \id_{\CC^*(x)}^{\otimes n} \otimes a \otimes \id_{\CC^*(x)}^{\otimes l_K-n-1},$$
		for a given $n$ which depends on $s,l$ and $K$,
	\end{itemize}
	such that for any $P\in \C\langle X_1,\dots, X_d\rangle$,
	\begin{align*}
		&\tau\left(\left(L^{p}_{s_p} \dots L^{1}_{s_1}\right)(P)\left(x^{T_p}\right)\right) \\
		&= \frac{1}{2^p} \sum_{l=1}^{c_{s_1,\dots,s_p}} \sum_{K\in \mathcal{K}_{4p+1}}\ \sum_{\substack{1\leq z_1,\dots,z_{m_K-1} \leq d \\ \forall E\in \pi_l, \forall n,m\in E, z_n=z_m}} \tau^{\otimes l_K} \left( u^{l,K,T_p}_z\left(\left(\partial_{z_1}\otimes \cdots \otimes \partial_{z_{m_K-1}}\right)(P) \right) \right).
	\end{align*}
	
\end{lemma}

\begin{proof}
	
	Thanks to Lemma \ref{poiuytrew}, with the same notation, one has that 
	\begin{align*}
		\left(L^{p}_{s_p} \dots L^{1}_{s_1}\right)(P) &= \frac{1}{2^p} \sum_{l=1}^{\widehat{c}_{s_1,\dots,s_p}} \sum_{\substack{1\leq z_1,\dots,z_{4p} \leq d \\ z_{\widehat{\pi}_l(1)} = z_{\widehat{\pi}_l(2)}, \dots, z_{\widehat{\pi}_l(4p-1)} = z_{\widehat{\pi}_l(4p)}}} \ev_l\circ\widehat{\sigma}_l\circ m\circ\left(\partial_{z_1}\otimes \cdots \otimes \partial_{z_{4p}}\right)(P), \nonumber
	\end{align*}
	Consequently one can find a permutation $\sigma_l\in \mathbb{S}_{4p+1}$ and $d$-tuples of free semicircular variables $x^1,\dots,x^{4p+1}$ whose covariance depends only on $T_p$ such that for any $A_i\in \C\langle X_1,\dots, X_d\rangle$,
	$$ \tau\left(\left(\ev_l\circ\widehat{\sigma}_l\circ m\circ\left(A_1\otimes \cdots \otimes A_{4p+1}\right)\right)(x^{T_p})\right) = \tau\left(A_{\sigma_l(1)}(x^1) \cdots A_{\sigma_l(4p+1)}(x^{4p+1})\right) $$
	Hence, thanks to Lemma \ref{sdtyuikmn}, one has,
	\begin{align*}
		&\tau\left(\left(\ev_l\circ\widehat{\sigma}_l\circ m\circ\left(\partial_{z_1}\otimes \cdots \otimes \partial_{z_{4p}}\right)(P)\right)(x^{T_p})\right) \\
		&= \sum_{K\in \mathcal{K}_{4p+1}}\ \sum_{z\in [1,d]^{\widetilde{K}}}\ H^{K}_z\left( \bigotimes_{i\in [1,4p+1]}\ \left(\bigotimes_{j\in C_i^{**}(K)} \partial_{z_{i,j}}\right)\left(\sigma_l\left(\left(\partial_{z_1}\otimes \cdots \otimes \partial_{z_{4p}}\right)(P)\right)\right) \right). \\
	\end{align*}
	And thanks to \ref{nkytrdds}, one has that
	\begin{align*}
		&\tau\left(\left(\ev_l\circ\widehat{\sigma}_l\circ m\circ\left(\partial_{z_1}\otimes \cdots \otimes \partial_{z_{4p}}\right)(P)\right)(x^{T_p})\right) \\
		&= \sum_{K\in \mathcal{K}_{4p+1}}\ \sum_{z\in [1,d]^{\widetilde{K}}}\ \tau^{\otimes l_K} \left(\rule{0cm}{1cm}\right. u^{\sigma_l,K,T_p}_z \left(\rule{0cm}{1cm}\right. \bigotimes_{a\in [1,4p+1]} \ \left(\rule{0cm}{1cm}\right.\bigotimes_{b\in \sigma_l\left(C_{\sigma_l^{-1}(a)}^{**}(K)\right)} \partial_{z_{\sigma_l^{-1}(a),\sigma_l^{-1}(b)}}\left.\rule{0cm}{1cm}\right) \\ 
		&\quad\quad\quad\quad\quad\quad\quad\quad\quad\quad\quad\quad\quad\quad\quad\quad\quad\quad\quad\quad\quad\quad\quad\quad\quad\quad\quad\quad\quad\quad\quad\quad \left(\partial_{z_1}\otimes \cdots \otimes \partial_{z_{4p}}\right)(P) \left.\rule{0cm}{1cm}\right) \left.\rule{0cm}{1cm}\right).
	\end{align*}
	Hence with $c_{s_1,\dots,s_p}=\widehat{c}_{s_1,\dots,s_p}$,
	\begin{align*}
		&\tau\left(\left(L^{p}_{s_p} \dots L^{1}_{s_1}\right)(P)\left(x^{T_p}\right)\right) = \frac{1}{2^p} \sum_{l=1}^{c_{s_1,\dots,s_p}} \sum_{K\in \mathcal{K}_{4p+1}}\ \sum_{\substack{1\leq z_1,\dots,z_{4p} \leq d \\ z_{\widehat{\pi}_l(1)} = z_{\widehat{\pi}_l(2)}, \dots, z_{\widehat{\pi}_l(4p-1)} = z_{\widehat{\pi}_l(4p)}}} \sum_{z\in [1,d]^{\widetilde{K}}} \\
		& \tau^{\otimes l_K} \left( u^{\sigma_l,K,T_p}_z\left( \bigotimes_{a\in [1,4p+1]} \ \left(\bigotimes_{b\in \sigma\left(C_{\sigma^{-1}(a)}^{**}(K)\right)} \partial_{z_{\sigma^{-1}(a),\sigma^{-1}(b)}}\right)\left(\partial_{z_1}\otimes \cdots \otimes \partial_{z_{4p}}\right)(P) \right) \right).
	\end{align*}
	To begin with, note that the number of noncommutative differential in the formula above is the following,
	$$ 4p + \sum_{i=1}^{4p+1} \# C_{\sigma^{-1}(i)}^{**}(K) = 4p + \sum_{i=1}^{4p+1} \# C_i^*(K)-1 = \# \widehat{K} + c_K -1 = m_K-1.$$
	Consequently, one can find
	\begin{itemize}
		\item a partition $\pi_l$ of $[1,m_K-1]$ whose element are either pair or singleton,
		\item a family of operators $u^{l,K,T_p}_z: \C\langle X_1,\dots,X_d\rangle^{\otimes m_K}\to \CC^*(x)^{\otimes l_K}$, such that
		\begin{align*}
			u^{l,K,T_p}_z\left( \bigotimes_{ s\in [1,m_K]}\ A_{s}\right) := \prod_{s\in [1,m_K]} u_{s}^{l,K}({A}_{s}(x)) \Lambda^{l,T_p,K}_{z,s},
		\end{align*}
		where $\Lambda^{l,T_p,K}_{z,s} \in B(\F)^{\otimes l_K}$ has norm smaller than one, and $u_{s}^{l,K} : \CC^*(x)\to \CC^*(x)^{\otimes l_K}$ is a map of the following form
		$$ u_{s}^{l,K}(a) = \id_{\CC^*(x)}^{\otimes n} \otimes a \otimes \id_{\CC^*(x)}^{\otimes l_K-n-1},$$
		for a given $n$ which depends on $s,l$ and $K$,
	\end{itemize}
	
	\noindent such that,
	\begin{align*}
		&\tau\left(\left(L^{p}_{s_p} \dots L^{1}_{s_1}\right)(P)\left(x^{T_p}\right)\right) \\
		&= \frac{1}{2^p} \sum_{l=1}^{c_{s_1,\dots,s_p}} \sum_{K\in \mathcal{K}_{4p+1}}\ \sum_{\substack{1\leq z_1,\dots,z_{m_K-1} \leq d \\ \forall E\in \pi_l, \forall n,m\in E, z_n=z_m}} \tau^{\otimes l_K} \left( u^{l,K,T_p}_z\left(\left(\partial_{z_1}\otimes \cdots \otimes \partial_{z_{m_K-1}}\right)(P) \right) \right).
	\end{align*}
		
\end{proof}

Thus, we can now formulate an operator-valued version of the expansion of Lemma \ref{3apparition} which will be more practical to handle for our purpose.

\begin{prop}
	
	\label{lvleavalvlkd}
	
	Given the following objects,
	\begin{itemize}
		\item $X^N=\left(X_1^N,\dots,X_d^N\right)$ a $d$-tuple of GUE random matrices,
		\item $P\in\A\langle X_1,\dots,X_d\rangle$ a polynomial of degree $n$,
	\end{itemize}
	With the notations of previous lemmas,
	\begin{align*}
		\E\Big[\id_{\A}\otimes\ts_N\Big(P^k\left(X^N\right) & \Big)\Big] = (\id_{\A}\otimes\tau)\left(P(x)^k\right) \\
		&+ \sum_{p= 1}^{nk/4} \frac{1}{2^{p}N^{2p}} \int_{A_p} e^{-t_1-\cdots-t_{2p}} \sum_{l=1}^{c_{T_p}} \sum_{K\in \mathcal{K}_{4p+1}}\ \sum_{\substack{1\leq z_1,\dots,z_{m_K-1} \leq d \\ \forall E\in \pi_l, \forall n,m\in E, z_n=z_m}} \\
		&\quad\quad\quad\quad\quad\quad \id_{\A}\otimes\tau^{\otimes l_K} \left(v^{l,K,T_p}_z\Big( \left(\partial_{z_1}\circ\dots\circ \partial_{z_{m_K-1}}\left(P^k\right)\right) \Big) \right)dt_1\dots dt_{2p}, \nonumber
	\end{align*}
	where $c_{T_p}\leq (4p)!$ and $v^{l,K,T_p}_z: \A\langle X_1,\dots,X_d\rangle^{\otimes m_K} \to \A\otimes \CC^*(x)^{\otimes l_K}$ is a linear map such that for all $P_s\in \A\langle X_1,\dots,X_d\rangle$,
	$$ \norm{v^{l,K,T_p}_z\left( \bigotimes_{ s\in [1,m_K]}\ P_{s}\right)} \leq \prod_{ s\in [1,m_K]}\ \norm{P_{s}(x)}.$$
	Besides,
	\begin{align*}
		\partial_{z_1}\circ\dots\circ \partial_{z_{m_K-1}}\left(P^k\right) =&\ \sum_{x=1}^{m_K-1} \sum_{r_1+\dots+r_x = m_K-1} \sum_{1\leq k_1<\dots< k_x\leq k} P^{k_1-1} \left(\partial_{z_1}\circ\dots\circ \partial_{z_{r_1}}P \right) P^{k_2-k_1-1}\dots\nonumber \\ 
		&\quad\quad\quad\quad\quad\quad\quad\quad\quad \dots P^{k_{x} - k_{x-1}-1} \left(\partial_{z_{r_1+\dots+r_{x-1}}}\circ \dots\circ\partial_{z_{m_K-1}} P\right) P^{k - k_{x}},
	\end{align*}
	with the convention that if $K,K'\in\A\langle X_1,\dots,X_d\rangle$, $a\in \A$, $A_1,\dots,A_l\in\C\langle X_1,\dots X_d\rangle$,
	\begin{align*}
		K \left(a_1\otimes A_1\otimes\dots\otimes A_l \right) K' = (K (a_1\otimes A_1))\otimes (1\otimes A_2) \otimes \cdots\otimes (1\otimes A_{l-1}) \otimes((a_1\otimes A_1)K').
	\end{align*}

\end{prop}

\begin{proof}
	
	To begin with, if we write $P=\sum_M a_M\otimes M$, one has that
	$$ P^k = \sum_{M_1,\dots,M_k} a_{M_1}\dots a_{M_k} \otimes M_1\dots M_k.$$
	Thus thanks to Lemma \ref{3apparition}, one has that
	\begin{align*}
		&\E\left[\id_{\A}\otimes\ts_N\Big(M_1\dots M_k\left(X^N\right)\Big)\right] \\
		&= \sum_{p= 0}^{nk/4} \frac{1}{N^{2p}} \int_{A_p} e^{-t_1-\cdots-t_{2p}} \sum_{M_1,\dots,M_k} a_{M_1}\dots a_{M_k} \tau\Big( \left(L^{{T}_p} \dots L^{{T}_1}\right)(M_1\dots M_k) \left(x^{T_p}\right) \Big)\  dt_1\dots dt_{2p}. \nonumber
	\end{align*}
	Note that thanks to Equation \eqref{xdffgyui}, one can determine the value of $(s_1,\dots,s_p)$ given $T_p$. Thus we set $c_{T_p}:=c_{s_1,\dots,s_p}$ for the appropriate choice of $s_1,\dots,s_p$. Thanks to Lemma \ref{gfdcvbnjuytr}, we get that for $p$ larger than $1$,
	\begin{align}
		\label{sdicoscs}
		&\sum_{M_1,\dots,M_k} a_{M_1}\dots a_{M_k} \tau\Big( \left(L^{{T}_p} \dots L^{{T}_1}\right)(M_1\dots M_k) \left(x^{T_p}\right) \Big) \\
		&= \frac{1}{2^{p}} \sum_{l=1}^{c_{T_p}} \sum_{K\in \mathcal{K}_{4p+1}}\ \sum_{\substack{1\leq z_1,\dots,z_{m_K-1} \leq d \\ \forall E\in \pi_l, \forall n,m\in E, z_n=z_m}} \sum_{M_1,\dots,M_k} a_{M_1}\dots a_{M_k} \nonumber \\
		&\quad\quad\quad\quad\quad\quad\quad\quad\quad\quad\quad\quad\quad\quad\quad\quad\quad\quad\quad\quad\quad\quad \tau^{\otimes l_K} \left( u^{l,K,T_p}_z\left(\left(\partial_{z_1}\otimes \cdots \otimes \partial_{z_{m_K-1}}\right)(M_1\dots M_k) \right) \right). \nonumber
	\end{align}
	However,
	\begin{align*}
		&\sum_{M_1,\dots,M_k} a_{M_1}\dots a_{M_k} \tau^{\otimes l_K} \left( u^{l,K,T_p}_z\left(\left(\partial_{z_1}\otimes \cdots \otimes \partial_{z_{m_K-1}}\right)(M_1\dots M_k) \right) \right) \\
		&= \sum_{M_1,\dots,M_k} \sum_{x=1}^{m_K-1} \sum_{r_1+\dots+r_x = m_K-1} \sum_{1\leq k_1<\dots< k_x\leq k} a_{M_1}\dots a_{M_k} \tau^{\otimes l_K} \Big( u^{l,K,T_p}_z\Big( M_1\dots M_{k_1-1} \\
		&\quad\quad\quad\quad\quad\quad \left(\partial_{z_1}\circ\dots\circ \partial_{z_{r_1}}M_{k_1}\right) M_{k_1+1}\dots M_{k_{x}-1} \left(\partial_{z_{r_1+\dots+r_{x-1}}}\circ\dots\circ\partial_{z_{m_K-1}} M_{k_{x}}\right) M_{k_{x}+1}\dots M_{k} \Big) \Big),
	\end{align*}
	where we used the following convention, for any $A_l,B_l,C_l\in\C\langle X_1,\dots X_d\rangle$,
	\begin{align*}
		&A_1 \left(B_1\otimes C_1\right) A_2 \left(B_2\otimes C_2\right) A_3 \dots A_{m_K-1} \left(B_{m_K-1}\otimes C_{m_K-1}\right) A_{m_K} \\
		&= (A_1B_1)\otimes(C_1A_2B_2)\otimes\dots \otimes(C_{m_K-2}A_{m_K-1}B_{m_K-1}) \otimes(C_{m_K-1}A_{m_K}).
	\end{align*}
	Let us remember that,
	\begin{align*}
		u^{l,K,T_p}_z\left( \bigotimes_{ s\in [1,m_K]}\ A_{s}\right) := \prod_{s\in [1,m_K]} u_{s}^{l,K}({A}_{s}(x)) \Lambda^{l,T_p,K}_{z,s},
	\end{align*}
	where $\Lambda^{l,T_p,K}_{z,s} \in B(\F)^{\otimes l_K}$ has norm smaller than one, and $u_{s}^{l,K} : \CC^*(x)\to \CC^*(x)^{\otimes l_K}$ is a map of the following form
	$$ u_{s}^{l,K}(a) = \id_{\CC^*(x)}^{\otimes n} \otimes a \otimes \id_{\CC^*(x)}^{\otimes l_K-n-1},$$
	for a given $n$ which depends on $s,l$ and $K$. Thus if we define 
	\begin{align*}
		v^{l,K,T_p}_z\left( \bigotimes_{ s\in [1,m_K]}\ P_{s}\right) := \prod_{s\in [1,m_K]} (\id_{\A}\otimes u_{s}^{l,K})(P_{s}(x)) \Lambda^{l,T_p,K}_{z,s},
	\end{align*}
	then,
	$$ \norm{v^{l,K,T_p}_z\left( \bigotimes_{ s\in [1,m_K]}\ P_{s}\right)} \leq \prod_{ s\in [1,m_K]}\ \norm{P_{s}(x)}, $$
	and with the convention of the theorem,
	\begin{align*}
		&\sum_{M_1,\dots,M_k} a_{M_1}\dots a_{M_k} \tau^{\otimes l_K} \left( u^{l,K,T_p}_z\left(\left(\partial_{z_1}\otimes \cdots \otimes \partial_{z_{m_K-1}}\right)(M_1\dots M_k) \right) \right) \\
		&= \sum_{M_1,\dots,M_k} \sum_{x=1}^{m_K-1} \sum_{r_1+\dots+r_x = m_K-1} \sum_{1\leq k_1<\dots< k_x\leq k} \\
		&\quad\quad\quad\quad \id_{\A}\otimes \tau^{\otimes l_K} \Biggl( v^{l,K,T_p}_z\Big( (a_{M_1}\otimes M_1) \dots (a_{M_{k_1-1}}\otimes M_{k_1-1}) \left(a_{M_{k_1}}\otimes \left(\partial_{i_1}\circ\dots\circ \partial_{i_{r_1}}M_{k_1}\right)M_{k_1}\right) \nonumber\\ 
		&\quad\quad\quad\quad\quad\quad\quad\quad\quad (a_{M_{k_1+1}}\otimes M_{k_1+1})\dots (a_{M_{k_{x}-1}}\otimes M_{k_{x}-1}) \left(a_{M_{k_x}}\otimes \left(\partial_{i_{r_1+\dots+r_{x-1}}}\circ\dots\circ\partial_{i_{m_K-1}} M_{k_{x}}\right)\right) \\
		&\quad\quad\quad\quad\quad\quad\quad\quad\quad\quad\quad\quad\quad\quad\quad\quad\quad\quad\quad\quad\quad\quad\quad\quad\quad\quad\quad\quad\quad (a_{M_{k_x+1}}\otimes M_{k_x+1})\dots (a_{M_{k}}\otimes M_{k}) \Big) \Biggl) \\
		&= \sum_{x=1}^{m_K-1} \sum_{r_1+\dots+r_x = m_K-1} \sum_{1\leq k_1<\dots< k_x\leq k} \id_{\A}\otimes\tau^{\otimes l_K} \Biggl(v^{l,K,T_p}_z\Big( P^{k_1-1} \left(\partial_{i_1}\circ\dots\circ \partial_{i_{r_1}} P \right) P^{k_2-k_1-1}\dots\nonumber \\ 
		&\quad\quad\quad\quad\quad\quad\quad\quad\quad\quad\quad\quad\quad\quad\quad\quad\quad\quad\quad\quad \dots P^{k_{x} - k_{x-1}-1} \left(\partial_{i_{r_1+\dots+r_{x-1}}}\circ \dots\circ\partial_{i_{m_K-1}}P \right) P^{k - k_{x}} \Big) \Biggl) \\
		&= \id_{\A}\otimes\tau^{\otimes l_K} \Big(v^{l,K,T_p}_z\Big( \partial_{z_1}\circ\dots\circ \partial_{z_{m_K-1}}\left(P^k\right) \Big) \Big).
	\end{align*}
	
	\noindent Consequently, one has that
	\begin{align*}
		&\E\left[\id_{\A}\otimes\ts_N\Big(P^k\left(X^N\right)\Big)\right] = (\id_{\A}\otimes\tau)\left(P(x)^k\right) + \sum_{p= 1}^{nk/4} \frac{1}{2^{p}N^{2p}} \int_{A_p} e^{-t_1-\cdots-t_{2p}} \sum_{l=1}^{c_{T_p}} \sum_{K\in \mathcal{K}_{4p+1}} \\
		&\quad\quad\quad\quad\quad\quad\quad\quad\quad\quad \sum_{\substack{1\leq z_1,\dots,z_{m_K-1} \leq d \\ \forall E\in \pi_l, \forall n,m\in E, z_n=z_m}} \id_{\A}\otimes\tau^{\otimes l_K} \Big(v^{l,K,T_p}_z\Big( \partial_{z_1}\circ\dots\circ \partial_{z_{m_K-1}}\left(P^k\right) \Big) \Big) dt_1\dots dt_{2p}.
	\end{align*}
	
\end{proof}

\subsection{Proof of Theorem \ref{popopohdvc}}

	We want to use Proposition \ref{lvleavalvlkd}. First, one has that
	\begin{align}
		\label{doicdom}
		&\norm{\id_{\A}\otimes\tau^{\otimes l_K} \left(v^{l,K,T_p}_z\Big( \left(\partial_{z_1}\circ\dots\circ \partial_{z_{m_K-1}}\left(P^k\right)\right) \Big) \right)} \\
		&\leq \sum_{x=1}^{m_K-1} \sum_{r_1+\dots+r_x = m_K-1} \sum_{1\leq k_1<\dots< k_x\leq k} C^{x} \Big( \norm{P^{k_1-1}(x)} \sup_{M \text{ monomial}}\norm{a_{M}}  \nonumber \\
		&\quad\quad\quad\quad\quad\quad\quad\quad\quad\quad\quad\quad\quad \norm{P^{k_2-k_1-1}(x)}\dots  \norm{P^{k_{x} - k_{x-1}-1}(x)} \sup_{M \text{ monomial}}\norm{a_{M}} \norm{P^{k - k_{x}}(x)} \Big) \nonumber \\
		&\leq \sum_{x=1}^{m_K-1} \sum_{r_1+\dots+r_x = m_K-1} \sum_{1\leq k_1<\dots< k_x\leq k} C^{x} \norm{P(x)}^{k-x} \sup_{M \text{ monomial}} \norm{a_{M}}^x, \nonumber
	\end{align}
	where $C$ is a constant which depends only on $n$ and $d$. Indeed a polynomial of degree $n$ in $d$ variables is the sum of at most $\frac{d^{n+1}-1}{d-1}$ monomials, besides the norm of the differential of a monomial evaluated in free semicircular variables of degree $n$ is bounded by,
	$$\sup_{1\leq l\leq n} \frac{n!}{(n-l)!}2^{n-l} = 2 n!.$$
	Thus, coupled with the help of Lemma \ref{xsdfgyujk}, we get that there exists a constant $C$ which depends only on $n$ and $d$ such that,
	\begin{align}
		\label{dlmvsomv}
		&\norm{\id_{\A}\otimes\tau^{\otimes l_K} \left(v^{l,K,T_p}_z\Big( \left(\partial_{z_1}\circ\dots\circ \partial_{z_{m_K-1}}\left(P^k\right)\right) \Big) \right)} \\
		&\leq C^{m_K-1} \norm{P(x)}^k \sum_{x=1}^{m_K-1} \sum_{r_1+\dots+r_x = m_K-1} \sum_{1\leq k_1<\dots< k_x\leq k}1. \nonumber
	\end{align}
	
	\noindent Besides, the cardinal of the set
	$$ \{(x,r_1,\dots,r_x,k_1,\dots,k_x)\ |\ 1\leq k_1<\dots< k_x\leq k,\ r_1+\dots+r_x = m_K-1, x\in [1,m_K-1] \} $$
	is equal to the one of 
	$$ \{(k_1,\dots,k_{m_K-1})\ |\ 1\leq k_1\leq\dots\leq k_{m_K-1}\leq k\} $$
	since the map which to $(x,r_1,\dots,r_x,k_1,\dots,k_x)$ associates the $m_K-1$-tuple where each $k_i$ is repeated $r_i$ times is injective. Hence let us prove by induction that 
	$$ \#\{1\leq k_1\leq\dots\leq k_l\leq k\} \leq \frac{(k+l-1)^l}{l!}.$$
	Indeed it is true for $l=1$, and if it is true for a given $l$, then
	\begin{align*}
		\#\{(k_1,\dots,k_{l+1})\ |\ 1\leq k_1\leq\dots\leq k_{l+1}\leq k\} &= \sum_{k_{l+1}=1}^k \#\{(k_1,\dots,k_l)\ |\ 1\leq k_1\leq\dots\leq k_l\leq k_{l+1}\} \\
		&\leq \sum_{k_{l+1}=1}^k \frac{(k_{l+1}+l-1)^l}{l!} \\
		&\leq \int_1^{k+1} \frac{(k_{l+1}+l-1)^l}{l!} dk_{l+1} \\
		&\leq \frac{(k+l)^{l+1}}{(l+1)!}.
	\end{align*}
	Thus we get that
	\begin{align*}
		&\norm{\id_{\A}\otimes\tau^{\otimes l_K} \left(v^{l,K,T_p}_z\Big( \left(\partial_{z_1}\circ\dots\circ \partial_{z_{m_K-1}}\left(P^k\right)\right) \Big) \right)} \leq C^{m_K-1} \norm{P(x)}^k \frac{(k+m_K-2)^{m_K-1}}{(m_K-1)!}.
	\end{align*}
	But since $P^k$ is a polynomial of degree $nk$ that we differentiate $m_K-1$ times, one can always assume that $m_K-1\leq nk$, otherwise this term will be equal to $0$. Hence by replacing $C$ with $C(n+1)$, one can always assume that
	\begin{align}
		\label{xdrfdyuijhjghj}
		&\norm{\id_{\A}\otimes\tau^{\otimes l_K} \left(v^{l,K,T_p}_z\Big( \left(\partial_{z_1}\circ\dots\circ \partial_{z_{m_K-1}}\left(P^k\right)\right) \Big) \right)} \leq C^{m_K-1} \norm{P(x)}^k \frac{k^{m_K-1}}{(m_K-1)!}.
	\end{align}
	Thus thanks to Proposition \ref{lvleavalvlkd}, we get that
	\begin{align*}
		&\norm{\E\left[\id_{\A}\otimes\ts_N\Big(P^k\left(X^N\right)\Big)\right]} \\
		&\leq \norm{P(x)}^k \left(1+ \sum_{p=1}^{nk/4} \frac{(4p)!}{N^{2p}} \int_{A_p} e^{-t_1-\cdots-t_{2p}}dt_1\dots dt_{2p} \sum_{g\geq 4p+1}\ \sum_{K\in \mathcal{K}_{4p+1} \text{ such that } m_K=g } d^{g-1} C^{g-1} \frac{k^{g-1}}{(g-1)!}\right) \\
		&\leq \norm{P(x)}^k \left(1+ \sum_{p=1}^{nk/4} \frac{(4p)!}{N^{2p}} \int_{A_p} e^{-t_1-\cdots-t_{2p}}dt_1\dots dt_{2p}\ \#\mathcal{K}_{4p+1} \sum_{g= 4p+1}^{16p-2} d^{g-1} C^{g-1} \frac{k^{g-1}}{(g-1)!} \right),
	\end{align*}
	where in the last line we used Lemma \ref{ssxdcffvg} which insures that $m_K$ is smaller than $4(4p+1)-6=16p-2$. Indeed, if $K\in \mathcal{K}_{4p+1}$, then with $m=\#\widehat{K}$, one can view $K$ as an element of $\K_{4p+1-m}$, and hence
	$$ m_K \leq m + 4(4p+1-m) -6 \leq 16p-2.$$
	Besides
	\begin{align*}
		\int_{A_{p}} e^{- \sum_{r=1}^{2p} t_{r}} dt_1\dots dt_{2p} &\leq \int\limits_{0\leq t_2\leq t_4\leq \dots \leq t_{2p}} \prod_{r=1}^{p} e^{-t_{2r}}\quad dt_2 dt_4\dots dt_{p} = \frac{1}{p!}.
	\end{align*}
	So thanks once again to Lemma \ref{ssxdcffvg}, there exists a constant $C$ which only depends on $n$ and $d$ such that,
	\begin{align*}
		\norm{\E\left[\id_{\A}\otimes\ts_N\Big(P^k\left(X^N\right)\Big)\right]} &\leq \norm{P(x)}^k \left(1+\sum_{p= 1}^{nk/4} \frac{1}{N^{2p}} \sum_{g= 4p}^{16p-3} C^{g} \frac{(4p)! k^{g}}{p!g!}\right).
	\end{align*}
	Hence,	
	\begin{align*}
		\norm{\E\left[\id_{\A}\otimes\ts_N\Big(P^k\left(X^N\right)\Big)\right]} &\leq \norm{P(x)}^k \left(1+\sum_{p= 1}^{nk/4} \frac{1}{N^{2p}} \sum_{g= 4p}^{16p-3} C^{g} \frac{k^{g}}{(g-3p)!}\right) \\
		&\leq \norm{P(x)}^k \left(1+\sum_{p= 1}^{nk/4} \sum_{g= 0}^{12p-3} \frac{(Ck)^{g+4p}}{N^{2p}(g+p)!}\right) \\
		&\leq \norm{P(x)}^k \sum_{p=0}^{nk/4} \sum_{g= 0}^{12p} \frac{(Ck)^{g+4p}}{N^{2p}(g+p)!}.
	\end{align*}
	Let us now remark that for $p\leq nk/4$, the set $[0,12p]$ is included in the set $\{\lfloor \frac{4s p}{n k} \rfloor\ |\ s\in [0,3nk] \}$. Indeed, given $l\in [0,12p]$, we define $s=\lfloor \frac{nkl}{4p} \rfloor$, thus
	$$ \frac{nkl}{4p} \leq s < \frac{nkl}{4p}+1.$$
	Consequently,
	$$ l \leq \frac{4ps}{nk} < l+\frac{4p}{nk},$$
	and since we assumed that $\frac{4p}{nk}\leq 1$, this implies that $l=\lfloor \frac{4s p}{n k} \rfloor$, hence the conclusion. As a consequence,
	\begin{align*}
		\norm{\E\left[\id_{\A}\otimes\ts_N\Big(P^k\left(X^N\right)\Big)\right]} &\leq \norm{P(x)}^k \sum_{p= 0}^{nk/4} \sum_{s =0}^{3nk} \frac{(Ck)^{4p+\lfloor \frac{4s p}{n k} \rfloor}}{N^{2p} \left(\lfloor \frac{4s p}{n k}\rfloor+p \right)!}\\
		&\leq \norm{P(x)}^k \sum_{s =0}^{3nk} \sum_{p= 0}^{nk/4} \left(k \times C^{1+\frac{3p}{\lfloor \frac{4s p}{n k}\rfloor+p}} \left(\frac{k^3}{N^2}\right)^{\frac{p}{\lfloor \frac{4s p}{n k}\rfloor+p}}\right)^{\lfloor \frac{4s p}{n k}\rfloor+p} \frac{1}{(\lfloor \frac{4s p}{n k}\rfloor+p)!}.
	\end{align*}
	Note that $0\leq \lfloor \frac{4s p}{n k}\rfloor \leq 12p$, hence
	$$  C^{1+\frac{3p}{\lfloor \frac{4s p}{n k}\rfloor+p}} \leq \max\left( C^{4},C^{16/13}\right),$$
	$$ \left(\frac{k^3}{N^2}\right)^{\frac{p}{\lfloor \frac{4s p}{n k}\rfloor+p}} \leq \max\left( \left(\frac{k^3}{N^2}\right), \left(\frac{k^3}{N^2}\right)^{1/13} \right).$$
	Consequently, one can find a constant $C_{n,d}$ which only depends on $n$ and $d$ such that
	\begin{align*}
		\norm{\E\left[\id_{\A}\otimes\ts_N\Big(P^k\left(X^N\right)\Big)\right]} &\leq \norm{P(x)}^k \sum_{s =0}^{3nk} \sum_{p= 0}^{nk/4} \left(k \times C_{n,d} \max\left( \left(\frac{k^3}{N^2}\right), \left(\frac{k^3}{N^2}\right)^{1/13} \right) \right)^{\lfloor \frac{4s p}{n k}\rfloor+p} \frac{1}{(\lfloor \frac{4s p}{n k}\rfloor+p)!}.
	\end{align*}
	Finally, note that the sequence $p\mapsto p + \lfloor \frac{4s p}{n k}\rfloor$ is strictly increasing, and since for any strictly increasing sequence $(k_p)_{p\geq 0}$ one has that for any $X\geq 0$,
	$$ \sum_{p\geq 0} \frac{X^{k_p}}{k_p!} \leq e^X,$$
	we get that
	\begin{align*}
		\norm{\E\left[\id_{\A}\otimes\ts_N\Big(P^k\left(X^N\right)\Big)\right]} &\leq \norm{P(x)}^k \sum_{s =0}^{3nk} \exp\left( C_{n,d} k \max\left( \left(\frac{k^3}{N^2}\right)\right), \left(\frac{k^3}{N^2}\right)^{1/13} \right) \\
		&\leq (3nk+1) \norm{P(x)}^k \exp\left( C_{n,d} k \max\left( \left(\frac{k^3}{N^2}\right)\right), \left(\frac{k^3}{N^2}\right)^{1/13} \right).
	\end{align*}
	And thus
	\begin{align*}
		\E\left[ \norm{P(X^N)}_{L^{2k}} \right] &\leq \E\left[ \tau\otimes\ts_N\left(\left(P(X^N)^*P(X^N)\right)^{k}\right) \right]^{1/2k} \\
		&= \tau\left(\E\left[ \id_{\A}\otimes\ts_N\left(\left(P(X^N)^*P(X^N)\right)^{k}\right) \right]\right)^{1/2k} \\
		&\leq \norm{\E\left[ \id_{\A}\otimes\ts_N\left(\left(P(X^N)^*P(X^N)\right)^{k}\right) \right]}^{1/2k} \\
		&\leq (3nk+1)^{\frac{1}{2k}} \norm{P(x)} \exp\left( C_{2n,d} \max\left( \left(\frac{k^3}{N^2}\right)\right), \left(\frac{k^3}{N^2}\right)^{1/13} \right).
	\end{align*}

\subsection{Proof of Theorem \ref{detcase}}

The proof of Theorem \ref{detcase} will be nearly the same as the one of Theorem \ref{popopohdvc} with the one technical difference being that we need to prove an equivalent to Proposition \ref{lvleavalvlkd} which covers the case of deterministic matrices. We do so in Proposition \ref{lvleavalvlkd2}. The constant $C$ in Equation \eqref{doicdom} will also depend on $\sup_j \norm{Z_j^N}$. Besides, one cannot use Lemma \ref{xsdfgyujk} to bound 
$$\frac{\sup_{M \text{ monomial}}\norm{a_{M}}}{\norm{P(x,Z^N)}}.$$
Thus starting with Equation \eqref{dlmvsomv}, we replace the constant $C$ by $\frac{C'}{\norm{P(x,Z^N)}}$, where $C'$ is a constant which depends on $P$ as well as $\sup_j \norm{Z_j^N}$, hence the conclusion.

\begin{prop}
	
	\label{lvleavalvlkd2}
	
	Given the following objects,
	\begin{itemize}
		\item $X^{\mathfrak{n}N}=\left(X_1^{\mathfrak{n}N},\dots,X_d^{\mathfrak{n}N}\right)$ a $d$-tuple of GUE random matrices of size $\mathfrak{n}N$,
		\item $Z^N$ a $q$-tuple of deterministic matrices of size $N$.
		\item $P\in\A\langle X_1,\dots,X_d,Z_1,\dots,Z_{q}\rangle$ a polynomial of degree $n$,
	\end{itemize}
	With the notations of previous lemmas,
	\begin{align*}
		\E\Big[&\id_{\A}\otimes\ts_N\Big(P^k\left(X^{\mathfrak{n}N},I_{\mathfrak{n}}\otimes Z^N\right) \Big)\Big] = (\id_{\A}\otimes\tau)\left(P(x,Z^N)^k\right) \\
		&+ \sum_{p= 1}^{nk/4} \frac{1}{2^{p}\mathfrak{n}^{2p}} \int_{A_p} e^{-t_1-\cdots-t_{2p}} \sum_{l=1}^{c_{T_p}} \sum_{K\in \mathcal{K}_{4p+1}}\ \sum_{\substack{1\leq z_1,\dots,z_{m_K-1} \leq d \\ \forall E\in \pi_l, \forall n,m\in E, z_n=z_m}} \\
		&\quad\quad\quad\quad\quad\quad\quad\quad\quad\quad\quad \id_{\A}\otimes\ts_N\otimes \tau^{\otimes l_K} \left(w^{l,K,T_p}_z\Big( \left(\partial_{z_1}\circ\dots\circ \partial_{z_{m_K-1}}\left(P^k\right)\right) \Big) \right)dt_1\dots dt_{2p}, \nonumber
	\end{align*}
	where $c_{T_p}\leq (4p)!$, $\partial_{z_1}\circ\dots\circ \partial_{z_{m_K-1}}\left(P^k\right)$ is defined as in Proposition \ref{lvleavalvlkd} and
	$$w^{l,K,T_p}_z: \A\langle X_1,\dots,X_d,Z_1,\dots,Z_{q}\rangle^{\otimes m_K} \to \A\otimes\M_N(\C)\otimes \CC^*(x)^{\otimes l_K}$$
	is a linear map such that for all $P_s\in \A\langle X_1,\dots,X_d,Z_1,\dots,Z_{q}\rangle$,
	$$ \norm{w^{l,K,T_p}_z\left( \bigotimes_{ s\in [1,m_K]}\ P_{s}\right)} \leq \prod_{ s\in [1,m_K]}\ \norm{P_{s}(x,Z^N)}.$$
	
\end{prop}

\begin{proof}
	
	To begin with, if we write $P=\sum_M a_M\otimes M$, one has that
	$$ P^k = \sum_{M_1,\dots,M_k} a_{M_1}\dots a_{M_k} \otimes M_1\dots M_k.$$
	Thus thanks to Lemma \ref{3apparition}, one has that
	\begin{align*}
		&\E\left[\id_{\A}\otimes\ts_N\Big(P^k\left(X^{\mathfrak{n}N},I_{\mathfrak{n}}\otimes Z^N\right)\Big)\right] \\
		&= \sum_{p= 0}^{nk/4} \frac{1}{N^{2p}} \int_{A_p} e^{-t_1-\cdots-t_{2p}} \sum_{M_1,\dots,M_k} a_{M_1}\dots a_{M_k} \\
		&\quad\quad\quad\quad\quad\quad\quad\quad\quad\quad\quad\quad\quad\quad\quad\quad \tau_{\mathfrak{n}N}\Big( \left(L^{{T}_p} \dots L^{{T}_1}\right)(M_1\dots M_k) \left(x^{T_p},I_{\mathfrak{n}}\otimes Z^N\right) \Big)\  dt_1\dots dt_{2p}.
	\end{align*}
	However, as non-commutative random variables, $I_{\mathfrak{n}}\otimes Z^N$ and $Z^N$ have the same distribution. Thus since $I_{\mathfrak{n}}\otimes Z^N$ and $x^{T_p}$ are free, $\left(x^{T_p},I_{\mathfrak{n}}\otimes Z^N\right)$ and $\left(x^{T_p},Z^N\right)$ also have the same distribution. Thus
	\begin{align}
		\label{nknsc}
		&\tau_{\mathfrak{n}N}\Big( \left(L^{{T}_p} \dots L^{{T}_1}\right)(M_1\dots M_k) \left(x^{T_p},I_{\mathfrak{n}}\otimes Z^N\right) \Big) = \tau_N\Big( \left(L^{{T}_p} \dots L^{{T}_1}\right)(M_1\dots M_k) \left(x^{T_p},Z^N\right) \Big).
	\end{align}
	And thanks to Lemma \ref{poiuytrew}, with the same notation, one has that with $Q=M_1\dots M_k$,
	\begin{align}
		\label{dmvclsmcsaa}
		&\tau_N\Big( \left(L^{p}_{s_p} \dots L^{1}_{s_1}\right)(Q)\left(x^{T_p},Z^N\right) \Big) \\
		&= \frac{1}{2^p} \sum_{l=1}^{\widehat{c}_{s_1,\dots,s_p}} \sum_{\substack{1\leq z_1,\dots,z_{4p} \leq d \\ z_{\widehat{\pi}_l(1)} = z_{\widehat{\pi}_l(2)}, \dots, z_{\widehat{\pi}_l(4p-1)} = z_{\widehat{\pi}_l(4p)}}} \tau_N\Big( \ev_l\circ\widehat{\sigma}_l\circ m\circ\left(\partial_{z_1}\otimes \cdots \otimes \partial_{z_{4p}}\right)(Q)\left(x^{T_p},Z^N\right) \Big), \nonumber
	\end{align}
	Besides one can find a permutation $\sigma_l\in \mathbb{S}_{4p+1}$ and $d$-tuples of free semicircular variables $x^1,\dots,x^{4p+1}$ whose covariance depends only on $T_p$ such that for any $A_i\in \C\langle X_1,\dots,X_d,Z_1,\dots,Z_{q}\rangle$,
	$$ \tau_N\left(\left(\ev_l\circ\widehat{\sigma}_l\circ m\circ\left(A_1\otimes \cdots \otimes A_{4p+1}\right)\right)(x^{T_p},Z^N)\right) = \tau_N\left(A_{\sigma_l(1)}(x^1,Z^N) \cdots A_{\sigma_l(4p+1)}(x^{4p+1},Z^N)\right). $$
	But thanks to Proposition \ref{dojvdsomv}, one has that with $x^N$ as in Definition \ref{eiucskc},
	$$ \tau_N\left(\left(\ev_l\circ\widehat{\sigma}_l\circ m\circ\left(A_1\otimes \cdots \otimes A_{4p+1}\right)\right)(x^{T_p},Z^N)\right) = \tau_N\left(A_{\sigma_l(1)}(x^{N,1},Z^N) \cdots A_{\sigma_l(4p+1)}(x^{N,4p+1},Z^N)\right). $$
	Hence, assuming now that $(A_i)_{1\leq i\leq 4p+1}$ are monomials, one can write $x_j^{s,N} = \sum_{1\leq i\leq N^2} E_i\otimes x_i^{j,s}$, where $(E_i)_{1\leq i\leq N^2}$ is the union of the families 
	$$ \left(\frac{E_{u,u}}{\sqrt{N}}\right)_{1\leq u\leq N},\quad \left(\frac{E_{u,v}+E_{u,v}}{\sqrt{2N}}\right)_{1\leq u<v\leq N},\quad \left(\i \frac{E_{u,v}-E_{u,v}}{\sqrt{2N}}\right)_{1\leq u<v\leq N},$$
	where $E_{u,v}$ is the matrix of size $N$ whose only non zero coefficient is the one on line $u$ and row $v$. Besides $(x_i^{j,s})_{1\leq i\leq N^2}$ is a free semicircular system, and the correlation between $x_i^{j,s}$ and $x_i^{j',s'}$ is the same as the one between $x_j^s$ and $x_j^s$. Thus we have that
	\begin{align*}
		&\tau_N\left(\left(\ev_l\circ\widehat{\sigma}_l\circ m\circ\left(A_1\otimes \cdots \otimes A_{4p+1}\right)\right)(x^{T_p},Z^N)\right) \\
		&= \sum_{\substack{1\leq i_{a,s}\leq N^2 \\ 1\leq a\leq \deg A_s, 1\leq s\leq 4p+1 }} \ts_N\left(B_{\sigma_l(1)}^{i_{\sigma_l(1)}} \cdots B_{\sigma_l(4p+1)}^{i_{\sigma(4p+1)}}\right) \tau\left(A_{\sigma_l(1)}^{i_{\sigma_l(1)}}(x^{N,1}) \cdots A_{\sigma_l(4p+1)}^{i_{\sigma(4p+1)}}(x^{N,4p+1})\right),
	\end{align*}
	where $A_s^{i_s}(x^{N,s})$ is the monomial $A_s$ evaluated in a specific way, if the $a$-th variable of this monomial belongs to $\{Z_1,\dots,Z_{q}\}$, then we evaluate it in $1$, if on the contrary it belongs to $\{X_1,\dots,X_d\}$, then we evaluate it in $x_{i_{a,s}}^{j,s}$. On the contrary, $B_s^{i_s}$ is the monomial $A_s$ evaluated in the following way, if the $a$-th variable of this monomial belongs to $\{Z_1,\dots,Z_{q}\}$, then we evaluate it in $Z^N$, if on the contrary it belongs to $\{X_1,\dots,X_d\}$, then we evaluate it in $E_{i_{a,s}}$.
	
	Next, let us focus on the trace on the space of matrices. We are going to prove that for the permutations $\sigma_l$ that we consider, one can find $U_1,\dots,U_{2p}$ independent Haar unitary matrices of size $N$ and monomials $M_1^l,\dots,M_{4p+1}^l$ such that
	$$ \ts_N\left(B_{\sigma_l(1)}^{i_{\sigma_l(1)}} \dots B_{\sigma_l(4p+1)}^{i_{\sigma(4p+1)}}\right) = N^{2p}\ \E\left[\ts_N\left( B_1^{i_1} M_1^l(U,U^*) \dots B_{4p+1}^{i_{4p+1}} M_{4p+1}^l(U,U^*) \right)\right]. $$
	We will proceed by induction. The case where $p=1$ is simply a consequence of Lemma 5.1 of \cite{sept} combined with the definition of $L_1^1$ in Equation \eqref{xfgukokgc}. Let us now consider $p>1$, then given the definition by induction of $\widehat{\sigma}_l$ in the proof of Lemma \ref{poiuytrew}, by using once again Lemma 5.1 of \cite{sept}, one can write 
	$$ \ts_N\left(B_{\sigma_l(1)}^{i_{\sigma_l(1)}} \dots B_{\sigma_l(4p+1)}^{i_{\sigma(4p+1)}}\right) = N^2\ \E\left[\tr_N\left( C_1 U C_2 V C_3 U^* C_4 V^* \right)\right], $$
	and one can find $\nu_l\in \S_{4p-3}$ such that by traciality, one can rewrite $C_1 U C_2 V C_3 U^* C_4 V^*$ as a product $\widetilde{B}_{\nu_l(1)}\dots \widetilde{B}_{\nu_l(4p-3)}$. Hence the conclusion by induction.

	Besides, thanks to Lemma \ref{sdtyuikmn}, one also have that
	\begin{align*}
		&\tau\left(A_{\sigma_l(1)}^{i_{\sigma_l(1)}}(x^{N,1}) \cdots A_{\sigma_l(4p+1)}^{i_{\sigma_l(4p+1)}}(x^{N,4p+1})\right) \\
		&= \sum_{K\in \mathcal{K}_{4p+1}}\ \sum_{z\in [1,dN^2]^{\widetilde{K}}}\ H^{K}_z\left( \bigotimes_{s\in [1,4p+1]}\ \left(\bigotimes_{t\in C_s^{**}(K)} \partial_{z_{s,t}}\right)\left(A_{\sigma_l(s)}^{i_{\sigma_l(s)}}\right) \right).
	\end{align*}
	And thanks to Lemma \ref{nkytrdds}, one has that
	\begin{align*}
		&\tau\left(A_{\sigma_l(1)}^{i_{\sigma_l(1)}}(x^{N,1}) \cdots A_{\sigma_l(4p+1)}^{i_{\sigma_l(4p+1)}}(x^{N,4p+1})\right) \\
		&= \sum_{K\in \mathcal{K}_{4p+1}}\ \sum_{z\in [1,dN^2]^{\widetilde{K}}}\ \tau^{\otimes l_K} \left(\rule{0cm}{1cm}\right. u^{\sigma_l,K,T_p}_z \left(\rule{0cm}{1cm}\right. \bigotimes_{a\in [1,4p+1]} \ \left(\rule{0cm}{1cm}\right.\bigotimes_{b\in \sigma_l\left(C_{\sigma_l^{-1}(a)}^{**}(K)\right)} \partial_{z_{\sigma_l^{-1}(a),\sigma_l^{-1}(b)}}\left.\rule{0cm}{1cm}\right) \left(A_{a}^{i_a}\right) \left.\rule{0cm}{1cm}\right) \left.\rule{0cm}{1cm}\right).
	\end{align*}
	Consequently, we get that
	\begin{align*}
		&\tau_N\left(\left(\ev_l\circ\widehat{\sigma}_l\circ m\circ\left(A_1\otimes \cdots \otimes A_{4p+1}\right)\right)(x^{T_p},Z^N)\right) \\
		&= N^{2p} \sum_{K\in \mathcal{K}_{4p+1}} \sum_{z\in [1,d]^{\widetilde{K}}} \ts_N\otimes\tau^{\otimes l_K} \left(\rule{0cm}{1cm}\right. w^{\sigma_l,K,T_p}_z \left(\rule{0cm}{1cm}\right. \bigotimes_{a\in [1,4p+1]} \left(\rule{0cm}{1cm}\right.\bigotimes_{b\in \sigma_l\left(C_{\sigma_l^{-1}(a)}^{**}(K)\right)} \partial_{z_{\sigma_l^{-1}(a),\sigma_l^{-1}(b)}}\left.\rule{0cm}{1cm}\right) \left(A_a\right) \left.\rule{0cm}{1cm}\right) \left.\rule{0cm}{1cm}\right),
	\end{align*}
	where $w^{\sigma_l,K,T_p}_z: \A_{p,q}^{\otimes m_K} \mapsto \M_N(\C)\otimes\CC^*(x)^{\otimes l_K}$ is defined by the following formula,
	\begin{align*}
		&w^{\sigma_l,K,T_p}_z: \left(\bigotimes_{a\in [1,4p+1]}\ \left(\bigotimes_{b\in \sigma_l\left(C_{\sigma_l^{-1}(a)}^{*}(K)\right)} B_{a,b}\right)\right) \\
		&\mapsto \E\left[\rule{0cm}{1cm}\right. \prod_{a\in [1,4p+1]} \left(\prod_{b\in \sigma_l\left(C_{\sigma_l^{-1}(a)}^{**}(K)\right),\ a<b} (\id_{\M_N(\C)}\otimes u_{a,b}^{\sigma_l,K}) \left(B_{a,b}(x^N,Z^N) L_{z_{\sigma_l^{-1}(a),\sigma_l^{-1}(b)}} \Delta_{\sigma_l^{-1}(a),\sigma_l^{-1}(b)} \right)\right) \\
		&\quad\quad\quad\quad\quad\ \quad\times \left(\prod_{b\in \sigma_l\left(C_{\sigma_l^{-1}(a)}^{*}(K)\right),\ a=b} (\id_{\M_N(\C)}\otimes u_{a,b}^{\sigma_l,K}) \left(B_{a,b}(x^N,Z^N) \right)\right) \\
		&\quad\quad\quad\quad\quad\ \quad\times \left(\prod_{b\in \sigma_l\left(C_{\sigma_l^{-1}(a)}^{**}(K)\right),\ b>a} (\id_{\M_N(\C)}\otimes u_{a,b}^{\sigma_l,K})\left( R_{z_{\sigma_l^{-1}(a),\sigma_l^{-1}(b)}} \sharp B_{a,b}(x^N,Z^N)\right) \right) \\
		&\quad\quad\quad\quad\quad\ \quad\times \left(\prod_{b\in \sigma_l\left(C_{\sigma_l^{-1}(a)}^{*}(K)\right),\ b= \sigma_l\left((\sigma_l^{-1}(a))^{-}\right)} (\id_{\M_N(\C)}\otimes u_{a,b}^{\sigma_l,K}) \left(B_{a,b}(x^N,Z^N) \Delta_{\sigma_l^{-1}(a),\sigma_l^{-1}(b)} \right)\right) \\
		&\quad\quad\quad\quad\quad\ \quad\times M_a^l(U,U^*)\otimes \id_{\CC^*(x)}^{\otimes l_K} \left.\rule{0cm}{1cm}\right],
	\end{align*}
	with 
	$$ L_{z} = \sum_{1\leq i\leq N^2} E_i\otimes \left(l_i^{z}\right)^*,\quad\quad R_{z} = \sum_{1\leq i\leq N^2} E_i\otimes r_i^{z}, $$
	$$ \text{for all } S=\sum_i E_i\otimes A_i, T=\sum_j E_j\otimes B_j,\quad  S\sharp T = \sum_{i,j} E_jE_i\otimes A_iB_j. $$
	Note in particular that
	$$ \norm{L_{z}} = \norm{L_zL_z^*}^{1/2} = \norm{\sum_{1\leq i,j\leq N^2} E_iE_j^*\otimes \left(l_i^{z}\right)^*  l_i^{z} }^{1/2} = \norm{\sum_{1\leq i\leq N^2} E_iE_i^* }^{1/2} = 1, $$
	and
	\begin{align*}
		\norm{R_z\sharp T} &= \norm{(R_z\sharp T)^*R_z\sharp T}^{1/2} \\
		&= \norm{\sum_{1\leq i,j,k,l\leq N^2} E_i^*E_l^* E_kE_j \otimes A_l^* r_i^*r_j B_k}^{1/2} \\
		&= \norm{\sum_{1\leq i\leq N^2} E_i^*\otimes 1 T^*T E_i\otimes 1}^{1/2} \\
		&= \norm{ \ts_N\otimes\id_{\CC^*(x)}( T^*T)}^{1/2} \leq \norm{T}.
	\end{align*}
	Hence one has that,
	$$ \norm{w^{\sigma_l,K,T_p}_z: \left(\bigotimes_{a\in [1,4p+1]}\ \left(\bigotimes_{b\in \sigma_l\left(C_{\sigma_l^{-1}(a)}^{*}(K)\right)} B_{a,b}\right)\right)} \leq \prod_{a\in [1,4p+1]}\ \prod_{b\in \sigma_l\left(C_{\sigma_l^{-1}(a)}^{*}(K)\right)} \norm{B_{a,b}(x,Z^N)}. $$
	Thus by using Equation \eqref{nknsc} and \eqref{dmvclsmcsaa}, we get that 
	\begin{align*}
		\label{nknsc}
		&\tau\Big( \left(L^{{T}_p} \dots L^{{T}_1}\right)(M_1\dots M_k) \left(x^{T_p},I_n\otimes Z^N\right) \Big) \\
		&= \frac{N^{2p}}{2^p} \sum_{l=1}^{\widehat{c}_{s_1,\dots,s_p}} \sum_{\substack{1\leq z_1,\dots,z_{4p} \leq d \\ z_{\widehat{\pi}_l(1)} = z_{\widehat{\pi}_l(2)}, \dots, z_{\widehat{\pi}_l(4p-1)} = z_{\widehat{\pi}_l(4p)}}} \sum_{K\in \mathcal{K}_{4p+1}}\ \sum_{z\in [1,d]^{\widetilde{K}}}\\
		&\quad\quad \ts_N\otimes\tau^{\otimes l_K} \left(\rule{0cm}{1cm}\right. w^{\sigma_l,K,T_p}_z \left(\rule{0cm}{1cm}\right. \bigotimes_{a\in [1,4p+1]} \ \left(\rule{0cm}{1cm}\right.\bigotimes_{b\in \sigma_l\left(C_{\sigma_l^{-1}(a)}^{**}(K)\right)} \partial_{z_{\sigma_l^{-1}(a),\sigma_l^{-1}(b)}}\left.\rule{0cm}{1cm}\right) \left(\partial_{z_1}\otimes \cdots \otimes \partial_{z_{4p}}\right)(Q) \left.\rule{0cm}{1cm}\right) \left.\rule{0cm}{1cm}\right).
	\end{align*}
	From there on, the proof follows as in the one of Proposition \ref{lvleavalvlkd}, starting with Equation \eqref{sdicoscs}.
	
\end{proof}

\subsection{Proof of Theorem \ref{1strongconv}}

	Let us start with the case where $d'=0$, i.e. when we only have GUE random matrices. Thanks to Proposition 4.6 of \cite{un}, in order to prove that almost surely the family $(X^N\otimes I_{M_N}, I_N\otimes Y^{M_N})$ converges strongly in distribution towards $(x\otimes 1, 1\otimes y)$, it is sufficient to show that for any polynomial $P$, 
	$$ \lim\limits_{N\to\infty} \E\left[\norm{P(X^N\otimes I_{M_N}, I_N\otimes Y^{M_N})}\right] = \norm{P(x\otimes 1, 1\otimes y)}.$$
	Thanks to Theorem 5.4.5 of \cite{alice}, if $h$ is a continuous function taking positive values on the interval $\left(\norm{PP^*(x\otimes 1, 1\otimes y)}-\varepsilon, \infty \right)$ and taking value $0$ elsewhere, then 
	$$\frac{1}{MN}\tr_{MN}(h(PP^*(X^N\otimes I_{M_N}, I_N\otimes Y^{M_N})))$$
	converges almost surely towards $\tau_{\A}\otimes_{\min}\tau_{\B} (h(PP^*(x\otimes 1, 1\otimes y)))$ which is positive since $\tau_{\A}\otimes_{\min}\tau_{\B}$ is faithful thanks to Lemma 4.1.8 from \cite{ozabr} and $h$ is non-negative and in particular positive on a subset of the spectrum of $PP^*(x\otimes 1, 1\otimes y)$. Thus for any $\varepsilon>0$, for $N$ large enough,
	$$ \norm{PP^*(X^N\otimes I_{M_N}, I_N\otimes Y^{M_N})} \geq \norm{PP^*(x\otimes 1, 1\otimes y)} - \varepsilon .$$
	As a consequence,  almost surely,
	\begin{equation*}
		\liminf_{N\to \infty} \norm{P(X^N\otimes I_{M_N}, I_N\otimes Y^{M_N})} \geq \norm{P(x\otimes 1, 1\otimes y)} .
	\end{equation*}
	And thanks to Fatou's Lemma, we deduce
	$$ \liminf_{N\to \infty}\ \E\left[ \norm{P(X^N\otimes I_{M_N}, I_N\otimes Y^{M_N})} \right] \geq \norm{P(x\otimes 1, 1\otimes y)}. $$
	As for the other direction, we have thanks to Theorem \ref{popopohdvc} that for any $k\geq 0$, with $P$ a polynomial of degree $n$,
	\begin{align*}
		\E\left[\norm{P\left(X^N\otimes I_{M_N}, I_N\otimes Y^{M_N}\right)}\right] &\leq \E\left[\tr_{M_N}\otimes\tr_N\left(\left(P^*P\left(X^N\otimes I_{M_N}, I_N\otimes Y^{M_N}\right)^k\right)\right)^{\frac{1}{2k}}\right] \\
		&\leq (NM_N)^{1/2k} \E\left[\norm{\left(P^*P\left(X^N\otimes I_{M_N}, I_N\otimes Y^{M_N}\right)^k\right)}_{L^{2k}}\right] \\
		&\leq (NM_N)^{1/2k} (3nk+1)^{\frac{1}{2k}} \norm{P(x\otimes I_{M_N},1\otimes Y_M)} \\
		&\quad\quad\quad\quad\quad\quad\quad\quad\quad\quad\times \exp\left( C_{2n,d} \max\left( \left(\frac{k^3}{N^2}\right)\right), \left(\frac{k^3}{N^2}\right)^{1/13} \right).
	\end{align*}
	Hence since $\ln(M_N) = o(N^{2/3})$, one can pick $k_N$ such that 
	$$ \limsup_{N\to \infty}\ \E\left[ \norm{P(X^N\otimes I_{M_N}, I_N\otimes Y^{M_N})} \right] \leq \limsup_{N\to \infty} \norm{P(x\otimes I_{M_N},1\otimes Y_M)}. $$
	However, thanks to Lemma 4.3 of \cite{un} and Corollary 17.10 from \cite{exact}, 
	$$ \limsup_{N\to \infty} \norm{P(x\otimes I_{M_N},1\otimes Y_M)} = \norm{P(x\otimes 1, 1\otimes y)},$$
	hence the conclusion. As for the case of Haar unitary matrices, one can deduce the proof from the case of GUE matrices by following step by step the strategy of Collins and Male detailed in Section 3 of \cite{male}.

\subsection{Proof of Theorem \ref{3lessopti}}

	To begin with, let us note that thanks to Taylor's theorem with integral remainder, one has that for all $x\in \R$,
	\begin{align*}
		e^{\i y x} &= \sum_{k=0}^{n-1} \frac{(\i y x)^k}{k!} + \frac{(\i x y)^{n}}{(n-1)!} \int_0^1 e^{\i y t x} (1-t)^{n-1} dt
	\end{align*}
	Consequently, for any polynomial $P$,
	\begin{align*}
		\norm{e^{\i y P(X^N)} - \sum_{k=0}^{n-1} \frac{(\i y P(X^N))^k}{k!}} \leq \frac{\left(\norm{P(X^N)} |y|\right)^{n}}{n!}
	\end{align*}
	Thus, thanks to well-known concentration inequalities, such as Proposition \ref{3bornenorme}, there exist constants $c,D$ such that for $n\leq cN$, 
	\begin{align*}
		\norm{\E\left[\id_{\A}\otimes\ts_N\left(e^{\i y P(X^N)}\right)\right] - \sum_{k=0}^{n-1} \frac{(\i y)^k}{k!} \E\left[\id_{\A}\otimes\ts_N\left(P^k(X^N)\right)\right] } \leq \frac{\left(D|y|\right)^{n}}{n!}
	\end{align*}
	We set $\alpha_0(y)= (\id_{\A}\otimes\tau)\left(P(x)^k\right)$ and
	\begin{align*}
		&\alpha_p(y) = \frac{1}{2^{p}} \sum_{k=0}^{\infty} \frac{(\i y)^k}{k!} \int_{A_p} e^{-t_1-\cdots-t_{2p}} \sum_{l=1}^{c_{T_p}} \sum_{K\in \mathcal{K}_{4p+1}}\ \sum_{\substack{1\leq z_1,\dots,z_{m_K-1} \leq d \\ \forall E\in \pi_l, \forall n,m\in E, z_n=z_m}} \\
		&\quad\quad\quad\quad\quad\quad\quad\quad\quad\quad\quad\quad\quad\quad\quad\quad \id_{\A}\otimes\tau^{\otimes l_K} \left(v^{l,K,T_p}_z\Big( \left(\partial_{z_1}\circ\dots\circ \partial_{z_{m_K-1}}\left(P^k\right)\right) \Big) \right)dt_1\dots dt_{2p},
	\end{align*}
	Note that thanks to Equation \eqref{xdrfdyuijhjghj}, as well as the fact that thanks to Theorem \ref{ssxdcffvg}, $m_K-1\in [4p,16p]$, there exist constant $\widehat{C},C$ such that for $16p\leq n$,
	\begin{align*}
		&\Bigg| \alpha_p(y) - \frac{1}{2^{p}} \sum_{k< n} \frac{(\i y)^k}{k!} \int_{A_p} e^{-t_1-\cdots-t_{2p}} \sum_{l=1}^{c_{T_p}} \sum_{K\in \mathcal{K}_{4p+1}}\ \sum_{\substack{1\leq z_1,\dots,z_{m_K-1} \leq d \\ \forall E\in \pi_l, \forall n,m\in E, z_n=z_m}} \\ 
		&\quad\quad\quad\quad\quad\quad\quad\quad\quad\quad\quad\quad\quad\quad\quad\quad \id_{\A}\otimes\tau^{\otimes l_K} \left(v^{l,K,T_p}_z\Big( \left(\partial_{z_1}\circ\dots\circ \partial_{z_{m_K-1}}\left(P^k\right)\right) \Big) \right)dt_1\dots dt_{2p} \Bigg| \\
		&\leq \frac{\widehat{C}^p}{p!} \sum_{k\geq n} \frac{|y|^k \norm{P(x)}^k k^{16p}}{k!} \\
		&\leq \frac{C^p}{p!} \sum_{k\geq n} \frac{|y|^k \norm{P(x)}^k }{(k-16p)!} \\
		&\leq \frac{C^p}{p!} \frac{|y|^n \norm{P(x)}^n }{(n-16p)!} e^{|y|\norm{P(x)}}.
	\end{align*}	
	Thus thanks to Proposition \ref{lvleavalvlkd}, one can find a constant $C$ such that if $L< n/16$, and $n^8<N$,
	\begin{align*}
		&\norm{\E\left[\id_{\A}\otimes\ts_N\left(e^{\i y P(X^N)}\right)\right] - \sum_{p= 0}^L \frac{\alpha_p(y)}{N^{2p}} } \\
		&\leq \frac{\left(D|y|\right)^{n}}{n!} + e^C \frac{(|y|\norm{P(x)})^n }{(n-16L)!} e^{|y|\norm{P(x)}} +\sum_{p>L} \sum_{k=0}^{n-1} \frac{1}{N^{2p}p!} C^p \frac{|y|^k}{k!}\times k^{16p} \norm{P(x)}^k \\
		&\leq \frac{\left(D|y|\right)^{n}}{n!} + e^C \frac{(|y|\norm{P(x)})^n }{(n-16L)!} e^{|y|\norm{P(x)}} +\sum_{p>L} \frac{1}{N^{2p}p!} C^p n^{16p} \sum_{k=0}^{n-1} \frac{(|y|\norm{P(x)})^k}{k!} \\
		&\leq \frac{\left(D|y|\right)^{n}}{n!} + \left( \frac{(|y|\norm{P(x)})^n }{(n-16L)!}  + \frac{1}{L!} \left(\frac{n^{16}C}{N^2}\right)^L \right) e^{|y|\norm{P(x)} +C} \\
		&\leq \frac{\left(D|y|\right)^{n}}{n!} + \left( \frac{(|y|\norm{P(x)})^n }{(n-16L)!}  + \frac{C^L}{L!} \right) e^{|y|\norm{P(x)} +C} \\
	\end{align*}
	Thus by picking $n=\lfloor N^{1/8} \rfloor$, $L=\lfloor \alpha n/16\rfloor$, then for any $|y|\leq N^{1/8-\varepsilon}$, one can pick $\alpha$ sufficiently small such that, for any $k\in\N$,
	\begin{align*}
		\norm{\E\left[\id_{\A}\otimes\ts_N\left(e^{\i y P(X^N)}\right)\right] - \sum_{p= 0}^L \frac{\alpha_p(y)}{N^{2p}} } = \mathcal{O}\left(\frac{1}{N^{2(k+1)}}\right).
	\end{align*}
	Now let us remark that
	$$ \int_0^1 \alpha^n(1-\alpha)^m d\alpha = \frac{n!m!}{(n+m+1)!}.$$
	Consequently, with $k_0=0$,
	\begin{align*}
		&\frac{(k_1-1)!(k_2-k_1-1)!\dots (k_{x} - k_{x-1}-1)!(k - k_x)! }{k!} \\
		&= \prod_{i=1}^x \frac{(k_i-k_{i-1}-1)! (k-k_i)!}{(k-k_{i-1})!} \\
		&= \prod_{i=1}^x \int_0^1 (1-\alpha)^{k_i-k_{i-1}-1} \alpha^{k-k_i} d\alpha \\
		&= \int_{[0,1]^x} \prod_{i=1}^x (1-\alpha_i)^{k_i-k_{i-1}-1}\alpha_i^{k-k_i} d\alpha_1\dots d\alpha_x \\
		&= \int_{[0,1]^x} \left(\prod_{i=1}^x\alpha_i^{x-i}\right) \prod_{i=1}^x \left((1-\alpha_i)\prod_{j=1}^{i-1}\alpha_j\right)^{k_i-k_{i-1}-1} (\alpha_1\dots\alpha_x)^{k-k_x} d\alpha_1\dots d\alpha_x.
	\end{align*}
	Hence one has that
	\begin{align*}
		&\sum_{k=0}^{\infty} \frac{(\i y)^k}{k!} \sum_{1\leq k_1<\dots< k_x\leq k} \id_{\A}\otimes\tau^{\otimes l_K} \Biggl(v^{l,K,T_p}_z\Big( P^{k_1-1} \left(\partial_{z_1}\circ\dots\circ \partial_{z_{r_1}}P \right) P^{k_2-k_1-1}\dots\nonumber \\ 
		&\quad\quad\quad\quad\quad\quad\quad\quad\quad\quad\quad\quad\quad\quad\quad\quad\quad\quad\quad \dots P^{k_{x} - k_{x-1}-1} \left(\partial_{z_{r_1+\dots+r_{x-1}}}\circ \dots\circ\partial_{z_{m_K-1}}P \right) P^{k - k_{x}} \Big) \Biggl) \\
		&= \int_{[0,1]^x} (\i y)^x \prod_{i=1}^x\alpha_i^{x-i} \sum_{1\leq k_1<\dots< k_x <\infty} \id_{\A}\otimes\tau^{\otimes l_K} \Biggl(v^{l,K,T_p}_z\Bigg( \frac{(\i (1-\alpha_1) yP)^{k_1-1}}{(k_1-1)!} \\
		&\quad \left(\partial_{z_1}\circ\dots\circ \partial_{z_{r_1}}P \right) \frac{(\i (1-\alpha_2)\alpha_1yP)^{k_2-k_1-1}}{(k_2-k_1-1)!}\dots \frac{(\i (1-\alpha_x)\alpha_1\dots\alpha_{x-1} yP)^{k_{x} - k_{x-1}-1}}{(k_{x} - k_{x-1}-1)!} \\
		&\quad \left(\partial_{z_{r_1+\dots+r_{x-1}}}\circ \dots\circ\partial_{z_{m_K-1}}P \right) \frac{(\i \alpha_1\dots\alpha_x yP)^{k - k_{x}}}{(k - k_{x})!} \Bigg) \Biggl) d\alpha_1\dots d\alpha_x \\
		&= \int_{[0,1]^x} (\i y)^x \prod_{i=0}^{x-1}\alpha_{x-i}^{i} \id_{\A}\otimes\tau^{\otimes l_K} \Biggl(v^{l,K,T_p}_z\Bigg( e^{\i (1-\alpha_1) yP} \left(\partial_{z_1}\circ\dots\circ \partial_{z_{r_1}}P \right) e^{\i (1-\alpha_2)\alpha_1yP}\dots \\	
		&\quad \dots e^{\i (1-\alpha_x)\alpha_1\dots\alpha_{x-1} yP} \left(\partial_{z_{r_1+\dots+r_{x-1}}}\circ \dots\circ\partial_{z_{m_K-1}}P \right) e^{\i \alpha_1\dots\alpha_x yP} \Bigg) \Biggl) d\alpha_1\dots d\alpha_x.
	\end{align*}
	Hence thanks to Theorem \ref{ssxdcffvg}, there exists a constant $C$ such that for all $p\in\N$,
	\begin{align*}
		&\norm{\alpha_p^P(y)} \leq \frac{C^p}{p!} (4p)!\sum_{l=4p}^{16p-3}\sum_{x=1}^{l} \sum_{r_1+\dots+r_x = l} \frac{|y|^x}{x!}.
	\end{align*}
	Besides, the cardinal of the set $\{r_1,\dots,r_x \geq 1\ |\ r_1+\dots+r_x = l\}$ is bounded by $l^x/x!$, hence 
	\begin{align*}
		&\norm{\alpha_p^P(y)} \leq \frac{C^p}{p!} (4p)!\times 12 p e^{16 p} \max_{x\in [1,16p-3]} \frac{|y|^x}{x!}.
	\end{align*}
	Thus one can find a constant $C$ such that
	\begin{equation}
		\label{sdcosdc}
		\norm{\alpha_p^P(y)} \leq C^p (p+|y|)^{3p} \times (1+|y|)^{13p-3}.
	\end{equation}
	Consequently, since we have $n=\lfloor N^{1/8} \rfloor$, $L=\lfloor \alpha n/16\rfloor$, then for any $|y|\leq N^{1/8-\varepsilon}$, there is a constant $D$ such that
	$$  \sum_{p= g}^L \frac{\norm{\alpha_p^P(y)}}{N^{2p}} \leq \sum_{p=g}^{L} C^p\frac{L^{3p} (1+|y|)^{13p}}{N^{2p}} \leq D \left(\frac{L^3 (1+|y|)^{13p}}{N^2}\right)^g$$
	Thus, for any $\varepsilon>0$, $|y|\leq N^{1/8-\varepsilon}$, for all $k\in\N$, one can find a given $g\in\N$ such that
	\begin{align*}
		\norm{\E\left[\id_{\A}\otimes\ts_N\left(e^{\i y P(X^N)}\right)\right] - \sum_{p= 0}^g \frac{\alpha_p^P(y)}{N^{2p}} } = \mathcal{O}\left(\frac{1}{N^{2(k+1)}}\right).
	\end{align*}
	Finally, Equation \eqref{sdcosdc} implies that there exists a constant $C_p$ such that
	$$ \norm{\alpha_p^P(y)} \leq C_p(1+|y|)^{16p-3}. $$
	Let now $f$ be a function $16(k+1)$ times differentiable, and $h\in\CC^{\infty}(\R)$ with a compact support bounded by $K$, then for all $\varepsilon>0$,
	\begin{align*}
		\int_{|y|\leq N^{1/8-\varepsilon}} \left|\widehat{fh}(y)\right| \norm{\alpha_p^P(y)} dy &\leq 2^{16(k+1)} C_p \int_{\R} \frac{\left|\widehat{fh}(y)\right|+\left|\widehat{(fh)^{(16(k+1))}}(y)\right|}{(1+|y|)^2} dy\ (1+N^{1/8})^{16 \max(p-k-1,0)} \\
		&\leq 2^{16(k+1)}  C_p K  \int_{\R} \frac{1}{(1+|y|)^2} dy\ (1+N^{1/8})^{16 \max(p-k-1,0)} \\
		&\quad\quad\quad\quad\quad\quad\quad\quad\quad\quad\quad\quad\quad\quad\quad\times \sup_{t\in\R} \left(|(fh)(t)|+|(fh)^{(16(k+1))}(t)|\right).
	\end{align*}
	Hence for all $p\geq k+1$, 
	$$ \frac{1}{N^{2p}}\int_{|y|\leq N^{1/8-\varepsilon}} \left|\widehat{fh}(y)\right| \norm{\alpha_p^P(y)} dy = \mathcal{O}\left(\frac{1}{N^{2(k+1)}}\right).$$
	Besides, for all $p\leq k$, with $\varepsilon$ sufficiently small,
	\begin{align*}
		\int_{|y|> N^{1/8-\varepsilon}} \left|\widehat{fh}(y)\right| \norm{\alpha_p^P(y)} dy &\leq C_p \int_{|y|> N^{1/8-\varepsilon}} \frac{\left|\widehat{fh}(y)\right| (1+|y|)^{16(k+1)}}{(1+|y|)^2} dy \times (1+N^{1/8-\varepsilon})^{16(p-k-1)-1} \\
		&\leq 2^{16(k+1)}  C_p K \int_{\R} \frac{1}{(1+|y|)^2} dy\\
		&\quad\quad\quad\quad\quad\quad\quad\quad\times \sup_{t\in\R} \left(|(fh)(t)|+|(fh)^{(16(k+1))}(t)|\right) \times N^{2(p-k-1)}.
	\end{align*}
	Hence for all $p\leq k$, 
	$$ \frac{1}{N^{2p}}\int_{|y|> N^{1/8-\varepsilon}} \left|\widehat{fh}(y)\right| \norm{\alpha_p^P(y)} dy = \mathcal{O}\left(\frac{1}{N^{2(k+1)}}\right).$$
	Consequently, with 
	$$ \alpha^f_p(P) = \int_{\R} \widehat{f}(y) \alpha_p^P(y)\ dy, $$
	one has that
	\begin{align*}
		\norm{\E\left[\id_{\A}\otimes\ts_N\left((fh)(P(X^N)) \right)\right] - \sum_{p= 0}^k \frac{\alpha_p^{fh}(P)}{N^{2p}} } = \mathcal{O}\left(\frac{1}{N^{2(k+1)}}\right).
	\end{align*}
	Let us now remark that if $P$ is of degree $s$, then 
	$$ \norm{P(X^N)} \leq \sup_{M} \norm{a_M} \times \sup_{M\text{ monomial, }\deg M\leq s} M(\norm{X_1^N},\dots, \norm{X_d^N}). $$
	Thus thanks to classical concentration inequalities, such as Proposition \ref{3bornenorme}, there exist constants $D,u$ such that
	$$ \P\left( \norm{P(X^N)} \geq K+ \varepsilon \right) \leq e^{- u\varepsilon^{1/s} N}. $$
	Hence let $h\in\CC^{\infty}(\R)$ be a compact support function which takes the value $1$ on $[-K-1,K+1]$, since we assumed that $f$ was bounded, for some constant $u$,
	$$ \E\left[\id_{\A}\otimes\ts_N\left((fh)(P(X^N)) \right)\right] - \E\left[\id_{\A}\otimes\ts_N\left(f(P(X^N)) \right)\right] = \mathcal{O}(e^{-u N}). $$
	Thus we get that for any function $h\in\CC^{\infty}(\R)$ with a compact support and which takes the value $1$ on $[-K-1,K+1]$,
	$$ \E\left[\id_{\A}\otimes\ts_N\left(f(P(X^N)) \right)\right] = \sum_{p= 0}^k \frac{\alpha_p^{fh}(P)}{N^{2p}} +	 \mathcal{O}\left(\frac{1}{N^{2(k+1)}}\right).$$
	in particular, the coefficients $\alpha_p^{fh}(P)$ do not depend on $h$, hence one simply set $\alpha_p^{f}(P) := \alpha_p^{fh}(P)$ for some function $h$ which satisfies the assumptions above. Let $(\beta_p)_{p\geq 0}$ be coefficients such that
	$$ \E\left[\id_{\A}\otimes\ts_N\left(f(P(X^N)) \right)\right] = \sum_{p= 0}^k \frac{\beta_p}{N^{2p}} +	 \mathcal{O}\left(\frac{1}{N^{2(k+1)}}\right).$$
	Then let $p' = \min \{p\geq 0| \alpha_p^{f}(P) \neq \beta_p\}$, necessarily by definition if $p'\leq k$,
	$$ \beta_{p'} - \alpha_{p'}^{f}(P) = \mathcal{O}(N^{-2}).$$
	Hence the coefficients $\alpha_p^{f}(P)$ are unique.

\subsection{Proof of Theorem \ref{masterineq}}

	Let $Q=P_1\otimes\cdots\otimes P_n\in\A\langle X_1,\dots,X_d\rangle^{\otimes n}$, let us study the quantity $(\tau_{\A}\otimes\tau)\left( \big(m_{\sigma}(Q(x))\big)^k \right)$. We view $\sigma$ as a permutation of $[1,kn]$ by defining $\sigma(ln+i) = ln+\sigma(i)$, then thanks to Lemma \ref{sdtyuikmn},
	\begin{align*}
		&(\tau_{\A}\otimes\tau)\left( m_{\sigma}\left(Q(x)\right)^k \right) \\
		&= \sum_{K\in \mathcal{K}_{nk}}\ \sum_{z\in [1,d]^{\widetilde{K}}} \tau_{\A}\left( m\otimes H^K_z\left( \bigotimes_{\forall i\in [1,kn]}\ \left(\bigotimes_{j\in C_i^{**}(K)} \partial_{i_{\{i,j\}}}\right) \right)\left(\sigma\left(Q(x)^{\otimes k}\right)\right) \right),
	\end{align*}
	where 
	$$m(a_1\otimes\dots\otimes a_{kn}) = a_1\dots a_{kn}. $$
	$$ \sigma((a_1\otimes A_1)\otimes\dots\otimes (a_{kn}\otimes A_{kn})) = (a_1\otimes A_{\sigma(1)})\otimes\dots\otimes (a_{kn}\otimes A_{\sigma(kn)}). $$
	Consequently, with Lemma \ref{nkytrdds},
	\begin{align*}
		&(\tau_{\A}\otimes\tau)\left( m_{\sigma}\left(Q(x)\right)^k \right) \\
		&= \sum_{K\in \mathcal{K}_{nk}}\ \sum_{z\in [1,d]^{\widetilde{K}}} \tau_{\A}\otimes\tau^{\otimes l_K} \left( m\otimes u^{\sigma,K}_z\left( \bigotimes_{p\in [1,kn]}\ \left(\bigotimes_{q\in \sigma\left(C_{\sigma^{-1}(p)}^{**}(K)\right)} \partial_{z_{\sigma^{-1}(p),\sigma^{-1}(q)}}\right) \right)\left(Q(x)^{\otimes k}\right) \right) \\
		&= \sum_{K\in \mathcal{K}_{nk}}\ \sum_{z\in [1,d]^{\widetilde{K}}} \tau_{\A}\otimes\tau^{\otimes l_K} \left( \prod_{p\in [1,kn]} u^{\sigma,K}_{p}\left( \ \left(\bigotimes_{q\in \sigma\left(C_{\sigma^{-1}(p)}^{**}(K)\right)} \partial_{z_{\sigma^{-1}(p),\sigma^{-1}(q)}}\right)\left(P_{p}\right) \right) \Delta^{p,\sigma,K}_z \right),
	\end{align*}
	where $\Delta^{p,\sigma,K}_z \in \A\otimes \CC^*(x)^{\otimes l_K}$ is an operator of norm smaller than one, and 
	$$u^{\sigma,K}_{p}:\A\otimes\CC^*(x)^{\otimes \# C_{\sigma^{-1}(p)}^{**}(K)} \to \A\otimes\CC^*(x)^{\otimes l_K}$$
	is such that 
	$$ u^{\sigma,K}_{p}(a\otimes A_1\otimes\cdots\otimes A_r) = a\otimes \id_{\CC^*(x)}^{\otimes n_1} \otimes A_1(x) \otimes \id_{\CC^*(x)}^{\otimes n_2}\otimes\cdots\otimes \id_{\CC^*(x)}^{\otimes n_r} \otimes A_r(x) \otimes \id_{\CC^*(x)}^{\otimes n_{r+1}},$$
	for given $n_1,\dots,n_{r+1}\in\N$ which depends on $\sigma,K,p$. Besides, one has that,
	$$ \sum_{i=1}^{nk} \# C_{\sigma^{-1}(i)}^{**}(K) = \sum_{i=1}^{nk} \# C_i^*(K)-1 = \widehat{K} + c_K -nk \leq 3nk,$$
	where in the last line we used Lemma \ref{ssxdcffvg} which insures that $\widehat{K} + c_K$ is smaller than $4nk$. Indeed, if $K\in \mathcal{K}_{nk}$, then with $m=\#\widehat{K}$, one can view $K$ as an element of $\K_{nk-m}$, and hence
	$$ \widehat{K} + c_K \leq m + 4(nk-m) -6 \leq 4nk.$$
	Thus, one has that
	\begin{align}
		\label{xrfgyuikkoi}
		&\left|(\tau_{\A}\otimes\tau)\left( m_{\sigma}\left(Q(x)\right)^k \right)\right| \\
		&\leq \#\mathcal{K}_{nk}\times d^{3kn}\times \sup_{K\in\K_{nk},z_{p,q}\in[1,d]} \prod_{p\in [1,kn]} \norm{\left(\bigotimes_{q\in C_{\sigma^{-1}(p)}^{**}(K)} \partial_{i_{p,q}}\right) \left(P_{p}\right)(x)} \nonumber \\
		&\leq \#\mathcal{K}_{nk}\times d^{3kn}\times \sup_{\substack{(d_{i,j})_{i\geq 0, j\in [1,s]} \\ \sum_{i,j\geq 0} i d_{i,j}\leq 3nk \\ \sum_i d_{i,j} = k}} \prod_{j\in [1,n]} \prod_{i\geq 0} \left( \sup_{p_1,\dots,p_i\in [1,d]} \norm{\partial_{p_1}\otimes\cdots\otimes \partial_{p_i}(P_j)(x) } \right)^{d_{i,j}}, \nonumber
	\end{align}
	Besides, one can find $\widehat{\sigma}\in\S_{2n}$ such that
	\begin{align*}
		(\tau_{\A}\otimes\tau)\left( \big(m_{\sigma}\left(Q(x)\right)^* m_{\sigma}\left(Q(x)\right)\big)^k \right) &= (\tau_{\A}\otimes\tau)\left( \big(m_{\widehat{\sigma}}(Q^*(x)\otimes Q(x))\big)^k \right),
	\end{align*}
	where 
	$$ \left(\sum_i (a_1\otimes A_1)\otimes\dots\otimes (a_n\otimes A_n)\right)^* = \sum_i (a_n^*\otimes A_n^*)\otimes\dots\otimes (a_1^*\otimes A_1^*). $$
	Consequently, thanks to Equation \eqref{xrfgyuikkoi} and Lemma \ref{ssxdcffvg}, with $P_{s+j}=P_j^*$,
	\begin{align*}
		&(\tau_{\A}\otimes\tau)\left( \big(m_{\widehat{\sigma}}(Q^*(x)\otimes Q(x))\big)^k \right)^{1/2k} \\
		&\leq (\#\mathcal{K}_{2nk})^{\frac{1}{2k}}\times d^{3n}\times \sup_{\substack{(d_{i,j})_{i\geq 0, j\in [1,2n]} \\ \sum_{i,j\geq 0} i d_{i,j}\leq 6nk \\ \forall j,\ \sum_i d_{i,j} = k}} \prod_{j\in [1,n]} \prod_{i\geq 0} \left( \sup_{p_1,\dots,p_i\in [1,d]} \norm{\partial_{p_1}\otimes\cdots\otimes \partial_{p_i}(P_j)(x) } \right)^{d_{i,j}/(2k)} \\
		&\leq (80e)^{n}\times d^{3n}\times \sup_{\substack{(d_{i,j})_{i\geq 0, j\in [1,2n]} \\ \sum_{i,j\geq 0} i d_{i,j}\leq 6nk \\ \forall j,\ \sum_i d_{i,j} = k}} \prod_{j\in [1,n]} \prod_{i\geq 0} \left( \sup_{p_1,\dots,p_i\in [1,d]} \norm{\partial_{p_1}\otimes\cdots\otimes \partial_{p_i}(P_j)(x) } \right)^{d_{i,j}/(2k)}
	\end{align*}
	However, if $P=\sum_M a_M M$, 
	\begin{align*}
		&\norm{\partial_{p_1}\otimes\cdots\otimes \partial_{p_i}(P^*)(x) } \\
		&= \norm{\sum_M \sum_{M^*=A_1X_{p_1}A_2\dots X_{p_i}A_{i+1}} a_M^*\otimes A_1(x)\otimes\cdots\otimes A_{i+1}(x) } \\
		&= \norm{\sum_M \sum_{M= A_{i+1}^*X_{p_i}\dots A_2^* X_{p_1} A_1^*} a_M^*\otimes A_{i+1}(x)\otimes\cdots\otimes A_1(x) } \\
		&= \norm{\left(\partial_{p_1}\otimes\cdots\otimes \partial_{p_i}(P)(x)\right)^* } \\
		&=\norm{\partial_{p_1}\otimes\cdots\otimes \partial_{p_i}(P)(x) }
	\end{align*}
	Consequently,
	\begin{align*}
		&(\tau_{\A}\otimes\tau)\left( \big(m_{\widehat{\sigma}}(Q^*(x)\otimes Q(x))\big)^k \right)^{1/2k} \\
		&\leq (80e)^{n}\times d^{3n}\times \sup_{\substack{(d_{i,j})_{i\geq 0, j\in [1,n]} \\ \sum_{i,j\geq 0} i d_{i,j}\leq 6nk \\ \sum_i d_{i,j} = 2k}} \prod_{j\in [1,n]} \prod_{i\geq 0} \left( \sup_{p_1,\dots,p_i\in [1,d]} \norm{\partial_{p_1}\otimes\cdots\otimes \partial_{p_i}(P_j)(x) } \right)^{d_{i,j}/(2k)} \\
		&\leq (80e)^{n}\times d^{3n}\times \sup_{\substack{(\alpha_{i,j})_{i\geq 0, j\in [1,n]}\in\R^+ \\ \sum_{i,j\geq 0} i \alpha_{i,j}\leq 3n \\ \forall j,\ \sum_i \alpha_{i,j} = 1}} \prod_{j\in [1,n]} \prod_{i\geq 0} \left( \sup_{p_1,\dots,p_i\in [1,d]} \norm{\partial_{p_1}\otimes\cdots\otimes \partial_{p_i}(P_j)(x) } \right)^{\alpha_{i,j}}
	\end{align*}
	Hence the conclusion since we assumed that $\tau_{\A}$ was faithful, and since $\tau$ is by construction also faithful, thanks to Lemma 4.1.8 from \cite{ozabr}, $\tau_{\A}\otimes\tau$ is also faithful, and thus 
	
	\begin{align*}
		\norm{m_{\sigma}\left(Q(x)\right)} &= \lim_{k\to \infty}\ (\tau_{\A}\otimes\tau)\left( \big(m_{\sigma}\left(Q(x)\right)^* m_{\sigma}\left(Q(x)\right)\big)^k \right)^{\frac{1}{2k}} \\
		&= \lim_{k\to \infty}\ (\tau_{\A}\otimes\tau)\left( \big(m_{\widehat{\sigma}}(Q^*(x)\otimes Q(x))\big)^k \right)^{\frac{1}{2k}} \\
		&\leq (80e)^{n}\times d^{3n}\times \sup_{\substack{(\alpha_{i,j})_{i\geq 0, j\in [1,n]}\in\R^+ \\ \sum_{i,j\geq 0} i \alpha_{i,j}\leq 3n \\ \forall j,\ \sum_i \alpha_{i,j} = 1}} \prod_{j\in [1,n]} \prod_{i\geq 0} \left( \sup_{p_1,\dots,p_i\in [1,d]} \norm{\partial_{p_1}\otimes\cdots\otimes \partial_{p_i}(P_j)(x) } \right)^{\alpha_{i,j}}.
	\end{align*}

\bibliographystyle{abbrv}

\end{document}